\documentclass[11pt]{amsart}
\usepackage{fullpage}
\usepackage{amssymb,amsmath,amsthm,amscd,mathrsfs,graphicx, color}
\usepackage[cmtip,all,matrix,arrow,tips,curve,color]{xy}
\usepackage[dvipsnames]{xcolor}
\usepackage{tikz}
\usepackage{pdfpages}

\numberwithin{equation}{section}
\newtheorem{teo}{Theorem}[section]
\newtheorem{pro}[teo]{Proposition}
\newtheorem{lem}[teo]{Lemma}
\newtheorem{cor}[teo]{Corollary}

\newtheorem{fact}[teo]{Fact}

\theoremstyle{definition}
\newtheorem{dfn}[teo]{Definition}
\newtheorem{exa}[teo]{Example}

\newtheorem{construction}[teo]{Construction}

\theoremstyle{remark}
\newtheorem{rem}[teo]{Remark}

\newcommand{\calA}{\mathcal{A}}
\newcommand{\calB}{\mathcal{B}}
\newcommand{\calD}{\mathcal{D}}

\newcommand{\calH}{\mathcal{H}}

\newcommand{\calM}{\mathcal{M}}

\newcommand{\calJ}{\mathcal{J}}
\newcommand{\calR}{\mathcal{R}}
\newcommand{\calX}{\mathcal{X}}
\newcommand{\calS}{\mathcal{S}}

\newcommand{\Sp}{\operatorname{Sp}}

\newcommand{\Pic}{\operatorname{Pic}}

\newcommand{\Sing}{\operatorname{Sing}}
\newcommand{\Prym}{\operatorname{Prym}}

\newcommand{\ZZ}{\mathbb{Z}}
\newcommand{\QQ}{\mathbb{Q}}
\newcommand{\RR}{\mathbb{R}}
\newcommand{\CC}{\mathbb{C}}

\newcommand{\PP}{\mathbb{P}}

\newcommand{\ab}[1][g]{\calA_{#1}}
\newcommand{\Vor}[1][g]{{\overline{\calA}_{#1}^{V}}}
\newcommand{\Perf}[1][g]{{\overline{\calA}_{#1}^{P}}}
\newcommand{\Centr}[1][g]{{\overline{\calA}_{#1}^{C}}}
\newcommand{\Matr}[1][g]{{{\calA}_{#1}^{\operatorname{matr}}}}
\newcommand{\Sat}[1][g]{{{\calA}_{#1}^{*}}}

\newcommand{\AV}[1][5]{\overline{\calA}_{#1}^V}
\newcommand{\AP}[1][5]{\overline{\calA}_{#1}^P}

\newcommand{\tM}{\widetilde{\calM}}
\newcommand{\hM}{\widehat{\calM}}
\newcommand{\bM}{\overline{\calM}}
\newcommand{\BG}{(\calB/\Gamma)^*}
\newcommand{\tDA}[1][1]{\widetilde{D}_{A_{#1}}}

\newcommand{\tDD}[1][4]{\widetilde{D}_{D_{#1}}}

\newcommand{\tDh}{\widetilde{D}_{h}}

\newcommand{\hDh}{\widehat{D}_{h}}

\newcommand{\cDh}{{\calD}_{h}}
\newcommand{\cDn}{{\calD}_{n}}

\newcommand*\circled[1]{\tikz[baseline=(char.base)]{
  \node[shape=circle,draw,inner sep=1pt] (char) {#1};}}
\newcommand{\stab}{{\operatorname{stab}}}
\newcommand{\real}{{\operatorname{real}}}

\begin{document}


\title{Complete moduli of cubic threefolds and their intermediate Jacobians}

\author[S. Casalaina-Martin]{Sebastian Casalaina-Martin}
\address{University of Colorado, Department of Mathematics, Boulder, CO 80309}
\email{casa@math.colorado.edu}

\author[S. Grushevsky]{Samuel Grushevsky}
\address{Stony Brook University, Department of Mathematics, Stony Brook, NY 11794-3651}
\email{sam@math.stonybrook.edu}

\author[K. Hulek]{Klaus Hulek}
\address{Institut f\"ur Algebraische Geometrie, Leibniz Universit\"at Hannover,  30060 Hannover, Germany}
\email{hulek@math.uni-hannover.de}

\author[R. Laza]{Radu Laza}
\address{Stony Brook University, Department of Mathematics, Stony Brook, NY 11794-3651}
\email{rlaza@math.stonybrook.edu}

\thanks{Research of the first author was supported in part  by NSF grant DMS-11-01333 and a Simons Foundation
Collaboration Grant for Mathematicians
(317572). Research of the second author was supported in part by NSF grants DMS-12-01369 and DMS-15-01265, and a Simons Fellowship in mathematics. Research of the third author was supported in part by DFG grant Hu-337/6-2. Research of the fourth author was supported in part by NSF grants DMS-12-00875 and DMS-12-54812.}

\begin{abstract}
The intermediate Jacobian map, which associates to a smooth cubic threefold its intermediate Jacobian,  does not extend to the GIT compactification of the space of cubic threefolds, not even as a map to the Satake compactification of the moduli space of principally polarized abelian fivefolds. A better ``wonderful'' compactification $\tM$ of the space of cubic threefolds was constructed by the first and fourth authors --- it has a modular interpretation, and divisorial normal crossing boundary. We prove that the intermediate Jacobian map extends to a morphism from $\tM$ to the second Voronoi toroidal compactification of $\ab[5]$ --- the first and fourth author previously showed that it extends to the Satake compactification. Since the second Voronoi compactification has a modular interpretation, our extended intermediate Jacobian map encodes all of the geometric information about  the degenerations of intermediate Jacobians, and allows for the  study of the  geometry of cubic threefolds via degeneration techniques. As one application we  give a complete classification of all degenerations of intermediate Jacobians of cubic threefolds of torus rank 1 and 2.
\end{abstract}


\makeatletter
\let\@wraptoccontribs\wraptoccontribs
\makeatother

\date{\today}

\maketitle


Cubic threefolds are rich geometric objects that have played a tremendous role in algebraic geometry,  providing an important testing ground for a wide range of
questions in Hodge theory, birational geometry, and the theory of algebraic cycles.
A famous example of this is the proof of Clemens and Griffiths
that a smooth cubic threefold $X \subset \PP^4$ is unirational, but never rational \cite{cg}.
The moduli space of cubic threefolds $\mathcal M$  has also proven to be a remarkably interesting object.
In \cite{a}  and \cite{yokoyama} respectively,  Allcock   and Yokoyama gave an explicit description of the compactification $\overline{\mathcal M}$ obtained via  geometric invariant theory.  With this as a starting point, Allcock--Carlson--Toledo \cite{act} and Looijenga--Swierstra \cite{ls} showed that $\mathcal M$ is the complement of an arithmetic arrangement $\mathcal H$ of hyperplanes in a ten dimensional ball quotient  $\mathcal B/\Gamma$; the Baily--Borel compactification then gives a  compactification $(\mathcal B/\Gamma)^*$.  Moreover, inside of this model  lie  a number of other moduli spaces with arithmetic descriptions, including the ball-quotient model of genus five hyperelliptic curves due to Deligne--Mostow \cite{dm, act}, as well as a model of Kondo's ball quotient for genus $4$ curves \cite{k2,cmjl1}.

The \cite{act} ball quotient model was obtained using the period map for cubic fourfolds \cite{voisin,laza, looijenga4folds}. It is more natural to study cubic threefolds via the period map used by Clemens and Griffiths, associating to a cubic threefold its intermediate Jacobian, which is a principally polarized abelian fivefold. Recall that they prove the Torelli theorem for cubic threefolds:  the map
\begin{equation}
IJ: \mathcal M \to \ab[5]
\end{equation}
from the moduli space  $\mathcal M$ of cubic threefolds to the moduli space $\ab[5]$ of principally polarized abelian varieties of dimension $5$ is an injective immersion (the local and global Torelli Theorems hold and
the automorphism groups  coincide, see  \cite[Thm. 9.8 (b)]{griffiths}, \cite[(0.11)]{cg}, and \cite[Thm. 5.1]{zz}).
It is thus natural to compactify this map, i.e., to study degenerations of intermediate Jacobians. Our main result is:

\begin{teo}[Main theorem]\label{teo:main}
The map $IJ: \calM\to\calA_5$ sending a smooth cubic threefold to
its intermediate Jacobian extends to a morphism $\widetilde{IJ}^{V}:\tM\to\Vor[5]$
from the wonderful blow-up of the GIT compactification of the moduli
space of cubic threefolds to the second Voronoi toroidal compactification
of the moduli space of principally polarized abelian fivefolds.
\end{teo}
The wonderful compactification $\tM$ appeared for the first time in the work \cite{cml} of the first and fourth author; it has a normal crossing divisorial boundary (in a stack sense), and a modular interpretation that we will review below.

The method of our proof of Theorem \ref{teo:main}  is to study the monodromy cones. These turn out to be spanned by rank $1$ quadrics (see Remark \ref{rem_moncone}), and thus the image of $\widetilde {IJ}^V$ is contained in the matroidal locus $\Matr[5]$, which by the results of Melo and Viviani \cite{MV12} is the maximal partial compactification of $\ab[5]$ contained in both the second Voronoi compactification $\Vor$ and the perfect cone compactification
$\Perf$. As a result we obtain:

\begin{cor}\label{cor:matroidal}
The image of the intermediate Jacobian map  $\widetilde{IJ}^{V}:\tM\to\Vor[5]$
 is contained
in the matroidal locus $\Matr[5]$. In particular, the intermediate Jacobian map also extends to a morphism  $\widetilde{IJ}^{P}:\tM\to\Perf[5]$.
\end{cor}
These results are the culmination
 of much of our previous work on the subject. The first and the fourth author introduced and studied the wonderful
compactification in \cite{cml}, the second and third authors investigated the class of the cycle given by intermediate Jacobians in \cite{grhu1} and
together we studied the extension of period maps, in particular of the Prym map in  \cite{cmghl}.  Clearly, all of this is closely connected to the extensive literature of
degenerations of Jacobians and Pryms (see esp. \cite{fs}, \cite{abh}, \cite{ab}, and \cite{g}).

One motivation for our work is to make it possible to study the geometry of intermediate Jacobians of cubic threefolds via degeneration techniques --- recall that by the results of \cite{alexeev02}, \cite{olsson} the Voronoi compactification $\Vor[5]$ has a modular interpretation, see in particular Theorems \ref{theo:codimension1intro} and \ref{theo:divisors}, below.
In fact, after the appearance of the first version of our manuscript, the degenerations of intermediate Jacobians of cubic threefolds were used in \cite{LSV} to give an alternative construction of  O'Grady's exceptional $10$-dimensional compact hyperk\"ahler manifold (see \cite{OG10}). While the two papers share a number of ideas and techniques, their goals are  complementary. Namely, here the focus is on the modularity of the degenerations of intermediate Jacobians, leading us to considering the wonderful blow-up $\widetilde \calM$ of the moduli space. In contrast, in \cite{LSV}, it is essential to compactify the relative intermediate Jacobian family (allowing worse degenerations than those we consider here) arising from hyperplane sections of a fixed cubic fourfold without modifying the base (isomorphic to $\PP^5$) of the fibration.
In a different direction, we point out that while a mildly singular cubic threefold is rational, and thus the standard degeneration techniques do not apply for the rationality question, the degeneration data (e.g. for the Segre cubic) from the perspective of abelian varieties is rich enough to decide the irrationality of the generic cubic threefold (see \cite{g}).

The question of extending morphisms, given by period maps, between moduli spaces is a classical problem in algebraic geometry.
The prime example is the Torelli map $t_g:  \mathcal M_g \to \ab[g]$ from the moduli space of curves.
It is a classical result of Mumford and Namikawa \cite{nam76I, nam76II} that the Torelli map has a natural extension $\overline t^{V}:\overline \calM_g\to  \Vor$ from the Deligne--Mumford compactification $\overline \calM_g$.
Alexeev--Brunyate \cite{ab} showed that the image of the Torelli map lies in the matroidal locus, and thus the Torelli maps also extends to a morphism $\overline t^{P}:\overline\calM_g\to\Perf$. In contrast, the Prym map does not extend to a morphism from  $\overline{\mathcal R}_{g}$, the compactification of the moduli space  of connected \'etale  double covers of curves  \cite{beauville}, to any of the usual toroidal compactifications, as shown by Friedman and Smith \cite{fs}. This result was further refined in \cite{abh}, \cite{vologodsky} and \cite{cmghl} where the indeterminacy locus of the Prym map was studied.

The Prym map is closely related to our situation: Mumford showed  that intermediate Jacobians of cubic threefolds are Pryms of \'etale double covers of plane quintic curves, and thus many of the arguments for Prym varieties still apply in the case of cubic threefolds.
From the results about the Prym map it is anything but obvious that the intermediate Jacobian map extends to a good compactification of $\calM$. It is the main result of our paper that this is indeed the case.

At this point it is necessary to discuss what the correct compactification of the moduli space $\mathcal M$ of smooth cubics is.
The natural starting point  is the GIT compactification $\overline \calM$ (studied in \cite{a}), however, it was already shown in  \cite{cml} that the intermediate Jacobian map
does not extend to a morphism on  $\overline \calM$, not even to the Satake compactification $\Sat[5]$.
Indeed, the general theory of degenerations of Hodge structures (e.g. \cite{cks}) tells us that one should modify this compactification so that the discriminant locus
$\Sigma_{\overline \calM} =\overline \calM \setminus \calM$
becomes a normal crossing divisor.
It is, of course, always possible to consider a log resolution of $(\overline \calM,\Sigma_{\overline \calM})$.  However,  we want this to be done in a controllable and geometrically meaningful way, so that we will have some hope of using this resolution of the period map to describe the image.
The result is the wonderful compactification $\tM$ which made its first appearance in \cite{cml} where it was
shown (as an easy consequence of the Borel extension theorem) that one obtains an extended period map $\tM\to \Sat[5]$.
However, the map to the Satake compactification loses a lot of information: $\Sat[5]$ is not a fine moduli space, and the map there does not help in studying degenerations of intermediate Jacobians in families.  For instance,  note that on the one hand all degenerations with trivial abelian part (i.e., torus rank $5$) will correspond to a single point in $\Sat[5]$, while, on the other hand, it is known from Gwena \cite{g} that the extension data for degenerations to the  Segre cubic are rich enough to distinguish it from degenerations of Jacobians (and thus to prove non-rationality of the general cubic threefold). For this reason, the knowledge of the map to  $\Sat[5]$ is not sufficient from a geometric point of view, while the extension of the intermediate Jacobian map to a morphism to $\Vor[5]$ contains the relevant geometric information for applying degeneration methods.

While a complete description of the boundary strata of $\tM$ is quite intricate, the corresponding combinatorics is fairly easy -- it is governed by the root lattices associated to the singularities. The map $\widetilde {IJ}^V$ must send a codimension $k$ boundary stratum in $\tM$ to a codimension at least $k$ stratum in the boundary of the locus of intermediate Jacobians, which we can then describe explicitly. In the case of codimension $1$ strata the result is the following:

\begin{teo}\label{theo:codimension1intro}
The boundary of the  ten-dimensional  intermediate Jacobian locus $IJ$ in $\Vor[5]$
consists of
the closures of the following  nine-dimensional loci:
\begin{eqnarray*}
\mathcal H_5 &:=& \text{The hyperelliptic locus in } \mathcal A_5, \ \\
{\bf K}&:=&\lbrace  A \mid  A=E\times J(C)\in \mathcal A_1\times (\mathcal J_4\cap \theta_{\operatorname{null}})\rbrace,\\
{\bf A}&:=&\lbrace (A,z_4)\mid A=J(C) \in\calJ_4,z_4\in 2_*\Sing\Theta_C\rbrace,\\
{\bf B}&:=&\lbrace (A,z_4)\mid A=E\times J(C')\in\calA_1\times \calH_3,
  z_4=(z_1,z_3),  z_3\in \Theta_{C'} = C' - C' \rbrace.
\end{eqnarray*}
Here $\calH_g$ denotes the hyperelliptic locus in genus $g$,
 $\calJ_4$ denotes
the Schottky locus in $\ab[4]$ and
$\theta_{\operatorname{null}}$ denotes the locus in $\mathcal A_4$ of abelian varieties with a vanishing theta-null, while
  $J(C)$ is the (degree $0$) Jacobian. In $\mathbf A$ we take  $\Theta_C$ to be any symmetric theta divisor in $J(C)$, in $\mathbf B$ we take
 $\Theta_{C'}$ to be the unique symmetric theta divisor in $J(C')$ whose singularity is at the origin. We use $z_k$ to denote a point in an abelian variety of dimension $k$, giving the extension datum determining the corresponding point in Mumford's partial compactification (torus rank $1$).  The same result holds for the closure of $IJ$ in the perfect cone compactification $\Perf[5]$.
\end{teo}

The boundary divisors corresponding to  $\mathcal H_5$ and $\mathbf A$ were known to Collino \cite{col}  and Clemens--Griffiths \cite{cg}.  The boundary divisor corresponding to  $\mathbf K$ can be found in \cite{cml, grsmordertwo}
and the boundary divisor corresponding to $\mathbf B$ was discovered  in \cite{grhu1}.
The key addition in the theorem above is that, using the explicit resolution $ \widetilde {\mathcal M}\to \Vor[5]$, we are now able to  show that there are no other boundary divisors in $\overline{IJ}$.
In fact, our proof  provides a complete answer to how the boundary divisors of $\tM$ are related to degenerations of intermediate Jacobians. We recall that the boundary divisors of $\tM$ are labeled $\tDA[1],\tDA[2],\tDA[3],\tDA[4],\tDA[5],\tDD[4],\tDh$, corresponding to the blow-ups in the GIT compactification of the loci of cubics with an isolated singularity of the type indicated by the subscript, and,  to the blow-up of the point corresponding to the chordal cubic (giving the divisor $\tDh$).
The result is then:

\begin{teo}\label{theo:divisors} Under the map $\widetilde{IJ}^{V}:\tM\to\Vor[5]$
all but four boundary  divisors of the wonderful blow-up are contracted. The boundary divisors which are not contracted and their images are:

\begin{itemize}

\item[(1)] The chordal divisor $\tDh$ is mapped to the closure $\overline \calH_5$ of the hyperelliptic locus.

\item[(2)] The divisor $\tDA[2]$ is mapped to the closure $\overline {\mathbf K}$ of the locus $\mathbf K$.

\item[(3)] The divisor $\tDA[1]$ is mapped to the closure $\overline {\mathbf A}$ of the locus $\mathbf A$.

\item[(4)]  The divisor $\tDA[3]$ is mapped to the closure $\overline {\mathbf B}$ of the locus $\mathbf B$.

\end{itemize}
The same statement holds for the map $\widetilde{IJ}^{P}:\tM\to\Perf[5]$ to the perfect cone compactification.
\end{teo}

The fact that the chordal divisor $\tDh$ is mapped to the closure $\overline \calH_5$ of the hyperelliptic locus is due to Collino \cite{col}.
The fact that $\overline{\mathbf A}$ is the image of  $\tDA[1]$ is well-known (see e.g., \cite{cg,cm}).
The fact that  $\overline{\mathbf K}$ is the image of  $\tDA[2]$ was shown in \cite{cml}.

Our techniques also yield a geometric description of deeper boundary strata; we easily see that the divisor $\tDA[4]$ is contracted to a locus whose generic point lies in $\mathcal A_5$,  the divisor $\tDA[5]$ is contracted to a locus whose generic point is contained in the torus rank $1$ boundary, and the divisor $\tDD[4]$ is contracted to a locus whose generic point is contained in the torus rank $2$ boundary. Theorem \ref{theo:degenerationscod2} gives a complete description of the boundary of the locus of intermediate Jacobians of torus rank $2$; it uses both the Prym theoretic and the theta function approach. In principle any stratum can be studied geometrically by our methods, and various further  cases have been considered in  \cite{havasi}.

\subsection*{Structure of the paper}
In Section \ref{sec:tM} we start by recalling the won\-der\-ful compactification $\tM$ of the moduli space $\calM$ of cubic threefolds, which was
first constructed in \cite{cml}.
For this we also use the  Allcock--Carlson--Toledo ball quotient model $\BG$, where $\calB$ is the $10$-dimensional ball,
$\Gamma$ a suitable arithmetic group acting on $\calB$ and $\BG$ the Baily--Borel compactification
of the moduli space of cubic threefolds (see \cite{act, ls}). This has the
advantage that we have good control of the discriminant locus.  A common resolution $\widehat \calM$ of $\overline \calM$  (the GIT quotient) and  $\BG$
was constructed in \cite{act} and \cite{ls}; in fact $\widehat\calM$ is a  Kirwan blow-up  of  $\overline \calM$. The discriminant locus in  $\widehat \calM$ is now essentially a hyperplane arrangement.
Consequently, blowing up the linear strata (starting with the deepest one) will give a normal crossing compactification (where technically, we mean normal crossing in the sense of stacks; i.e.,  we get normal crossing after passing to a finite  cover).
Here  various choices are involved and it is not clear which one is optimal,  but there is,
in a sense,  a minimal one --- the so called wonderful blow-up of de Concini--Procesi.   Due to the fact that the cubic threefolds parameterized by $\overline{\mathcal M}$ have $AD$-singularities (with the exception of the chordal cubic which is dealt with differently), this wonderful blow-up is also distinguished due to its arising in a natural way  from the associated $ADE$ root systems.  Applying this wonderful blow-up, we
get an explicit normal crossing compactification $\tM$  (whose boundary divisors are indexed by singularities of the cubic).

In Section \ref{sec:Ag} we first recall some basic facts about compactifications of $\ab[g]$. In principle,
there are several possible approaches to the extension problem. One is to study directly the degenerations of  intermediate Jacobians as abelian varieties, an approach
which is implicit in  \cite{grhu1}. Here we choose a different approach which has proved very useful for the study of intermediate Jacobians of cubic
threefolds, namely we relate the problem to Prym varieties.
The point is that projection from a general line $\ell \subset X$ defines a conic bundle $\widetilde{X}_{\ell} \to \PP^2$. The discriminant curve of this conic bundle
is a plane quintic $D\subset \PP^2$ and if $\ell$ is chosen sufficiently general, we will call such an $\ell$ non-special,
then all singular conics are reduced pairs of lines. This defines an \'etale double cover $\widetilde{D} \to D$
and by Mumford's result the intermediate Jacobian $IJ(X)$  is isomorphic to the Prym variety of  $\widetilde{D} \to D$. Hence this connects our problem to the
extension of the Prym map and we recall the known results for this, in particular the definition of the Friedman--Smith loci.

This makes it  intuitively clear that the extension of the  intermediate Jacobian  map is related to the extension of the Prym map.
There are, however, some subtle points which have to be clarified in order to be able to handle the singular case and variations in families, and this is the content of Section \ref{sec:compare}.
Mumford's construction can indeed be generalized to GIT semistable cubics with isolated singularities (note that this excludes only the chordal cubic case) and will still give  a double cover  $\widetilde D\to D$, this time of a singular (reduced, but possibly reducible) quintic $D$.
By extending arguments due to Beauville in the smooth case, we show in Section \ref{sec:compare} that for a generic choice of line $\ell\subset X$, the natural map from deformations of the pair $(X,\ell)$ (deforming $X$ in $\mathbb P^4$) to deformations of the cover $\widetilde D \to D$ (deforming $D$ in $\mathbb P^2$) induced by projection from a line  is formally smooth (Proposition \ref{P:CtPQdef}).
This allows us to obtain an identification of discriminants (up to a smooth factor) between cubic threefolds and \'etale double covers of plane quintics (see Proposition \ref{P:DefCtoPQ}).

Recall that the construction of $\tM$ is based on performing a wonderful blow-up on the discriminant for cubics.
Via the identification of discriminants, we obtain thus a wonderful blow-up for double covers of quintics. As discussed in \cite{cml2}, this is essentially the same as performing a simultaneous semistable reduction for the associated universal family $\widetilde {\mathcal D}\to \mathcal D$ of double covers of plane  quintics.
In this way one can reduce the extension problem for pairs $(X,\ell)$ to an extension  problem for the Pryms of \'etale double covers of plane quintics.
Once this is established, a descent argument  then shows the extension of the intermediate Jacobian map to the boundary components
associated to cubics with $AD$-singularities of $\tM$ itself. Thus Section \ref{sec:compare} reduces the extension problem of the intermediate Jacobian map, at least outside the chordal cubic locus, to a study of \'etale double covers of plane quintics. This is the technical core
of our paper.

Once the reduction to Pryms  is established, the actual extension results are proved in Sections \ref{sec:proof} and \ref{sec:chordal}. Namely, in Section \ref{sec:proof}, an essentially combinatorial argument shows that \'etale double covers of reduced plane quintics never lie in the closure of  the Friedman--Smith loci in $\overline{\calR}_6$, and thus the period map for Pryms extends there.
Finally, in Section \ref{sec:chordal} we consider the degenerations to the chordal cubic (this is the only remaining case, where the techniques discussed above do not apply).
In this case, we note, based on \cite{act} and a natural compatibility of the wonderful blow-up, that the locus $\tDh$ in $\tM$ corresponding to chordal cubics is naturally identified with the moduli  of stable hyperelliptic genus $5$ curves (see also \cite{col}). Hence, by Mumford--Namikawa, there will be a natural extended period map from $\tDh$ to
$\Vor[5]$. Finally, a monodromy argument shows that this extension is compatible with the extension over $\tM\setminus \tDh$,  therefore giving an extension over all of
the wonderful blow-up  $\tM$.

In Sections \ref{sec:components} and \ref{sec:2strata} we will apply the extension result and classify the torus rank $1$ and torus rank $2$ degenerations of intermediate Jacobians respectively.

\subsection*{Acknowledgements} We are thankful to I.~Dolgachev for interesting discussions on cubics and their Fano surfaces of lines and to G.~M.~Greuel and E.~Shustin
for discussions about equisingular loci. In particular Remark \ref{rem:components} is due to E.~Shustin.
Through the years (investigating cubics, Pryms, and related objects) we have benefited from discussion with many people, especially V.~Alexeev, R.~Friedman, and F.~Viviani.

Much of the work was completed while K.~Hulek and R.~Laza were in residence at the Institute for Advanced Study in Princeton. During this period S.~Casalaina-Martin and
S.~Grushevsky visited IAS for short terms. We are all grateful for the wonderful research environment at IAS. The stay of R.~Laza was partially supported through the NSF institute grant DMS-1128155, the stay of K.~Hulek was partially supported by the Fund for Mathematics.

We are grateful to  the referee for very substantial feedback to our manuscript. This has led to significant improvements of the presentation and clarifications in various places.

\subsection*{Convention} Throughout the paper we will be working  over the complex numbers $\CC$.


\section{The wonderful compactification $\tM$ of the moduli of cubic threefolds}\label{sec:tM}
The purpose of this section is to recall the construction from \cite{cml} of a normal crossing (up to passing to finite covers) compactification $\tM$ for the moduli space of cubic threefolds. For such a compactification one can deduce easily from the Borel extension theorem that there exists an extended period map
$\tM\to \Sat[5]$ to the Satake compactification. The content of this paper is to prove that this extended period map actually lifts to the Voronoi compactification $\Vor[5]$. While the only thing needed for the extension to the Satake compactification is local liftability and the normal crossing assumption, for the extension to the Voronoi   compactification an extra ingredient is needed, namely sufficient geometric information about  $\tM$ that will  eventually allow one to control the monodromy. We review this construction here, based on \cite{cml} and the further refinements of \cite{cml2}.

The starting point of the construction of a normal crossing compactification $\tM$ (with weak geometric meaning) for the moduli space of cubic threefolds is to first
have a compactification $\widehat \calM$ for which we have
\begin{itemize}
\item[(1)] some geometric meaning (essentially in the sense of GIT),
\item[(2)] control of the discriminant locus $\widehat \calM\setminus \calM$.
\end{itemize}
Such a compactification was constructed in \cite{act} and \cite{ls} in connection with the ball quotient model of Allcock--Carlson--Toledo. The following three theorems summarize the necessary information about the compactification $\widehat \calM$.

The starting point of the construction of $\widehat \calM$ (and $\widetilde \calM$) is the GIT compactification:

\begin{teo}[{GIT compactification for cubic threefolds,  \cite{a}}] \label{T:GITcub}
Let $\overline{\calM}:=\overline{\calM}^{GIT}$ be the GIT compactification of the moduli space of smooth cubic threefolds. The following hold:
\begin{itemize}
\item[(1)] A cubic threefold is GIT stable if and only if it has at worst $A_1, \ldots ,A_4$-singularities.
\item[(2)] A cubic threefold is GIT semistable if and only if it has at worst $A_1, \ldots ,A_5$ or $D_4$-singularities or its orbit closure contains the chordal cubic\footnote{This final case corresponds to $A_k$ singularities ($6\le k\le 11$) satisfying a certain condition (\cite[Thm. 1.3(iii)]{a}) or the chordal cubic. In particular,  the chordal cubic is the only semistable cubic with non-isolated singularities.}.  In particular all (semi)stable cubics are integral.
\item[(3)] The GIT boundary (i.e., the complement of the stable locus in $\overline{\calM}$) consists of a rational curve $T$ and an isolated point $\Delta$.
\item[(4)] The polystable orbit corresponding to $\Delta$ is a cubic with $3D_4$-singularities (explicitly the cubic $V(x_0x_1x_2+x_3^3+x_4^3)$). Under a suitable identification $T\cong \PP^1$, the polystable orbits parameterized by $T\setminus\{0,1\}$ correspond to cubics with precisely $2A_5$-singularities. The special point $0\in T$ corresponds to a cubic with $2A_5+A_1$-singularities, and the special point $1\in T$ corresponds to the orbit of the chordal cubic. For future reference we denote the point corresponding to the orbit of the chordal cubic by $\Xi \in T$.
\end{itemize}
\end{teo}

\begin{rem}
We recall that the chordal cubic is the secant variety of the rational normal curve of degree $4$ in $\PP^4$. Thus, it follows that the chordal cubic is stabilized by a $\mathrm{PGL}(2)$ subgroup of $\mathrm{PGL}(5)$.
 In contrast, the stabilizers of all the other polystable cubics are virtually abelian. In particular $\overline \calM$ has toric singularities away from the special point $\Xi$. Kirwan has defined a stratification of the singularities of GIT quotients in terms of the stabilizers of the associated polystable orbits. From this perspective, the point $\Xi$ is the worst singularity of $\overline \calM$. By applying the Kirwan blow-up procedure (essentially, blow-up the semistable locus along the orbit of the chordal cubic), we obtain a partial resolution of $\overline \calM$ which will have  only toric singularities. It turns out that this is the model $\widehat \calM$ which we are looking for, see Theorem \ref{resgitball} below.
\end{rem}

\begin{dfn}[{\cite[Def.~2.2]{cml}}]\label{def-allowable}  We say that a hypersurface singularity is \emph{allowable} if it is either of type $A_k$ for some $k\le 5$, or of type $D_4$.
\end{dfn}

\begin{rem}\label{rem-allowable}
In view of Theorem \ref{T:GITcub}, saying that a cubic threefold $X$ has allowable singularities is a shorthand for
saying that $X$ is GIT semistable, but the closure of its orbit does not contain the chordal cubic. The separation of the semistable cubics into two cases (with allowable singularities, or  the chordal cubic) is natural. In the  former case, the reduction to the Prym map  works (see Sections \ref{sec:compare}), while in the latter, a different method is needed (see Section \ref{sec:chordal}).
In fact, for all our arguments, we can assume without loss of generality that we are working with polystable cubics (in which case allowable simply means not the chordal cubic). Namely, the Luna Slice Theorem implies that the local structure of the GIT quotient is governed by the slices transverse to the orbits of polystable points.  Moreover, Shah's polystable reduction for GIT \cite[Prop. 2.1, p.488]{shah} applied in our situation states that any family of polystable cubics over a punctured disk can be replaced by a family of isomorphic cubics that has limit a polystable cubic.
\end{rem}

The following Lemma will be used repeatedly:
\begin{lem}\label{L:TotTju}
 Let $X$ be a cubic threefold with allowable singularities. The following holds
\begin{equation}\label{tautotallow}
\tau_{tot}(X)\le 12,
\end{equation}
where $\tau_{tot}(X)=\sum_{p\in \Sing(X)} \tau(X,p)$ denotes the total Tjurina number, or equivalently (in this set-up) the total Milnor number.
\end{lem}
\begin{proof}
If $X$ has at worst $A_4$-singularities, then $X$ is GIT stable (Theorem \ref{T:GITcub}(1)). In particular, it has finite stabilizer. Applying \cite[Cor. 1.6]{dP}, we see that the universal family of cubic threefolds gives a simultaneous versal deformation for the singularities of $X$. It follows that $\tau_{tot}(X)$ is less than the dimension of the moduli space of cubic threefolds, i.e., $10$ (N.B. $10$ can be achieved, e.g. the Segre cubic). In the remaining cases, $X$ has either an $A_5$ or a $D_4$-singularity. Furthermore, $X$ is strictly semistable, and the orbit closure of $X$ contains a polystable cubic $X_0$ with $2A_5$, or $2A_5+A_1$, or $3D_4$-singularities. Using the semicontinuity of Tjurina numbers, we conclude that $\tau_{tot}(X)\le \tau_{tot}(X_0)$, and thus less than or equal to $12$ (which can be only achieved for the $3D_4$ case).
\end{proof}

While the GIT compactification has a rather nice geometric description, it is difficult to understand the structure of the discriminant. To get a hold on the discriminant, one needs to use the ball quotient model of Allcock--Carlson--Toledo \cite{act}. This model is based on an auxiliary construction involving the period map for cubic fourfolds. We will not recall the details, but only note the following:

\begin{teo}[{The ball quotient model,  \cite{act} and \cite{ls}}]\label{thm_ballmodel}
 Let $\calB/\Gamma$ be the ball quotient model of \cite{act}. The following hold:
 \begin{itemize}
\item[(1)] The period map
 $$P:\calM\to \calB/\Gamma$$
 is an open embedding with the complement of the image the union of two irreducible Heegner divisors $D_n:=\cDn/\Gamma$ (called the nodal divisor) and $D_h:=\cDh/\Gamma$ (called the  hyperelliptic divisor), where  $\cDn$ and $\cDh$ are $\Gamma$-invariant hyperplane arrangements.
 \item[(2)] The boundary of the Satake--Baily--Borel compactification $\BG$ consists of two cusps (i.e., $0$-dimensional boundary components).
 \end{itemize}
\end{teo}

Finally, we note that the ball quotient model is closely related to the GIT model. Specifically, the following holds:

\begin{teo}[{GIT to ball quotient comparison,  \cite{act} and \cite{ls}}]\label{resgitball}
As above, let $\overline{\calM}$ be the GIT compactification. Let $\BG$ be the Baily-Borel compactification of the ball quotient model of \cite{act}. Then there exists a diagram
$$
\xymatrix{
&\widehat{\mathcal M} \ar[ld]_{p} \ar[rd]^{q}&\\
\overline {\mathcal M}\ar@{-->}[rr]^{\overline P} &&  \BG \\
}
$$
resolving the birational map between $\overline \calM$ and $\BG$ such that
\begin{itemize}
\item[(1)] $p:\widehat \calM\to \overline\calM$ is the Kirwan blow-up of the point $\Xi\in \overline\calM$. The exceptional divisor $E:=p^{-1}(\Xi)$
 of this blow-up is naturally identified with the GIT quotient for $12$ unordered points in $\PP^1$.
 \item[(2)] $q:\hM \to  \BG$ is a small semi-toric modification as constructed by Looijenga \cite{l1}. The morphism $q$ is an isomorphism over the interior $\calB/\Gamma$ and one of the two cusps of  $\BG$.  The preimage under $q$ of the other cusp is a curve, which is identified with the strict transform of $T\subset \bM$ under $p$.
 \end{itemize}
 In particular note that the period map $P:\calM\to \calB/\Gamma$  extends to a morphism $\overline P$ everywhere on $\bM$  except the point $\Xi$. Furthermore, the following hold

 \begin{itemize}
  \item[(3)] The exceptional divisor $E\subset \hM$ of $p$ maps to the closure $D_h^*$ of the Heegner divisor $D_h=\cDh/\Gamma$ in $\BG$ and is an isomorphism over $D_h$ (i.e., $q(E)=D_h^*$, and $q_{\mid q^{-1}(D_h)}:q^{-1}(D_h)\subset E\to D_h$ is an isomorphism).
 \item[(4)] $q$ is an isomorphism over the stable locus (cubics with at worst $A_4$-singularities) in $\hM$ and in a neighborhood of $\Delta$ (the polystable orbit of cubics with $3D_4$-singularities). The image of the locus of cubics with $A_1,\dots,A_4$-singularities is $(\cDn\setminus\cDh)/\Gamma$.
 \item[(5)] $q$ maps $\Delta$ and the strict transform of the curve $T$ to the two cusps of $\BG$ respectively (N.B. $q$ is an isomorphism near $\Delta$, and a small contraction near $T$).
  \end{itemize}
\end{teo}

\begin{rem}
As discussed in \cite{act}, the identification of the exceptional divisor $E$ (which is by construction the GIT quotient for $12$ unordered points in $\PP^1$) with the Heegner divisor
$D_h=\cDh/\Gamma$ (which is naturally a $9$-dimensional ball quotient) is compatible with the Deligne--Mostow \cite{dm} construction. Furthermore, the moduli of $12$ distinct unordered points in $\PP^1$, which can be identified with $(\cDh\setminus \cDn)/\Gamma$, is naturally identifiable with the moduli of (smooth) hyperelliptic curves of genus $5$.   This will be relevant later in Section \ref{sec:chordal}.
Note that this is also related to Collino's proof that one-parameter degenerations of smooth cubic threefolds to the chordal cubic give rise to a hyperelliptic Jacobian in the limit of the associated family of abelian varieties \cite{col}.  Geometrically, given a pencil say $F_0+tF_1$, where $F_0$ is a homogeneous form defining the chordal cubic, and $F_1$ is a homogeneous form defining a smooth cubic threefold, the intersection of $\{F_1=0\}$ with the rational normal curve determines $12$ points on $\mathbb P^1$, and the limit abelian variety is the Jacobian of the hyperelliptic curve obtained as the double cover of $\mathbb P^1$ branched at those $12$ points.
\end{rem}

At this point we have obtained a compactification $\hM$ of the moduli of cubic threefolds on which we have both a good geometric description  (coming from GIT) and a good structure for the discriminant locus (coming from the ball quotient model).
However, for the purpose of studying the degenerations of the intermediate Jacobians of cubic threefolds, or equivalently of the associated
weight one variations of Hodge structures (VHS),
$\hM$ is not yet suitable. Namely, from a degeneration of VHS perspective, we need a smooth normal crossing compactification. In fact it suffices to allow smooth and normal crossing in a stack sense (i.e., up to passing to a finite cover). We  constructed such a model $\tM$ in \cite[\S  6]{cml}. While this construction is not unique, our space
$\tM$ has the advantage of being quite explicit and, in a certain sense, minimal.

\begin{construction}\label{constructmtilde}
Let $\widehat \calM$ be as in Theorem \ref{resgitball}. We construct $\widetilde \calM\to \widehat \calM$ as follows:
\begin{itemize}

\item[(Step 1)] {\it Consider the full toroidal resolution $\hM^{\operatorname{tor}}$ of $\BG$}.\\ Note that, since we are in a ball quotient case (i.e., rank $1$),  $\hM^{\operatorname{tor}}$ is canonical (i.e., it does not  depend on any choices). By  the general theory, $\hM^{\operatorname{tor}}$ will be smooth after passing to a finite cover (e.g.,  by taking a neat subgroup of $\Gamma$). Furthermore, since $\hM$ is semi-toric, by construction $\hM^{\operatorname{tor}}$ will dominate $\hM$ (see \cite[\S 5, esp. 5.2]{l1} for a precise discussion).

\item[(Step 2)] {\it Consider the wonderful blow-up associated to the hyperplane arrangement $\cDn$. This will induce a blow-up $\tM$ of  $\hM^{\operatorname{tor}}$, which will be a smooth normal crossing compactification of $\calM$ (after passing to finite covers).}\\
Let us recall that the complement of $\calM$ in $\calB/\Gamma$ is the union of two Heegner divisors $D_n=\cDn/\Gamma$ and $D_h=\cDh/\Gamma$, where $\cDn$ and $\cDh$ are hyperplane arrangements in $\calB$. Since we are working up to finite covers, it makes no difference if we  consider $\cDn$ and $\cDh$ or the associated quotients  $\cDn/\Gamma$ and
$\cDh/\Gamma$. Also, while there are infinitely many hyperplanes in these arrangements, the arrangements are always locally finite, and finite modulo the action of $\Gamma$.

From \cite{act} (see Section \ref{sec:chordal} below for further details), we see that $\cDh$ does not self-intersect and that it meets $\cDn$ transversally. Thus a blow-up of $\cDn$ making it normal crossing will make the entire complement of $\calM$  normal crossings. A hyperplane arrangement is easily resolved to normal crossing by blowing up linear strata starting with the minimal ones in the natural ordering. We apply here the wonderful blow-up, which simply means that one needs to blow up only the so called irreducible strata (see \cite{cml}, \cite{cml2}). This has the advantage of being somewhat minimal and does not to depend on the order in which the linear strata are blown up.   In fact, considering $\cDn$, the hyperplane arrangement is locally determined by hyperplane arrangements determined by $ADE$ root systems, and this can be explained also combinatorially (see \cite{cml2}).  A similar statement can be made for $\cDh$, although the geometric interpretation is different, having instead to do with $12$ points on $\mathbb P^1$ (see Remark \ref{rem12pts} for a discussion).

Having discussed the wonderful blow-up on the interior, we now recall  how this induces the blow-up on the toroidal resolution.  The exceptional divisors for the toroidal resolution $\hM^{\operatorname{tor}}\to (\mathcal B/\Gamma)^*$ are (up to finite covers) abelian varieties lying over the two cusps.  The hyperplane arrangement in the interior cuts on these exceptional divisors an abelian arrangement (see \cite[Ex. 1.6, \S3.3]{l1}). This is, in this case, determined  by a root system. Using this description one checks that the maximal irreducible strata in this abelian arrangement are induced from the hyperplane arrangement on the interior. It follows that the wonderful blow-up on the interior extends to give the wonderful blow-up
$\tM\to \hM^{\operatorname{tor}}$.
\end{itemize}
\end{construction}

We conclude:
\begin{teo}[{\cite[\S 6]{cml}}]\label{theo:wonderfulblowup}
Let $\tM$ be the space constructed in (\ref{constructmtilde}). Then
\begin{itemize}
\item[(1)] $\tM$ is a compactification of $\calM$ such that locally \'etale $\calM\subset \tM$ is the quotient of a normal crossing compactification by a finite group.
\item[(2)] $\tM$ resolves the rational map $\hM\dashrightarrow \Sat[5]$ as in the diagram:
\begin{equation}\label{tmdiagram}
\xymatrix{
&&{\tM} \ar[rrdd] \ar[ld]&&\\
&{\hM} \ar[ld]\ar@{-->}[rrrd] \ar[rd]&&&\\
{\bM} &&  \BG & & \Sat[5].\\
}
\end{equation}
\item[(3)] The boundary divisors $\tM \setminus \calM$ are $\tDA[1], \tDA[2], \tDA[3], \tDA[4], \tDA[5] , \tDD$
and $\tDh$ corresponding to the strict transform of the discriminant divisor in the moduli of cubics, the blow-up of the loci of cubics with a single $A_2,\dots, A_5, D_4$-singularity, and the blow-up of $\Xi\in \bM$ respectively.
\end{itemize}
\end{teo}
\begin{proof}
Item (1) follows by construction, (2) follows from (1) and the Borel extension theorem. Finally, the boundary divisors come from tracking the irreducible strata and the connection to the singularities of cubics --- see \cite[\S  6]{cml}, and also \cite{cml2} for further discussion.
\end{proof}
\begin{rem}\label{rem:components}
A priori we do not know whether all boundary divisors mentioned in the above theorem are irreducible. As pointed out to us by E. Shustin, one can use normal forms of the  cubic equations to see that  the divisors $\tDA[k]$ are irreducible for $k \leq 3$. Similarly, the locus of cubics with two $A_1$-singularities is irreducible, since
one can assume that the singularities are at given points in $\PP^4$ and then the only closed conditions are the vanishing of the gradients at these two points.
We will use this in Sections \ref{sec:components} and \ref{sec:2strata}.
\end{rem}

The main result of \cite{cml} is to identify the image of $\tM$ in $\Sat[5]$ and to describe where the various top strata are mapped --- see \cite[Table 1, p.52]{cml} and
Table \ref{tablecon}
for a summary.
The goal of the current paper is to lift the map $\tM\to \Sat[5]$ to the second Voronoi and perfect cone toroidal compactifications. Since $\calM\subset \tM$ has a toroidal structure given by the snc discriminant (property (1) of Theorem \ref{theo:wonderfulblowup} above), the question of lifting the period map to toroidal compactifications is a question about the combinatorics of monodromy cones. In principle, this can be analyzed directly in terms of the topology of degenerations of cubic threefolds. However, we have not been able to make a complete analysis using this approach. Therefore, we are using the auxiliary construction of Pryms associated to cubic threefolds (due to Mumford and Beauville) and analyze the degenerations of plane quintics (and their \'etale double covers) to conclude the desired extension from $\tM$ to $\Vor[5]$ and to $\Perf[5]$.

In the following sections, we will see that the wonderful blow-up construction is compatible with an analogous construction for plane quintics, and this is in turn closely related to the simultaneous semistable resolution for curves (as discussed in \cite{cml2}).


\section{Toroidal compactifications, the Prym map, and Friedman--Smith loci}\label{sec:Ag}
In order to describe our result we first have to recall some basic facts about compactifications of the moduli space ${\mathcal A}_g$ of principally polarized abelian varieties. The
{\em Satake compactification}, sometimes also called the {\em Baily--Borel--Satake compactification}  $\Sat$,  is the compactification given by the ring of modular forms. Set-theoretically
$$
\Sat = {\mathcal A}_g \sqcup {\mathcal A}_{g-1} \sqcup \ldots \sqcup {\mathcal A}_0.
$$
The boundary of $\Sat$ has codimension $g$ and the space is highly singular along its boundary. {\em Toroidal compactifications} were first introduced in \cite{AMRT} and have the property that
the boundary is a divisor. They, however,  depend on the choice of an admissible cone decomposition of the rational closure
of the cone $\operatorname{Sym}_{>0}^2(\RR^g)$ of positive  definite symmetric real $g \times g$ matrices. There are three well known possible such decompositions coming from the
reduction theory of quadratic forms. The resulting compactifications are  the {\em second Voronoi compactification}  $\Vor[g]$, the {\em first Voronoi} or {\em perfect cone compactification} $\Perf[g]$ and the
{\em central cone compactification} $\Centr[g]$.  For higher $g$  these decompositions are all different and none is a refinement of the other. For small $g$, however, there are some coincidences.
In particular all three compactifications coincide if $g \leq 3$. For $g=4$ the second Voronoi decompositions is a refinement of the perfect cone decomposition which, in this genus, coincides
with the central cone decomposition. Consequently, there is  a morphism $\Vor[4] \to \Perf[4]=\Centr[4]$.  The most relevant case for us is genus $5$. Here the second Voronoi decomposition
is still a refinement of the perfect cone decomposition \cite{rb}, a fact which fails for $g \geq 6$, see \cite{er}. Thus we have a morphism $\Vor[5] \to \Perf[5]$.
The central cone decomposition is now different from the perfect cone decomposition and neither is a refinement of the other.

By now the geometric meaning of the three toroidal compactifications is well known. The second Voronoi compactification $\Vor$ represents a moduli functor, as was shown by Alexeev \cite{alexeev02},
see also Olsson \cite{olsson}, and hence has an interpretation in terms of degenerate abelian varieties. More precisely, the boundary points correspond to stable semi-abelic varieties.
The perfect cone compactification $\Perf$, on the other hand, for $g\ge 12$ is a canonical model of $\calA_g$ in the sense of the minimal model program,
see Shepherd-Barron \cite{sb}, and
finally, the central cone compactification $\Centr$ is the Igusa blow-up of the Satake compactification $\Sat$, see Namikawa \cite{Nam}.

In order to put our results in perspective we want to recall very briefly what the situation is for extending the Torelli and the Prym maps to compactifications of the moduli spaces.
The Torelli map $t_g : {\mathcal M}_g \to {\mathcal A}_g$ which maps a curve $C$ to its polarized Jacobian ($\operatorname{Jac}(C), \Theta_C)$ is a morphism and it is natural to ask
whether this
can be extended to suitable compactifications of the source and target. For ${\mathcal M}_g$ there is a very natural choice of a compactification, namely the Deligne--Mumford
compactification $\overline{\mathcal M}_g$ parameterizing stable nodal curves.  For the target space ${\mathcal A}_g$ the choice is less clear. It was shown by Mumford and
Namikawa \cite{nam76I, nam76II} in the 1970's that the Torelli
map extends to a morphism $t_g^{V}: \overline{\mathcal M}_g \to \Vor$. The analogous question for the other toroidal compactifications of ${\mathcal A}_g$ was solved only much later, namely in 2012
by Alexeev and Brunyate \cite{ab}, see also \cite{aletal}.  They showed that the Torelli map extends as a morphism $t_g^{P}: \overline{\mathcal M}_g \to \Perf$. They also found that the situation is different
for the central cone compactification. The map $t_g^{C}: \overline{\mathcal M}_g \dasharrow \Centr$ is a morphism for $g \leq 8$, but has
points of indeterminacy for $g \geq 9$.
The case which is vital for our results is the extension of the Prym map. Let $C$ be a smooth curve of genus $g+1$ and $\pi: \widetilde C \to C$ be  a connected  \'etale degree $2$ cover.
The Prym variety $P=\operatorname{Prym}(\widetilde C \to C)$ is defined as the
identity component of the kernel of the norm map $\operatorname{Nm}(\pi): \operatorname{Jac}^0(\widetilde C) \to \operatorname{Jac}^0(C)$ given by
$\sum n_iP_i \mapsto \sum n_i\, \pi(P_i)$.
Clearly $P$ is an abelian variety of dimension $g$ and the restriction
of the theta divisor of $\widetilde C$ to $P$ gives ${\Theta_{\widetilde C}|}_P=2\Theta_P$ where $\Theta_P$ is a principal polarization.  Let
$$
{\mathcal R}_{g+1}= \{\pi: \widetilde C \to C\mid \, \pi \, \mbox{is a connected \'etale}\, \ 2:1 \, \mbox{ cover} \}
$$
be the space of  connected \'etale degree $2$ covers of smooth curves of genus $g+1$. Then the above procedure defines the so-called {\em Prym map}
$$
P_g : {\mathcal R}_{g+1} \to {\mathcal A}_g.
$$
The space ${\mathcal R}_{g+1}$ has a natural normal crossing (in a stack sense, i.e., after passing to a finite cover)  compactification  $\overline{{\mathcal R}}_{g+1}$ consisting of {\em admissible double  covers} of stable curves \cite{beauville}, which means that the
involution $\iota$ which is induced by the double cover does not have fixed nodes where the two
local branches are interchanged. In analogy we will also speak of {\em admissible involutions}. It was already noticed by Friedman and Smith in the mid 1980's that the Prym map
does not extend as a morphism to any (reasonable) toroidal compactification. Their examples were the so-called $FS_2$ covers.
We recall that a {\em Friedman--Smith cover of type $FS_n$} is an admissible double cover
$\pi: \widetilde C \to C$ where $\widetilde C= \widetilde C_1 \cup \widetilde C_2$ is the union of two smooth curves $\widetilde C_i, i=1,2$ intersecting in $2n$ points: $\widetilde C_1 \cap \widetilde C_2= \{P_1, \ldots , P_n,Q_1, \ldots ,Q_n\}$,
which admits an involution $\iota: \widetilde C \to \widetilde C$ which induces a fixed point free involution on each of the components $\widetilde C_i$ and interchanges the nodes $\iota(P_j)=Q_j, j=1 \ldots ,n$  pairwise,
such that  $C= \widetilde C / \langle \iota \rangle$ and $\pi: \widetilde C \to C$ is the quotient map. If we associate dual graphs to the curves $\widetilde C$ and $C$ in the usual way we obtain the  picture given in Figure \ref{Fig:dgFSG}.
\begin{figure}[htb]
\begin{equation*}
\xymatrix{
\widetilde \Gamma & *{\bullet} \ar @{-}@/_1pc/[rr]|-{\SelectTips{cm}{}\object@{>}}_{\tilde e_n^+} \ar @{-} @/_2.5pc/[rr] |-{\SelectTips{cm}{}\object@{>}}_{\tilde e_n^-}
 \ar@{-} @/^1pc/[rr]|-{\SelectTips{cm}{}\object@{>}}^{\tilde e_1^-}  \ar@{-}@/^2.5pc/[rr]|-{\SelectTips{cm}{}\object@{>}}^{\tilde e_1^+}^<{\tilde v_1}^>{\tilde v_2}  &\vdots &*{\bullet}
&&\Gamma&
 *{\bullet} \ar @{-}@/_1pc/[rr]|-{\SelectTips{cm}{}\object@{>}}_{e_n} \ar @{-}
 \ar@{-} @/^1pc/[rr]|-{\SelectTips{cm}{}\object@{>}}^{e_1}  ^<{ v_1}^>{ v_2}  &\vdots &*{\bullet}
}
\end{equation*}

\caption{  Dual graph of a Friedman--Smith example with $2n\ge 2$ nodes ($FS_n$).}\label{Fig:dgFSG}
\end{figure}
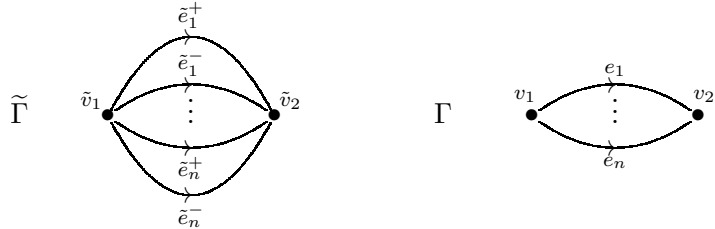

\noindent We also recall that the codimension of the $FS_n$ locus in $\overline {\mathcal R}_{g+1}$ is $n$.

The  Friedman--Smith covers are crucial for understanding the indeterminacy locus of the Prym map.
The first case to be analyzed in detail was the extension of the Prym map to the second Voronoi compactification.
It was shown, see  \cite{abh} and \cite{vologodsky} (see also \cite{cmghl})  that the indeterminacy locus of the Prym map
$$
P^{V}:\overline {\calR}_{g+1}\dashrightarrow \Vor
$$
is exactly the union $\cup_{n\geq 2}{\overline{FS}}_n$ of the closures of the loci of
Friedman--Smith covers of type $FS_n$ with $n\geq 2$.
In \cite{cmghl} we studied the situation for other toroidal compactifications. It turns out that the indeterminacy locus of the map $P^{P}:\overline {\calR}_{g+1}\dashrightarrow \Perf$ is strictly
smaller than for the map to $\Vor$. The closure $\overline{FS}_2 \cup \overline{FS}_3$ is still contained in the indeterminacy locus of $P^P$, but the open Friedman--Smith loci $FS_n, n \geq 4$ do not
meet the indeterminacy locus (although we could not exclude the possibility that some points in their closures, that are not already in $\overline{FS}_2 \cup \overline{FS}_3$,  might be in the indeterminacy locus). Finally, the indeterminacy locus of
$P^{{C}}:\overline {\calR}_{g+1}\dashrightarrow \Centr$ also contains
$\overline{FS}_2 \cup \overline{FS}_3$ and is disjoint from $FS_n, n\geq 4$, but for $g \geq 9$ it also contains points which are not degenerations of any $FS_n$-examples (in analogy to the
behavior of the Torelli map).

For future applications it is important to characterize the (degenerations of) Friedman--Smith covers combinatorially in terms of dual graphs with an admissible involution. For brevity we will  say that  a graph $\widetilde\Gamma$ together with an admissible involution $\iota$ lies
in the closure of the $FS_n$ locus if it arises from an admissible cover $\pi: \widetilde C \to C$ which is a degeneration of an $FS_n$ cover.
\begin{pro}\label{pro:FScombinatorics}
A dual graph $\widetilde\Gamma$ with an admissible involution $\iota$ lies in the closure $\overline{FS}_n$
of a Friedman--Smith locus $FS_n$  if and only if its set of vertices can
be decomposed into two disjoint subsets $\widetilde V=\widetilde V_1\sqcup \widetilde V_2$, such that the induced
subgraphs $\widetilde \Gamma_1$ and $\widetilde\Gamma_2$ (by which we mean the collections
of all edges both of whose endpoints are in $\widetilde V_i$) satisfy the following:
\begin{enumerate}
\item This decomposition is preserved by the involution $\iota$, i.e.,~$\iota(\widetilde \Gamma_i) =\widetilde \Gamma_i$;

\item the graphs $\widetilde\Gamma_i$ are connected;

\item $\widetilde\Gamma\setminus(\widetilde\Gamma_1\sqcup\widetilde\Gamma_2)$ is a collection of $2n$ edges (connecting a vertex in $\widetilde V_1$ to a vertex in $\widetilde V_2$), none of which are fixed by $\iota$.
\end{enumerate}
\end{pro}
(We note in particular that the last condition includes the statement that neither $\tilde V_1$ nor $\tilde V_2$ are empty.)
\begin{proof}
This is clear from the definition of a Friedman--Smith cover $\widetilde C = \widetilde C_1 \cup \widetilde C_2$ where the two components can degenerate further.
\end{proof}

\begin{rem}\label{R:combG}
For the dual graph $\Gamma$ of the curve $C$ there is a natural graph morphism $\pi:\widetilde \Gamma \to \Gamma$ (sending vertices to vertices and edges to edges).
The proposition above implies that if $\widetilde C\to C$ lies in $\overline{FS}_{n}$, then the graph $\Gamma$ has the property that its vertices can be decomposed  into disjoint subsets  $V=V(\Gamma)=V_1\sqcup V_2$ such that for the induced sub-graphs $\Gamma_1$ and $\Gamma_2$ we have
$\pi^{-1}(\Gamma_i)$ ($i=1,2$) connected, and $\Gamma_1$ and $\Gamma_2$ are connected by $n$ edges each of which has two pre-images (for instance, take $V_i=\pi(\widetilde V_i)$).
\end{rem}

\begin{rem}
An admissible graph $({\widetilde \Gamma},\iota)$ can have more than one decomposition of the type described above if it lies in the intersection of two components of the locus $\cup_{n \geq 2}{\overline {FS}}_n$, or in the self-intersection of some locus $\overline{FS}_{n}$.
\end{rem}

\begin{rem}
Vologodsky in \cite[Lem.~1.2]{vologodsky} has proven a slightly different sufficient criterion: if $\widetilde \Gamma$ contains  two disjoint connected equivariant subgraphs $\widetilde \Gamma'_1$ and
$\widetilde \Gamma'_2 $ which are connected by at least $2n$ edges and which are interchanged pairwise by the involution $\iota$, then this curve is a degeneration of a Friedman--Smith cover of
type $FS_m$ with $m \geq n$. Note that in his statement the union of the vertices $V(\widetilde \Gamma'_1) \sqcup V(\widetilde \Gamma'_2)$ may not be equal to all of $\widetilde V$.
We further note that Gwena in \cite[Lem.~4]{g} also states one direction of Proposition \ref{pro:FScombinatorics}.
\end{rem}


\section{Extending the intermediate Jacobian map versus extending the Prym map}\label{sec:compare}

The goal of this section is to reformulate the question of extending the period map for cubic threefolds as the  question  of extending the Prym period map for \'etale double covers of plane quintics.

\subsection{Statement of the theorem}
We start by stating the main theorem of this section.   Let us  fix  some notation.
  Fix  $X\subseteq \mathbb P^4$  a polystable cubic threefold with isolated singularities (see Theorem \ref{T:GITcub}).  Denote by $$\mathcal X\to B_{X\hookrightarrow \mathbb P^4}$$ an algebraic  semiuniversal
  embedded deformation of   $X$; for instance, we may take $B_{X\hookrightarrow \mathbb P^4}$ to be an open subset of the Hilbert scheme (the projective space $\mathbb P^{34}$) of cubic threefolds, and $\mathcal X$ to be the restriction of the universal family (see \S \ref{S:DefCub}).
We denote by $$F(\mathcal X)=F(\mathcal X/B_{X\hookrightarrow \mathbb P^4})\to B_{X\hookrightarrow \mathbb P^4}$$ the relative Fano variety of lines.
      Let $F^{ns}(\mathcal X)$ be the smooth open  subset  of $F(\mathcal X)$  consisting of non-special lines (see Definition \ref{D:NSL} and Proposition \ref{P:Fano}).
 Projecting a cubic from a non-special line induces an \'etale  double cover of the discriminant curve (see  \S \ref{S:cub+disc}).  This  induces a rational map
$$
F^{ns}(\mathcal X)\dashrightarrow \overline R_6,
$$
which is a morphism   over the locus of pairs consisting of a smooth cubic and a non-special line.

The discriminant in  $B_{X\hookrightarrow\mathbb P^4}$,  and its pull-back to  $F^{ns}(\mathcal X)$, are determined by an $ADE$ root system corresponding to the singularities of $X$ (see \S \ref{S:DiscCPQ}).   Therefore,  they admit wonderful blow-ups of Weyl covers
$\widetilde B'_{X\hookrightarrow\mathbb P^4}\to B_{X\hookrightarrow \mathbb P^4}$ and $\widetilde F^{'ns}(\mathcal S)\to F^{ns}(\mathcal X)$, respectively  (see \S \ref{S:DefWCWB}).  The wonderful blow-up of the Weyl cover is a canonically defined finite cover followed by a sequence of blow-ups, all determined by the root system, such that the discriminant is normal crossing.
The main point of this section is to prove   that the wonderful blow-up of the Weyl cover $\widetilde F^{'ns}(\mathcal X)  \to F^{ns}(\mathcal X)$ resolves the rational map $F(\mathcal X)^{ns}\dashrightarrow \overline R_6$:

\begin{teo}\label{T:IJtoP}
The rational map $ \widetilde F^{'ns}(\mathcal X)\dashrightarrow \overline R_6$  extends to a morphism.
\end{teo}

For the proof of the theorem (details in \S \ref{S:PfIJtoP}) we will reduce to the case of  double covers of plane quintics.   The basic strategy can be described as follows.
Due to  results of  Beauville \cite{bint,bdet},
there is (\'etale locally) a smooth morphism $F^{ns}(\mathcal X)\to B_{(\widetilde D,D\hookrightarrow \mathbb P^2)}$, where $B_{(\widetilde D,D\hookrightarrow \mathbb P^2)}$ is a  semiuniversal deformation   of the odd \'etale double cover of the discriminant   plane quintic (see \S \ref{S:DefCub}).
The discriminants in  both spaces  are determined by compatible $ADE$ root systems (see \S \ref{S:DiscCPQ}).   Therefore,   their wonderful blow-ups are  compatible in the sense that there is a commutative diagram
$$
\xymatrix@R=.5cm@C=.5cm{
\widetilde F^{'ns}(\mathcal X) \ar[r] \ar[d] &  \ar [d] F^{ns}(\mathcal X)\\
\widetilde B'_{(\widetilde D,D\hookrightarrow \mathbb P^2)}  \ar[r] & B_{(\widetilde D,D\hookrightarrow \mathbb P^2)}.\\
}
$$
Since the discriminant in the wonderful blow-up of the Weyl cover is normal crossing,  it follows from a    result of de Jong--Oort and Cautis  \cite{dJO,cautis} (see Theorem \ref{T:dJOC})   that $\widetilde B'_{(\widetilde D,D\hookrightarrow \mathbb P^2)} \dashrightarrow \overline R_6$ extends to a morphism.  This strategy of reducing the problem, via the explicit  wonderful blow-up of the Weyl cover, to a question about families of curves with  normal crossing boundary   is motivated  by the strategy of the proof of the main result of \cite{cml2}.

\vskip .1 in
Our motivation for proving Theorem \ref{T:IJtoP} comes from the following consequence, which essentially says that we have reduced the problem of resolving the period map for cubic threefolds  to resolving the Prym map for \'etale double covers of plane quintics.

\begin{cor}\label{C:IJtoP}
If the rational map $\widetilde F^{'ns}(\mathcal X)\to \overline R_6\dashrightarrow \Vor[5] $ extends to a morphism, then the rational map $\widetilde {\mathcal M}\to \Vor$ extends to a morphism in a neighborhood of the pre-image in $\widetilde {\mathcal M}$ of the point corresponding to $X$ in $\overline {\mathcal M}$.
\end{cor}

We can sketch the proof as follows; the details are in \S \ref{S:CubCor}.   Denote by $\mathcal X_{\text{abs}}\to B_X$   a semiuniversal deformation of $X$ as an abstract variety.    The discriminant in $B_X$ is determined by the same $ADE$ root system as in the other spaces (\S \ref{S:AbsDef}), and we  denote by $\widetilde B_X'\to B_X$ the wonderful blow-up.  During the course of the next several subsections, we explain how we  obtain a diagram  (where all arrows are defined \'etale locally)

\begin{equation}\label{E:ExtDiag}
\xymatrix@R=.5cm@C=.6cm{
&&&&&\widetilde F^{'ns}(\mathcal X) \ar[d]_<>(0.5){\text{smooth}} \ar@{.>}[llld] \ar@/^1pc/@{->}[rdddd]&\\
&&F^{ns}(\mathcal X) \ar@{.>}[d]_{\text{smooth}}&&&\widetilde B'_{X\hookrightarrow \mathbb P^4} \ar[ld]_<>(0.5){\text{smooth}} \ar[dd] \ar@{.>}[llld]&\\
&&B_{X\hookrightarrow \mathbb P^4} \ar@{.>}[ld]_{\text{smooth, forgetful functor \ \ \ }} \ar@{.>}[dd]^<>(0.1){(-)/\!\!/\operatorname{SL}}&&\widetilde B'_{X}  \ar[ld]^<>(0.2){(-)/\operatorname{Aut}(X)}  \ar@{.>}[llld]&&\\
&B_{X}\ar@{.>}[ld]_{(-)/\operatorname{Aut}(X)} &&\widetilde B'_{X}/\operatorname{Aut}(X)  \ar@{.>}[llld]\ar[rr]^{\text{\'etale}}&&\widetilde {\mathcal M} \ar@{-->}[rd]\ar@{.>}[llld] &\\
B_{X}/\operatorname{Aut}(X) \ar@{.>}[rr]^<>(0.6){\text{\'etale}}_<>(0.5){\text{Luna Slice Theorem}}&&\overline{\mathcal M} \ar@{-->}[rrrr]&&&&\Vor. \\
}
\end{equation}
The left side of the diagram has the stated properties via deformation theory, GIT, and the theory of Fano varieties.   The properties on the right hand side are deduced from those on the left hand side because the discriminants are all canonically identified, and consequently  the wonderful blow-ups are \'etale locally obtained from the others via fibered products.
The corollary then follows from the diagram via standard results on extending rational maps (e.g., Proposition \ref{P:ExtProp}).

\subsection{Preliminaries on cubic threefolds and discriminant double covers}\label{S:cub+disc}
In this section we review the connection between cubic threefolds  and double covers of plane quintics obtained from projecting the threefold from  a line.  The case of smooth cubic threefolds is the well-known story going  back to Mumford \cite{mprym}, and Beauville \cite{bint, bdet}.    Here we discuss  some  extensions to cubics with allowable singularities (in particular, polystable cubics with isolated singularities; see Remark \ref{rem-allowable}), referring also to  \cite[\S 3.2]{cml} and \cite{cmf}.  Most of the results in this section hold for cubics with slightly more general isolated double point singularities; however, for simplicity,  and uniformity with the rest of the paper, we only state the results for cubics with allowable singularities
(see  Definition \ref{def-allowable}).
In this section, the restrictions are used only to assure that the plane quintic cannot consist of a line meeting the residual quadric in a single point; later in \S \ref{S:DiscCPQ}, the hypothesis will be important in assuring that the total Tjurina number is less than $16$, in order to compare embedded versus abstract deformations.
We use the terminology allowable singularities, rather than polystable cubics with isolated singularities, so that we can use the same terminology for describing singularities of plane quintics.
(We recall that two hypersurface singularities of different dimensions are said to have the same type if, after an analytic change of coordinates, they differ by an iterated suspension, i.e., $V(f(x_1,\dots,x_k))\subset (\CC^k,0)$ and  $V(f(x_1,\dots,x_k)+x_{k+1}^2+\dots+x_{k+l}^2)\subset (\CC^{k+l},0)$ have the same type.)

\subsubsection{Cubic threefolds, conic bundles and the discriminant curve}
 Let $X$ be a cubic threefold  with isolated singularities, and $\ell\subseteq X$ a line not passing through any of the singular points of $X$.
   The blow-up $\mathbb P^4_\ell$ of the ambient projective space $\mathbb P^4$ along $\ell$ gives a commutative diagram
\begin{equation}\label{E:FinQdef}
\xymatrix{
X_\ell \ar@{^(->}[r] \ar@{->}[rd]_f & \mathbb P^4_\ell \ar@{->}[d]^\tau \\
& \mathbb P^2\\
}
\end{equation}
where $X_\ell$ is the strict transform of $X$, and $f$ and $\tau$ are the linear projections with center $\ell$.   The fibers of $f$ are conics (the residual conic to a plane through $\ell$), and  the general fiber of $f$ is a smooth conic.

Just as in the smooth case, $X_\ell/\mathbb P^2$ is a  quadric in the standard sense that there exists a rank $3$ vector bundle $E$ on $\mathbb P^2$, and a line bundle $L$ on $\mathbb P^2$ such that $X$ is defined in $\mathbb PE$ by $q \in  H^0(\mathbb P^2, \operatorname{Sym}^2(E^\vee) \otimes  L)=H^0(\mathbb PE,\mathcal O_{\mathbb PE}(2)\otimes f^*L)$.  The proof is \cite[Prop.~1.2 (ii)]{bint}; the key point to note is that since the fibers of $f$ are of dimension $1$, and Gorenstein (they are each hypersurfaces in $\mathbb P^2$), we may obtain $E$ as $(f_*\omega_{X_\ell/\mathbb P^2})^\vee$, so that we do not need the stronger  condition used in the proof of \cite[Prop.~1.2 (ii)]{bint} for the case of higher dimensional fibers, namely  that $X_\ell$ be factorial.

The discriminant $D$  of $f$ is by definition the locus in $\mathbb P^2$ where the fibers are singular.  This comes with a natural scheme structure.   For a  quadric defined by a section  $q:E\to E^\vee\otimes L$,
the discriminant  is defined by the section
$$
\delta= \det(q) \in  H^0(\mathbb P^2, \operatorname{Sym}^2(\det E^\vee) \otimes  L^{\otimes 3}).
$$

For cubic threefolds, it can be quite useful to describe the discriminant  explicitly with equations, and indeed we will use this later.
 After a change of coordinates, we may as well assume that $\ell$ is the line
 \begin{equation}\label{E:ExpL}
\ell=\{X_0=X_1=X_2=0\},
\end{equation}
 and we may then take $X$ to be defined by the homogeneous polynomial
\begin{equation}\label{E:ExpCubEq}
F(X_0,\cdots,X_4)=H+2X_3Q_1+2X_4Q_2+X_3^2L_1+2X_3X_4L_2+X_4^2L_3
\end{equation}
where $H$ (resp.~$Q_1,Q_2$, resp.~$L_1,L_2,L_3$) is a cubic (resp.~are quadrics, resp.~are linear forms) in $X_0,X_1,X_2$.  For each choice of $X_0,X_1,X_2$, we obtain a (non-homogeneous) quadric in $X_3,X_4$, which is the fiber of the fibration in quadrics $f:X_\ell\to \mathbb P^2$ (after projective completion).  The Jacobian criterion for smoothness then yields a determinantal condition for singularities; namely, it is immediate to check that the discriminant $D$ is defined by the determinant of the matrix
\begin{equation}\label{E:ExpCubMat}
M=
\left(
\begin{array}{ccc}
L_1 & L_2& Q_1\\
L_2& L_3 & Q_2\\
Q_1& Q_2& H
\end{array}
\right).
\end{equation}
Since the generic fiber of $f$ is smooth, this determinant is not identically zero, and therefore, $D=V(\det M)$ is a plane quintic.

\begin{rem}
Essentially the same analysis above shows that the conic bundle $X_\ell/\mathbb P^2$ determined by the cubic and line from Equation \eqref{E:ExpCubEq} and \eqref{E:ExpL} can be described explicitly in the following way.  Set $E=\mathcal O_{\mathbb P^2}(-2)\oplus \mathcal O_{\mathbb P^2}(-2)\oplus \mathcal O_{\mathbb P^2}(-3)$, set $L=\mathcal O_{\mathbb P^2}(-3)$ and take $q\in  H^0(\mathbb P^2, \operatorname{Sym}^2(E^\vee) \otimes  L)$ to be  defined by $M$ in \eqref{E:ExpCubMat}.  Then $X_\ell$ is determined by $V(q)$ in $\mathbb PE$.
\end{rem}

There is another discriminant that is important to consider.  Let $E\subseteq X_\ell$ be the exceptional divisor for the blow-up $X_\ell\to X$.  There is an induced map $f|_E:E\to \mathbb P^2$, which is a double cover (for a point in $\mathbb P^2$, one obtains a plane in $\mathbb P^4$ through $\ell$, and the intersection of the residual conic with $\ell$ give the two points in the pre-image on $E$).  Let $\mathcal Q\subseteq \mathbb P^2$ be the branch locus of  $f|_E\to \mathbb P^2$; this is identified generically with the locus where the residual conic meets $\ell$ in a single point.  One can check that $\mathcal Q$ is a conic, and in the coordinates above, one has that $\mathcal Q$ is defined by the determinant of the matrix
\begin{equation}\label{E:MatB}
B=
\left(
\begin{array}{cc}
L_1 & L_2\\
L_2 & L_3
\end{array}
\right).
\end{equation}

\subsubsection{The double cover of the discriminant}

We now examine the Stein factorization of $f|_{f^{-1}(D)}:f^{-1}(D)\to D$,  the restriction of the fibration in quadrics to the discriminant (here we really mean the blow-up of $f^{-1}(D)$ at the locus of singular points; this locus is isomorphic to $D$; we can also simply take the relative Fano variety of lines).   The following definition is quite useful.

\begin{dfn}[{\cite[Def.~3.4]{cml}}]\label{D:NSL}  Let $X$ be a cubic threefold with isolated singularities, and $\ell\subset X$ a line.  We say that $\ell$ is a \emph{non-special line} if for every other
 line $\ell'\subseteq X$ meeting $\ell$, the plane spanned by $\ell$ and $\ell'$ cuts out three distinct lines on $X$.
\end{dfn}

\begin{rem}\label{R:NSL}
A non-special line $\ell$ is contained in the smooth locus of $X$ (see \cite[Rem.~3.5 i)]{cml}).
\end{rem}

The main point from our perspective is the following:

\begin{pro}[{\cite[Pro.~3.6, Lem.~3.9, Clm.~1 p.39]{cml}}]\label{P:Fano}
Suppose that $X$ is a cubic threefold with allowable singularities.  Then the Fano variety  $F(X)$ of lines on $X$ is a (possibly reducible) surface, and the non-special lines form a (Zariski) open subset $F^{ns}(X)\subseteq F(X)$.

Moreover, if $\ell$ is a non-special line on $X$,
and $D$ is the discriminant curve obtained from the projection  $f:X_\ell \to \mathbb P^2$, then
\begin{enumerate}
\item There exists a natural $1$-to-$1$ correspondence between the singularities of $D$ and those of $X$, preserving the type of the $AD$-singularity.  More precisely, if $c$ is a local (analytic) equation for a singularity of $D$, then $c+x_3^2+x_4^2$ is a local equation for the corresponding singularity of $X$.

\item The Stein factorization of $f|_{f^{-1}(D)}:f^{-1}(D)\to D$ (see above) gives a sequence
$$
\begin{CD}
f^{-1}(D)@>>> \widetilde D @>\pi >> D
\end{CD}
$$
where $\pi:\widetilde D\to D$ is an \'etale double cover.
\end{enumerate}
In addition, for  the conic  $\mathcal Q\subseteq \mathbb P^2$  obtained as the branch locus of $f|_E:E\to \mathbb P^2$,  $\mathcal Q$ is smooth and does not pass through any singular points of $D$.
\end{pro}

\begin{rem}\label{R:NSL2}
From Remark \ref{R:NSL} and \cite[Thm.~1.10, p.11]{ak} the open subset $F^{ns}(X)\subseteq F(X)$ of non-special lines is contained in the smooth locus of $F(X)$.
\end{rem}

\subsubsection{Cubic threefolds and theta characteristics on plane quintics}

Let $X$ be a cubic threefold with allowable singularities, and let $\ell$ be a non-special line.  As described above, we obtain an
 \'etale double cover
$$
\pi:\widetilde D\to D
$$
of the discriminant curve $D$.  Let $\eta$ be the $2$-torsion line bundle associated to this  cover; concretely, $\pi_*\mathcal O_{\widetilde D}$ splits canonically into even and odd parts as $\mathcal O_D\oplus \eta$.   Since $D$ is a plane quintic, it has a distinguished theta characteristic, namely $\mathcal O_D(1)$, and we denote by  $\kappa=\eta \otimes \mathcal O_D(1)$ the theta characteristic on $D$ determined by the double cover.
We will call an \'etale double cover of a plane quintic {\em even} or {\em odd} depending on the parity of $h^0(\kappa)$.

\begin{pro}\label{P:B1} Let $(X,\ell )$ be a  cubic threefold $X$ with allowable  singularities and non-special line $\ell$, taken in coordinates so that $X$ and $\ell$ are determined by  equations \eqref{E:ExpCubEq} and \eqref{E:ExpL}.
  Let $\eta$ be the $2$-torsion line bundle associated to the \'etale double cover $\pi:\widetilde D\to D$  of the discriminant curve $D$ obtained by projection from $\ell$.  Then the associated theta characteristic on $D$,  $\kappa:=\eta\otimes \mathcal O_D(1)$,
admits a resolution of the form
\begin{equation}\label{E:TCRes}
\begin{CD}
0@>>> \mathcal O_{\mathbb P^2}(-2)^{2} \oplus  \mathcal O_{\mathbb P^2}(-3) @>M>>  \mathcal O_{\mathbb P^2}(-1)^{2} \oplus  \mathcal O_{\mathbb P^2}@>>> \kappa@>>>0.
\end{CD}
\end{equation}
with $M$ the matrix  in \eqref{E:ExpCubMat}.
Moreover,  $\kappa$
 satisfies:
\begin{enumerate}
\item $h^0(\kappa)=1$;

\item $h^0(\kappa(-1))=0$.

\item The non-trivial sections of $\kappa$ are not killed by a linear form;  i.e., the cup product map
$$
H^0(D,\mathcal O_D(1))\otimes H^0(D,\kappa)\to H^0(D,\kappa(1))
$$
is injective.

\end{enumerate}
\end{pro}

\begin{proof}  For a smooth cubic threefold, this is due to Beauville \cite[6.27]{bint} (for details see \cite[Prop.~4.2]{cmf}).     The outline of the argument is as follows. We have seen already that $D=V(\det M)$. This matrix $M$ also determines a sheaf $\kappa'$ via the short exact sequence
\begin{equation}\label{E:TCRes'}
\begin{CD}
0@>>> \mathcal O_{\mathbb P^2}(-2)^{2} \oplus  \mathcal O_{\mathbb P^2}(-3) @>M>>  \mathcal O_{\mathbb P^2}(-1)^{2} \oplus  \mathcal O_{\mathbb P^2}@>>> \kappa'@>>>0.
\end{CD}
\end{equation}
From general results of Beauville \cite[Prop.~4.2]{bdet}, $\kappa'$ is also  a theta characteristic on $D$.
An explicit argument in coordinates  (see  \cite[Prop.~4.2]{cmf}) establishes that $\kappa\cong \kappa'$.
Once one knows that $\kappa\cong\kappa'$ then (1) and (2) follow from the long exact sequence in cohomology (in the smooth case they can also be established directly from the geometry).  (3) is also clear since $D$ is irreducible (we give the argument below in the general case).

Now for the singular case.
We start with the same set-up, and
it is clear from the discussion above that the key point (at least up to establishing (3)) is showing that $\kappa$ and $\kappa'$ agree.  Note that \emph{a priori} from the arguments in \cite{bdet} we only know that $\kappa'$ is a rank $1$, torsion-free theta characteristic (i.e., ${\mathcal Hom}_{\mathcal O_D}(\kappa', \omega_D)\cong \kappa'$).
The main issue is that at the singular points of $D$, the matrix $M$ may drop rank by more than $1$, giving a sheaf that is  not locally free.  However, since the line $\ell$ is non-special, Proposition \ref{P:Fano} implies that $\mathcal Q$  does not pass through any singular points of $D$; in particular, the sub-matrix $B$ of $M$ \eqref{E:MatB} has non-zero determinant at the singular points of $D$, and so the rank of $M$ is at least $2$.  Consequently, we may conclude that $\kappa'$ has
 fibers of dimension $1$ at every point of $D$; i.e.,  for each $d\in |D|$, we have $\dim_{\kappa(d)}\kappa'\otimes \kappa(d)=1$, where $\kappa(d)$ is the residue field at $d$.   One can then conclude that $\kappa$ is a line bundle (e.g., \cite[Ex.~II.5.8]{h}).

Finally, since $\kappa$ and $\kappa'$ are line bundles, we can use a degeneration argument to show they agree.  That is, we take families of smooth cubic threefolds and non-special lines  $(X_t,\ell_t)$ over the unit disk $\Delta$ degenerating to the given pair $(X,\ell)$ at $t=0$.
When $t\ne 0$, we have $\kappa_t=\kappa_t'$.  Let $\mathscr D/\Delta$ be the associated family of discriminant plane quintics.  A theorem of Raynaud (see \cite[Thm.~7 p.258, Thm.~2, p.259]{neron})  implies that $\operatorname{Pic}^0_{\mathscr D/\Delta}$  is
separated.  Therefore the line bundles $\kappa$ and $\kappa'$ in the limit must be the same (we actually employ Raynaud's theorem for $\kappa_t \otimes \mathcal O_{D_t}(-1)=\kappa'_t \otimes \mathcal O_{D_t}(-1)$).

For the cup product statement (3), we can  argue geometrically.  Suppose a linear section killed the non-trivial global section (up to scaling it is unique).  Then this global section must be supported on a line; in other words, the discriminant curve must contain a line, and the theta characteristic has a unique non-trivial global section (up to scaling), which is supported on this line.   The restriction of the theta characteristic to this line must be a root of the restriction of the canonical line bundle $K_D=\mathcal O_D(2)$ to the line.  Therefore it is a section of $\mathcal O_{\mathbb P^2}(1)$ restricted to the line.  It must vanish at all intersections with other components of $D$ (to be killed by the linear form which is zero only on the line, the section must be zero already on the other components), and so the line meets the rest of $D$ in a unique point.
A reducible plane quintic $D$ consisting of a line $D'$ meeting the  residual plane quartic $D''$ in a single point $p$ will not have allowable  singularities (e.g., \cite[Proof of Cor.~3.7]{cml}), giving a contradiction with Proposition \ref{P:Fano}.  More precisely, if $p$ is a type $A_k$-singularity of $D$, then by B\'ezout, it will be of type $A_k$ with $k\ge 7$, and if it is a type $D_k$-singularity, it will be of  type $D_k$ with $k\ge 8$, none of which are allowable.
\end{proof}

\begin{rem}\label{R:B1}
Note that the fact that $h^0(\kappa)=1<h^0(\mathcal O_D(1))$ implies that $\eta$ is nontrivial, and consequently, that the cover $\widetilde D\to D$ is nontrivial.
\end{rem}

\begin{rem}\label{E:KappaConic}
The theta characteristic $\kappa$ has yet another description.  The curve $\mathcal Q$ restricted to $D$ defines an effective  Weil (in fact Cartier) divisor $\mathcal Q|_D$ on $D$ supported on smooth points, since it meets $D$ at smooth points.   There is an effective Weil  divisor $\sqrt {K_D}$ supported at the same  smooth points as $\mathcal Q|_D$ such that $2\sqrt {K_D}=\mathcal Q|_D$, and
$\kappa=\mathcal O_D(\sqrt{K_D})$. In the case of a smooth cubic threefold and non-special line this is shown in the proof of \cite[Prop.~4.2]{cmf}. In the case of singular cubic threefolds, using the fact that we know that $\mathcal Q$ is smooth and meets $D$ at smooth points, the same analysis as in \cite[Prop.~4.2]{cmf} shows that there is an effective Weil divisor $\sqrt {K_D}$ supported on smooth points of $D$, and hence Cartier,  such that $2\sqrt {K_D}=\mathcal Q|_D$.
The same degeneration argument as above shows that $\kappa=\mathcal O_D(\sqrt{K_D})$.
\end{rem}

We also have a converse:

\begin{pro} \label{P:B2}
Let $D$ be a plane quintic with allowable  singularities.  Let $\kappa$ be a theta characteristic (line bundle) on $D$ satisfying
\begin{enumerate}
\item $h^0(\kappa)=1$.

\item $h^0(\kappa(-1))=0$.

\item The non-trivial sections of $\kappa$ are not killed by a linear form;  i.e., the cup product map
$$
H^0(D,\mathcal O_D(1))\otimes H^0(D,\kappa)\to H^0(D,\kappa(1))
$$
is injective.

\end{enumerate}
Then $\kappa$ admits a resolution as in \eqref{E:TCRes}, with matrix $M$ as in   \eqref{E:ExpCubMat}, and $D=V(\det M)$.
Moreover,  setting $X$ to be the cubic defined by \eqref{E:ExpCubEq},  and $\ell$ to be the line defined by \eqref{E:ExpL}, if $\ell$ does not pass through any of the singularities of $X$, then $X$ has allowable singularities, $\ell$   is non-special, and $\kappa$ is the theta characteristic associated to the \'etale double cover of $D$ obtained by projecting  $X$ from $\ell$.
\end{pro}

\begin{proof} The existence of the resolution \eqref{E:TCRes} in
the case where $D$ is smooth is explained in \cite[Prop.~4.2]{bdet}.  The case where $D$ is singular is essentially the same.  Since $\kappa$ is a line bundle on a Cohen--Macaulay curve, it is arithmetically Cohen--Macaulay \cite[Def.~(1.1)]{bdet}.  This provides the existence of the two-step minimal  resolution of $\kappa$
$$
\begin{CD}
0@>>> L_0^\vee(-3)@>M>>  L_0 @>>> \kappa @>>> 0
\end{CD}
$$
where $L_0$ is a direct sum of line bundles on $\mathbb P^2$, and $D=V(\det M)$ \cite[Thm.~B, Cor.~1.8, (1.7)]{bdet}.
 The condition $h^1(\kappa(1))=h^0(\kappa(-1))=0$ ensures that $\kappa$ is $2$-Castelnuovo--Mumford regular.
 This means that $L_0$ can be taken to be $\mathcal O_{\mathbb P^2}(-2)^q\oplus \mathcal O_{\mathbb P^2}(-1)^p \oplus \mathcal O_{\mathbb P^2}$ for some non-negative integers $p,q$ (the single copy of $\mathcal O_{\mathbb P^2}$ is determined by $h^0(\kappa)=1$).
 Since the resolution is minimal, the summand $\mathcal O_{\mathbb P^2}(-1)$ of $L_0^\vee(3)$ is mapped into $\mathcal O_{\mathbb P^2}$; this implies that $q\le 1$.  However, condition (3) implies in fact that $q=0$.  The fact that $D=V(\det M)$ is a plane quintic then implies that $p=2$.
 Therefore,  $\kappa$ admits a resolution as in \eqref{E:TCRes}, with matrix $M$ as in   \eqref{E:ExpCubMat}, and $D=V(\det M)$.

Now that we have the forms in the matrix $M$, let $(X,\ell)$ be the cubic  and line defined by  \eqref{E:ExpCubEq} and \eqref{E:ExpL}.  Let $f:X_\ell\to \mathbb P^2$ be the associated fibration in conics \eqref{E:FinQdef}.  The first observation is that since $\kappa$ is a line bundle, the co-rank of $M$ at any point $p\in \mathbb P^2$, and hence the co-rank of the associated conic, i.e., the fiber of $f$ over $p$, is at most $1$.
 We will now use this to show that $X$ has isolated allowable singularities in one-to-one  correspondence with those of $D$, including the type of the $AD$-singularity,  and that $\ell$ is non-special.

By assumption, there are no singularities of $X$ along $\ell$.
To show that $X$ has isolated singularities, we argue as follows.
It is elementary that any singularity of $X$ maps to a singularity of $D$, so let us consider a singular point $p\in D$.
It follows immediately from the co-rank observation (see e.g., \cite[Proof of Lem.~1.5.2]{bint}) that there is a local analytic
neighborhood of  $p\in \mathbb P^2$ such that $X_\ell$ is defined by an equation $X_3^2+X_4^2+cT^2$, where $c$ is the formal equation of the discriminant $D$ near $p$ in $\mathbb P^2$.   Since the only singularity of the associated conic, i.e., the fiber of $f$ over $p$, is at the point $[0:0:1]$ (in homogeneous coordinates $X_3,X_4,T$), and the singularities of $D$ are isolated by assumption, we see that the singularities of $X$ are isolated, as well.  Moreover, since  the local equation of the singularity of $D$ at $p$ is given by $c$, and  the local equation of the corresponding  singularity of $X$ is given by $x_3^2+x_4^2+c$, we see that the singularity types of the singularities are identified, as well.  Finally, the co-rank condition  on the conics that are the fibers of $f$ implies immediately that for every other line $\ell'\subseteq X$ meeting $\ell$, the plane spanned by $\ell$ and $\ell'$ cuts out three distinct lines on $X$, so that $\ell$ is non-special.

All that remains to establish the proposition is to show that $\kappa$ is the theta characteristic associated to the double cover induced by projection of $X$ from $\ell$.  But this now follows  from Proposition \ref{P:B1}, since $X$ has allowable  singularities and $\ell$ is non-special.
\end{proof}

\subsection{Deformations of cubic threefolds and plane quintics}\label{S:DefCub}
In the previous section, we identified pairs $(X,\ell)$ and $(D,\kappa)$, where $X$ is a cubic with allowable singularities, $\ell\subseteq X$ is a non-special line, $D$ is a plane quintic with allowable singularities, and $\kappa$ is an odd theta characteristic line bundle on $D$ (satisfying several further conditions).  The goal of this section is to use this correspondence to identify the associated embedded deformation spaces, up to smooth factors (Proposition \ref{P:CtPQdef}); by embedded deformations we mean that we deform the cubics in $\mathbb P^4$ and the plane quintics in $\mathbb P^2$ respectively.
At the level of deformation functors, this is essentially a formal consequence of the results in the previous section, since the arguments  hold over local Artin rings (of finite type over $\CC$).
This implies the corresponding results for deformation spaces since the associated Hilbert functors, and relative Picard functors are representable. Note that  since plane quintics have a canonical theta characteristic $\mathcal O_D(1)$, a pair $(D,\kappa)$ is equivalent to a pair $(D,\eta)$, where $\eta$ is a $2$-torsion line bundle, which is in turn equivalent to an \'etale double cover $\widetilde D\to D$, and so we identify these deformation problems and deformation spaces, as well.
As in the previous section, the results hold for a slightly more general class of singular cubics and plane quintics, but for concision, we will only state the results for cubics and plane quintics with allowable singularities.

\smallskip
Let $X$ be a complex hypersurface of degree $d$ in $\mathbb P^n$.  We denote by $\operatorname{Def}_{X\hookrightarrow \mathbb P^n}$ the local Hilbert deformation functor at $X$; in other words the deformation functor of $X$ as a closed subscheme of $\mathbb P^n$ (e.g., \cite[\S 3.2.1]{sernesi}).  The global Hilbert functor is representable, by a projective space, and we will take
$$
\xymatrix{
X\ar[r] \ar[d]& \mathcal X \ar[d] \\
\operatorname{Spec} \mathbb C \ar[r]& B_{X\hookrightarrow\mathbb P^n}
}
$$
to be the restriction of the universal family to a sufficiently small
open neighborhood of the point corresponding to $X$.  Note that as our arguments progress, we may replace this open neighborhood with a suitable \'etale cover.
The fibered product diagram above induces a morphism of functors of Artin rings
$$
B_{X\hookrightarrow \mathbb P^n} \to \operatorname{Def}_{X\hookrightarrow \mathbb P^n}
$$
that is formally smooth and an isomorphism on tangent spaces (in other words, in the language of \cite[Def.~2.2.6]{sernesi},  $\mathcal X\to B_{X\hookrightarrow \mathbb P^n}$  is a flat algebraic deformation of $X\subseteq \mathbb P^n$ that induces a semiuniversal formal element).

We are also interested in deforming linear spaces  along with the hypersurface.  Given a linear space $\ell \subseteq X\subseteq \mathbb P^n$ of dimension $m$,  define $\operatorname{Def}_{(X\hookrightarrow \mathbb P^n,\ell)}$ to be the associated deformation functor.  Let $F(\mathcal X)\to B_{X\hookrightarrow \mathbb P^n}$ be the relative Fano scheme of linear subspaces of dimension $m$ for $\mathcal X\to B_{X\hookrightarrow \mathbb P^n}$, and let $\mathcal L$ be the universal family of linear spaces, in the sense that we have a diagram
$$
\xymatrix@R=1em@C=1em{
&\ell \ar[rr] \ar@{_(->}[ld] \ar@/^.7pc/[ldd]&&\mathcal L \ar@{^(->}[ld] \ar@/^1pc/[ldd]\\
X \ar[rr] \ar[d]&&\mathcal X\times_{B_{X\hookrightarrow \mathbb P^n}} F(\mathcal X) \ar[d]&\\
\operatorname{Spec}\mathbb C \ar[rr]&&F(\mathcal X)&\\
}
$$
inducing a morphism
$$
F(\mathcal X)\to \operatorname{Def}_{(X\hookrightarrow \mathbb P^n,\ell)}
$$
that is formally smooth and an isomorphism on tangent spaces.  In other words, Fano schemes of linear spaces are examples of Hilbert schemes, which are representable.
In fact, in the case of cubic threefolds and lines, if we take $\ell$ to be a non-special line on a cubic threefold $X$ with allowable singularities, and set $F^{ns}(\mathcal X)$ to be the locus of non-special lines, then we still have that
$$
F^{ns}(\mathcal X)\to \operatorname{Def}_{(X\hookrightarrow \mathbb P^n,\ell)}
$$
 is formally smooth and an isomorphism of tangent spaces.

From a cubic threefold with allowable singularities and a non-special line, we obtain an \'etale double cover $\widetilde D\to D$ of a plane quintic $D$.  We would like to relate the deformation functors.
Let $D\subseteq \mathbb P^2$ be a plane quintic with allowable singularities, and let $\operatorname{Def}_{D\hookrightarrow \mathbb P^2}$ be the associated deformation functor.
Let $\mathcal D\to B_{D\hookrightarrow \mathbb P^2}$, as above, be
the restriction of the universal family to a sufficiently small
open subset containing the point corresponding to $D$, giving a semiuniversal formal element.

Let $\widetilde D\to D$ be an odd connected \'etale double cover.
Setting $\operatorname{Def}_{(\widetilde D,D\hookrightarrow \mathbb P^2)}$ to be the deformation functor for the double cover $\widetilde D\to D$, in the sense that we consider embedded deformations of the base $D$ in $\mathbb P^2$, and \'etale double covers of the base curves,  we claim this deformation functor also admits an algebraic deformation inducing a semiuniversal formal element.  To this end,
let $P^0_{\mathcal D/B}\to B_{D\hookrightarrow \mathbb P^2}$ be the connected component of the identity in the relative Picard scheme.  Let $B_{(\widetilde D,D\hookrightarrow \mathbb P^2)}\subseteq P^0_{\mathcal D/B}\to B_{D\hookrightarrow \mathbb P^2}$ be the sub-scheme of $2$-torsion line bundles (the kernel of the composition of the diagonal and the addition map for the group scheme).  This is an \'etale group scheme over $B_{D\hookrightarrow \mathbb P^2}$
(being a quasi-finite flat group scheme in characteristic $0$; or alternatively, being the kernel of an \'etale homomorphism of group schemes).
We view $B_{(\widetilde D,D\hookrightarrow \mathbb P^2)}$ as a pointed (disconnected) scheme, with the base-point determined by the double cover $\widetilde D\to D$, and from now on we denote by   $B_{(\widetilde D,D\hookrightarrow \mathbb P^2)}$ the connected component containing this point.

Identifying \'etale double covers with $2$-torsion line bundles, and using the fact that the Picard scheme and universal object induce a semiuniversal formal element for the Picard sheaf, we see that $B_{(\widetilde D,D\hookrightarrow \mathbb P^2)}$ together with the restriction of the  universal object  induce a semiuniversal formal element for $\operatorname{Def}_{(\widetilde D,D\hookrightarrow \mathbb P^2)}$.
Note that $\operatorname{Def}_{(\widetilde D,D\hookrightarrow \mathbb P^2)} \to  \operatorname{Def}_{D\hookrightarrow \mathbb P^2}$ is formally smooth and an isomorphism on tangent spaces; for instance, this follows from the fact that $B_{(\widetilde D,D\hookrightarrow \mathbb P^2)}\to B_{D\hookrightarrow \mathbb P^2}$ is \'etale, as explained above.   

\begin{pro}\label{P:CtPQdef}
Let $X\subseteq \mathbb P^4$ be a cubic threefold with allowable  singularities and let $\ell\subseteq  X$ be a non-special line.  Let $\pi:\widetilde D\to D$ be the connected \'etale double cover of the discriminant  plane quintic determined from projection from $\ell$.   Projection from lines induces a formally smooth morphism of  deformation functors
$$
\operatorname{Def}_{(X\hookrightarrow \mathbb P^4,\ell)}\to \operatorname{Def}_{(\widetilde D,D\hookrightarrow \mathbb P^2)}.
$$
\end{pro}

\begin{proof}
This is  the statement that Beauville's analysis \cite{bdet}, in particular  Propositions \ref{P:B1} and  \ref{P:B2} above,  hold at the level of local Artin rings.  We leave the details to the reader.
\end{proof}

\subsection{Discriminants for cubic threefolds and plane quintics}\label{S:DiscCPQ}

For an isolated singularity $x\in X$, of analytic type $T$, we denote by $\operatorname{Def}_{T}$ the deformation functor of the singularity of $X$ at $x$.  For local complete intersection singularities, there is a deformation  $\mathcal X_T\to B_T$ of the singularity of $X$ at $x$  that is semiuniversal in the sense that the induced map
of deformation functors
$$
B_T\to \operatorname{Def}_{T}
$$
is formally smooth and an isomorphism on tangent spaces.  If $X$ has exactly $n$ singular points $x_1,\ldots,x_n$, which are isolated lci singularities  of types $T_1,\ldots, T_n$ respectively, then after replacing the $B_{T_i}$ with suitable \'etale covers,
one obtains a commutative diagram of deformation functors
\begin{equation}\label{E:DefDiag}
\xymatrix{
B_{X\hookrightarrow \mathbb P^n} \ar[r] \ar[d]& \prod B_{T_i} \ar[d] \\
\operatorname{Def}_{X\hookrightarrow \mathbb P^n}\ar[r]& \prod \operatorname{Def}_{T}.
}
\end{equation}

The basic fact we want to use is the following:

\begin{fact}[{du Plessis--Wall \cite{DPW00}}] \label{F:DPW}
 Given a complex hypersurface $X$ of degree $d$ in $\mathbb P^n$ with only isolated singularities, the universal  family of hypersurfaces of degree $d$ induces a simultaneous versal deformation of all the singularities of $X$, provided $\tau(X)<\delta(d)$, where $\tau(X)$ is the total Tjurina number, and $\delta(d)=16$, $18$ or $4(d-1)$, for   $d=3$, $4$ or $d\ge 5$, respectively.
 \end{fact}

\begin{rem}\label{R:FDPW}
From Lemma~\ref{L:TotTju}, we have that Fact~\ref{F:DPW} applies to cubics with allowable singularities.  Since on a cubic with allowable singularities there is a non-special line, and the discriminant plane quintic $D$ associated to projecting from this line is a plane quintic with the same type of singularities (Proposition \ref{P:Fano}), we have that $\tau_{tot}(D)=\tau_{tot}(X)\le 12$, so that  Fact \ref{F:DPW} applies also in this situation.
\end{rem}

In other words, under the hypotheses of Fact \ref{F:DPW}, the horizontal maps in the Diagram \eqref{E:DefDiag}  above are formally smooth, and we have an \'etale equivalence
\begin{equation}
 B_{X\hookrightarrow \mathbb P^n}\simeq_{\operatorname{et}} B_{T_1}\times\ldots\times B_{T_n}\times{\mathbb A}_\CC^m
\end{equation}
for some $m$.
Let  $\Delta_{X\hookrightarrow \mathbb P^n}\subseteq B_{X\hookrightarrow \mathbb P^n}$, and $\Delta_{T_i}\subseteq B_{T_i}$ be the discriminants.   For an lci singularity, the discriminant is a divisor.  Under the identification above, and setting $\pi_i$ to be the projection onto $B_{T_i}$, we have
\begin{equation}
\Delta_{X\hookrightarrow \mathbb P^n} = \pi_1^*\Delta_{T_1}+\ldots+\pi_n^*\Delta_{T_n}.
\end{equation}

Restricting to allowable singularities, after choosing a generic line $\ell\subset X$, guarantees simultaneous versal deformations for both $X$ and the associated quintic $D$.
Note that for $ADE$-singularities, the semiuniversal spaces and deformation functors  for singularities of the same type in different dimensions can be naturally identified.
The following proposition establishes a  natural corresponding identification of the associated discriminants.

\begin{pro}\label{P:DefCtoPQ}
Let $(X,\ell)$ be a cubic threefold with $n$ allowable singularities  $x_1,\ldots x_n$ of types $T_1,\ldots,T_n$ respectively, together with a non-special line $\ell$. Then \'etale locally,
\begin{equation}
 B_{(X\hookrightarrow \mathbb P^4,\ell)}:=F^{ns}(\mathcal X)\simeq_{\operatorname{et}} B_{T_1}\times\ldots\times B_{T_n}\times{\mathbb A}_\CC^m
\end{equation}
for some $m$, and under this identification, setting $\pi_i$ to be the projection onto $B_{T_i}$, the pull-back $\Delta_{(X\hookrightarrow \mathbb P^4,\ell)}$ of the discriminant $\Delta_{X\hookrightarrow \mathbb P^n}$ from $B_{X\hookrightarrow \mathbb P^n}$ to $B_{(X\hookrightarrow \mathbb P^4,\ell)}$ can be described as:
\begin{equation}
\Delta_{(X\hookrightarrow \mathbb P^4,\ell)} = \pi_1^*\Delta_{T_1}+\ldots+\pi_n^*\Delta_{T_n}.
\end{equation}

Let $\widetilde D\to D$ be the  \'etale double cover of the plane quintic obtained by projecting $X$ from $\ell$,   with $D$ having   singularities  $y_1,\ldots y_n$ of types $T_1,\ldots,T_n$ respectively, obtained from the singularities $x_1,\ldots,x_n$ of the cubic.
Then \'etale locally,
\begin{equation}
 B_{(\widetilde D,D\hookrightarrow \mathbb P^2)}\simeq_{\operatorname{et}} B_{T_1}\times\ldots\times B_{T_n}\times{\mathbb A}_\CC^{m'}
\end{equation}
for some $m'$, and under this identification, setting $\pi_i$ to be the projection onto $B_{T_i}$, we have
\begin{equation}
\Delta_{(\widetilde D,D\hookrightarrow \mathbb P^2)} = \pi_1^*\Delta_{T_1}+\ldots+\pi_n^*\Delta_{T_n}.
\end{equation}

Finally, the formally smooth morphism $\operatorname{Def}_{(X\hookrightarrow \mathbb P^4,\ell)}\to \operatorname{Def}_{(\widetilde D/D\hookrightarrow \mathbb P^2)}$ of Proposition \ref{P:CtPQdef} induces, after taking appropriate \'etale covers,  a smooth morphism
$$
B_{(X\hookrightarrow \mathbb P^4,\ell)}\to B_{(\widetilde D/D\hookrightarrow \mathbb P^2)}
$$
such that the $B_{T_i}$ and $\Delta_{T_i}$ in both spaces  are identified.
\end{pro}

\begin{proof}
This follows from the discussion above, and the commutative diagram of deformation functors:
$$
\xymatrix@R=1em@C=3em{
&\operatorname{Def}_{(X\hookrightarrow \mathbb P^4,\ell)}\ar@{->}[r]^{\text{formally smooth}}_{\text{Prop.~\ref{P:CtPQdef}}} \ar@{->}[d]_{\text{formally smooth,  non-special   line}}&\operatorname{Def}_{(\widetilde D,D\hookrightarrow \mathbb P^2) }\ar@{->}[d]^{\text{formally smooth, isom.~tangent spaces}} & \\
& \operatorname{Def}_{(X\hookrightarrow \mathbb P^4)} \ar@{->}[d] _{\text{formally smooth, du Plessis--Wall}}& \operatorname{Def}_{(D\hookrightarrow \mathbb P^2)}\ar@{->}[d]^{\text{formally smooth, du Plessis--Wall}}&  \\
&\prod \operatorname{Def}_{T_i}\ar@{=}[r]& \prod \operatorname{Def}_{T_i}.&\\
}
$$
\end{proof}

\subsection{Abstract deformations and forgetful functors} \label{S:AbsDef}
Let $X\subseteq \mathbb P^n$ be a hypersurface.  There is a natural forgetful functor
$$
\operatorname{Def}_{X\hookrightarrow \mathbb P^n}\to \operatorname{Def}_X
$$
to the functor of abstract deformations of $X$.
For reduced hypersurfaces of degree $d$ in $\mathbb P^n$ with $n\ge 3$, $d\ge 2$,  and $(n,d)\ne (3,4)$, the forgetful functor $\operatorname{Def}_{X\hookrightarrow \mathbb P^n}\to \operatorname{Def}_X$ is formally smooth (e.g., \cite[p.135]{sernesi}).  In particular, it is formally smooth for cubic threefolds with isolated singularities. For cubic threefolds with isolated singularities of the type prescribed by du Plessis--Wall (Fact \ref{F:DPW}), i.e., $\tau_{tot}(X)<16$,
 we have at the level of semiuniversal spaces a diagram
\begin{equation}\label{E:EmbAbsDef}
\xymatrix{
B_{X\hookrightarrow \mathbb P^n} \ar[rr]^{\text{smooth, forgetful functor}} \ar[rd]_{\text{smooth, du Plessis--Wall}\ \ \ } &&B_X \ar[ld]\\
&\prod B_{T_i}&\\
}
\end{equation}
identifying the discriminant spaces.  It follows from the diagram that $B_X\to \prod B_{T_i}$ is formally smooth, and in this way we also obtain
\begin{equation}
 B_{X}\simeq_{\operatorname{et}} B_{T_1}\times\ldots\times B_{T_n}\times{\mathbb A}_\CC^m
\end{equation}
for some $m$.

If $X$ is a polystable cubic with isolated singularities (i.e., with allowable singularities, so that in particular $\tau_{tot}(X)<16$ by Remark \ref{R:FDPW})
we can say a little more using the Luna Slice Theorem.   First, it is elementary to show that  a transverse slice to the orbit $O_X$ of the point corresponding to $X$ in the Hilbert scheme, and the restriction of the universal family to this slice, provides a semiuniversal deformation of $X$ as an abstract variety.
In other words, we may take $B_X$ to be the Luna Slice, and we obtain that  $B_{X\hookrightarrow \mathbb P^4}$ is \'etale equivalent to $O_X\times B_X$.  This then gives an explicit  description of the compatibility of the discriminants in Diagram \eqref{E:EmbAbsDef}, above; namely the additional smooth factor in $B_{X\hookrightarrow \mathbb P^4}$  corresponds to $O_X$.

  The Luna Slice Theorem states further  that there is an \'etale morphism $B_X/\operatorname{Stab}_X\to \overline {\mathcal M}$, where $\operatorname{Stab}_X$ is the stabilizer of the point corresponding to $X$.
Via the natural  isogeny $\operatorname{SL}_5(\mathbb C)\to \mathbb P\operatorname{GL}_5(\mathbb C)$, we may identify
$B_X/\operatorname{Stab}_X\cong B_X/\operatorname{Aut}(X)$. We note here that the automorphism groups of $X$ as a variety and of $X$ as a cubic threefold coincide. This follows since by the Lefschetz theorem  \cite[Cor.~3.7, p.158]{grothendieck1968sga} the Picard group of $X$ has a unique positive generator $\mathcal O_X(1)$, which must then be preserved under every automorphism, together
with the fact that we embed by a full linear system.
\begin{rem}
For plane curves, the forgetful functor is not in general formally smooth.
 However, for curves with locally planar singularities, it is well known (see e.g., \cite{cml2}) that the natural morphism $\operatorname{Def}_D\to \prod \operatorname{Def}_{T_i}$ is formally smooth, and
  as in Fact \ref{F:DPW}, we obtain a commutative diagram
$$
\xymatrix{
B_{D\hookrightarrow \mathbb P^2} \ar[rr] \ar[rd]_{\text{smooth, du Plessis--Wall\ \ \ \ }} &&B_D \ar[ld]^{\text{smooth}}\\
&\prod B_{T_i}.&\\
}
$$
\end{rem}

\subsection{Weyl covers and wonderful blow-ups}\label{S:DefWCWB}

As explained in \cite{cml2}, spaces of the form
$B=B_{T_1}\times\ldots\times B_{T_n}\times{\mathbb A}_\CC^m$ where the $T_i$ are $ADE$-singularity types admit wonderful blow-ups of Weyl covers, determined by the associated root systems (we refer the reader to \cite{cml2} for more details).  We denote the Weyl cover by $B'$ and the wonderful blow-up of the Weyl cover by $\widetilde B'$.

Since all of the spaces we have been discussing are of this type,
and have compatible discriminants (as $\tau(X)=\tau(D)<16$; Remark \ref{R:FDPW}), we obtain a commutative  diagram of wonderful blow-ups of Weyl covers
\begin{equation}\label{E:WBfibprod}
\xymatrix{
\widetilde B'_X  \ar[d]&  \ar[l] \widetilde B'_{X\hookrightarrow \mathbb P^4} \ar[d]& \ar[l] \widetilde B'_{(X\hookrightarrow \mathbb P^4,\ell)} \ar[r] \ar[d]& \widetilde B'_{(\widetilde D,D\hookrightarrow \mathbb P^2)} \ar[r] \ar[d]& \widetilde B'_{D\hookrightarrow \mathbb P^2}  \ar[d]& \\
B_X &  \ar[l] B_{X\hookrightarrow \mathbb P^4} & \ar[l] B_{(X\hookrightarrow \mathbb P^4,\ell)} \ar[r]& B_{(\widetilde D,D\hookrightarrow \mathbb P^2)} \ar[r] & B_{D\hookrightarrow \mathbb P^2} \\
}
\end{equation}
where each square is a fibered product.   Recall that we have shown that all of the horizontal morphisms are smooth.
 In fact, the wonderful blow-up is described entirely in terms of  the factors coming from the deformations of singularities, which are identified in each space.

\subsection{Completion of the proof of Theorem \ref{T:IJtoP}} \label{S:PfIJtoP}

The following theorem is a direct consequence of  a theorem due to de Jong--Oort and Cautis:

\begin{teo}[{de Jong--Oort \cite[Thm.~5.1]{dJO}, Cautis \cite[Thm.~A]{cautis}}]\label{T:dJOC}
Let $B$ be a regular scheme (over $\mathbb Z[1/2]$) and $\Delta$ an nc divisor on $B$. Set $U =B-\Delta$. Given a family
$$
\xymatrix{
\widetilde {\mathscr C}_U \ar[rr]^{\pi_U} \ar[rd]& & \mathscr C_U \ar[ld]\\
& U&
}
$$
of connected  \'etale double covers of smooth curves of genus $g\ge 2$, the induced rational map to the moduli scheme $B\dashrightarrow \overline{\mathcal  R}_g$ extends to a morphism over $B$.
\end{teo}

\begin{proof}
The family of connected  \'etale double covers over $U$ induces rational maps
$$
\xymatrix{
&\overline{\mathcal R}_g \ar[d]\\
B\ar@{-->}[r] \ar@{-->}[ru]& \overline{\mathcal M}_g.\\
}
$$
By the de Jong--Oort and Cautis theorem, the induced rational map $B\dashrightarrow \overline {\mathcal M}_g$ extends to a morphism.
The rational map $B\dashrightarrow \overline{\mathcal R}_g$ then extends to a morphism since $B$ is normal, $\overline{\mathcal R}_g\to \overline{\mathcal M}_g$ is finite surjective, and $\overline{\mathcal M}_g$ is a variety (see e.g.,  Proposition \ref{P:ExtProp}).
\end{proof}

We can now complete the proof of Theorem \ref{T:IJtoP}.

\begin{proof}[Proof of Theorem \ref{T:IJtoP}]
At this point we have explained the composition
$$
\widetilde F^{'ns}(\mathcal X)\to \widetilde B'_{(\widetilde D,D\hookrightarrow \mathbb P^2)}\dashrightarrow  \overline {\mathcal R}_6.
$$
  Since the discriminant in the wonderful blow-up of the Weyl cover is a normal crossings divisor in a smooth variety,  the extension of the map $B'_{(\widetilde D,D\hookrightarrow \mathbb P^2)}\dashrightarrow  \overline {\mathcal R}_6$ follows from the theorem of de Jong--Oort and Cautis.
\end{proof}

\subsection{Proof of Corollary \ref{C:IJtoP}}\label{S:CubCor}

As pointed out after the statement of Corollary \ref{C:IJtoP}, the proof will follow  from standard extension results, once we establish that the morphisms in \eqref{E:ExtDiag} have the stated properties.

We start on the left hand side of  Diagram \eqref{E:ExtDiag}. The fact that $F^{ns}(\mathcal X)\to B_{X\hookrightarrow \mathbb P^4}$ is  smooth
follows from Remark \ref{R:NSL2}. The fact that $B_{X\hookrightarrow \mathbb P^4}\to B_X$ is  smooth is a general result about forgetful functors for embedded deformations, discussed in \S \ref{S:AbsDef}.   Finally, the Luna Slice Theorem (see \S \ref{S:AbsDef}) asserts that there is an \'etale morphism $B_X/\operatorname{Aut}(X)\to \overline{\mathcal M}$.

We now move to the right hand side of Diagram \eqref{E:ExtDiag}.   We showed in \eqref{E:WBfibprod} that the wonderful blow-ups are compatible, and obtained from one another via fibered product diagrams.  In particular, the properties of the corresponding morphisms on the right hand side follow from those on the left hand side by base change.  In other words, $\widetilde F'^{ns}(\mathcal X)\to \widetilde B'_{X\hookrightarrow \mathbb P^4}$ and  $\widetilde B'_{X\hookrightarrow \mathbb P^4}\to \widetilde B'_X$ are formally smooth.
Now there is a natural action of $\operatorname{Aut}(X)$ on $\widetilde B'_X$, and from the construction of $\widetilde {\mathcal M}$, there is an \'etale morphism $\widetilde B_X'/\operatorname{Aut}(X)\to \widetilde {\mathcal M}$ (N.B. recall that $\widetilde {\mathcal M}$ is obtained via a wonderful blow-up of the discriminant in the ball quotient; up to passing to a finite cover, the discriminant is identified with the discriminant in the versal deformation).

We have now established  that the morphisms in \eqref{E:ExtDiag} have the stated properties.  To complete the proof we use the standard extension results for rational maps (see e.g., Proposition \ref{P:ExtProp} below).  In particular,  from Proposition \ref{P:ExtProp} we obtain a morphism $\widetilde B_X'\to \Vor$, since all of the morphisms in question are smooth and surjective.
The morphism $\widetilde B_X'\to \Vor$ is $\operatorname{Aut}(X)$-equivariant, and so we obtain a morphism $\widetilde B_X'/\operatorname{Aut}(X)\to \Vor$.    The morphism $\widetilde B_X'/\operatorname{Aut}(X)\to \widetilde {\mathcal M}$ is \'etale, and so the image is an open set over which the rational map $\widetilde {\mathcal M}\to \Vor$ extends to a morphism. This completes the proof.

\subsection{Extending rational maps}
For completeness we summarize below some  standard results we have been using regarding  extending rational maps, which follow  from \cite[Prop.~20.3.11, Prop.~20.4.4]{EGAIV4}.

\begin{pro}\label{P:ExtProp}
 Let $S$ be a scheme.
Let $X,Y$ be integral locally Noetherian $S$-schemes.
Assume that $X$ is normal, and $Y$ is separated and locally of finite type over $S$.
Consider a commutative diagram of morphisms and rational maps of $S$-schemes
$$
\xymatrix{
X'\ar@{-->}[r]^{\phi'} \ar[d]_f& Y' \ar[d]^g\\
X \ar@{-->}[r]^\phi& Y\\
}
$$
and assume that  $f$ is a composition of smooth surjective  and finite surjective morphisms, and $g$ is a finite surjective morphism.
Then $\phi$ extends to a morphism if and only if $\phi'$ extends to a morphism.
\end{pro}

\begin{proof}
The case where $f$ is smooth and surjective and $g$ is the identity  is a special case of  \cite[Prop.~20.3.11(ii)]{EGAIV4}.  The case where  $f$ is finite surjective and $g$ is the identity  can be deduced from  \cite[Prop.~20.4.4]{EGAIV4} (see also \cite[Lem.~2.4]{cautis} for a similar argument).   The case where $f$ is the identity and $g$ is finite surjective can also be deduced from \cite[Prop.~20.4.4]{EGAIV4}.  The general case follows from these special cases.
\end{proof}


\section{Combinatorics of monodromy cones for double covers of plane quintics}\label{sec:proof}
Every pair $(X,\ell)$ consisting of a semi-stable cubic $X$ with allowable (isolated) singularities  and a non-special line $\ell \subset X$ leads, via the conic bundle
given by projection from $\ell$, to an \'etale  double cover $\widetilde D \to D$
of a plane quintic $D$ with $AD$-singularities.
Given a one-parameter family of smooth cubics with non-special lines degenerating to $(X,\ell)$, one obtains a one-parameter family of \'etale double covers of smooth curves degenerating to $\widetilde D\to D$.
Moving to a stable reduction, one
finally obtains an \'etale admissible
double cover $\widetilde D_\stab\to D_\stab$. The aim of this section is to prove the following: whenever we start with $(X,\ell)$ as above, then the admissible double cover  $\widetilde D_\stab\to D_\stab$
is never contained in the closure $\overline{FS}_n\subset\overline\calR_g$ of any Friedman--Smith locus with any $n\ge 2$. This is an essential step in proving the extension of the
intermediate Jacobian map.

Since Proposition~\ref{pro:FScombinatorics} characterizes dual graphs of admissible covers which lie in the closure $\overline{FS}_n\subset\overline\calR_g$,
the question finally becomes combinatorial. In order to keep track of the singularities involved, as well as the process of taking a stable reduction, and possibly contracting
smooth rational components,  we introduce some further graph-theoretic notation, which allows us to deal combinatorially with arbitrary stable reductions of singular curves. Eventually we deduce from Proposition~\ref{pro:FScombinatorics} a necessary condition for some stable reduction of a singular curve to lie in $\overline{FS}_n$, which is the graph-theoretic Corollary~\ref{cor:bipartiteFS}.

Applying this condition to double covers of plane quintics is then a matter of analysis of cases, done in the proof of Proposition~\ref{prop:nonFS}. For our extension result we only need to consider plane quintics with $A_k, k\leq 4$ or
$D_4$-singularities, but this combinatorial proposition turns out to hold for all plane quintics with $AD$-singularities, which is the context we adopt for the rest of this section.

\subsection{Preliminaries on bipartite dual graphs of singular curves}
Throughout this section, we denote by $V(\Gamma)$ the set of all vertices of a graph $\Gamma$. When we speak of a {\em graph map}, we always mean  a map of cell complexes, i.e.,~edges are sent to edges, or collapsed to vertices, while vertices are mapped to vertices. Equivalently, this means that for each edge, the graph map either maps the edge isomorphically onto an edge, or contracts that edge to a vertex.

\begin{dfn}\label{def:bipartitegraph}
Let $C$ be a reduced curve, in particular $C$ has isolated singularities. We associate to $C$ a labeled {\em bipartite} graph~$\Gamma$ by the following procedure.
To each irreducible component of the curve $C$ we associate a black vertex (and typically label it
with the genus of its normalization) and to each singularity $p\in C$ we associate a white vertex and label it with the type of the singularity.
The edges of the graph are defined as follows. Let $\nu: C^{\nu}\to C$ be the normalization of $C$
(and note that the black vertices correspond also to the components of $C^{\nu}$).
To every pre-image $x\in \nu^{-1}(p)$ of a singular point $p \in C$ we associate an edge linking the white vertex corresponding to $p$ and the black
vertex corresponding to the component of $C^{\nu}$ on which $x$ lies. We call the result the {\em bipartite dual graph} of the curve $C$.
\end{dfn}

\begin{exa}\label{E:D4graph}
Suppose that $C=L_1\cup L_2\cup L_3$ is a reducible plane cubic comprised of three distinct lines $L_1,L_2,L_3$, all meeting at a single point $p$, which is a $D_4$-singularity.     Then the bipartite dual graph of $C$ is given in Figure \ref{F:D4}, below.
\end{exa}

\begin{figure}[htb]
\begin{equation*}
\xymatrix@R=1em@C=1em{
&*{\bullet} * \ar@{-}[d]|-{\SelectTips{cm}{}\object@{}}^<{0}^>{} &\\
&*{\circled{$D_4$}}* \ar@{-}[rd]|-{\SelectTips{cm}{}\object@{}}^<{}^>{0} \ar@{-}[ld]|-{\SelectTips{cm}{}\object@{}}^<{}_>{0}&\\
*{\bullet} * &&*{\bullet} * \\
}
\end{equation*}
\caption{  Bipartite dual graph $\Gamma$ associated to a plane cubic with a $D_4$-singularity}\label{F:D4}
\end{figure}
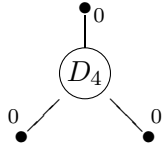

\begin{exa}\label{exa:singularquintic1}
Suppose that $D=Q\cup L$ is a plane quintic comprised of an irreducible plane quartic $Q$ and a line $L$.  Assume that $Q$ has exactly two singularities, one of which is a node (i.e.,~an
$A_1$-singularity of $Q$), and the other of which is a cusp (an $A_2$-singularity); the genus of the normalization of $Q$ is then $1$.  Assume that $L$ meets $Q$ transversally at two smooth points of $Q$ (thus each of these two intersections is an $A_1$-singularity of $D$), and also meets $Q$ at its node ``transversally'', so as to form a  $D_4$-singularity of $D$ at that point.   Then the bipartite dual graph of $D$ is
given in Figure \ref{F:PQuin}, below.
\end{exa}

\begin{figure}[htb]
\begin{equation*}
\xymatrix@R=1em@C=3em{
&&*{\circled{$A_1$}}*&\\
 *{\circled{$A_2$}}*   \ar@{-}[r] &*{\bullet} * \ar @{-}@/^1pc/[ru]|-{\SelectTips{cm}{}\object@{}}^<{1}^>{}    \ar @{-}@/_1pc/[rd]    \ar @{-}@/^1.5pc/[r]    \ar @{-}@/_1.5pc/[r]  &*{\circled{$D_4$}}* & *{\bullet } *    \ar @{-}@/^.7pc/[ld]  \ar @{-}@/_.7pc/[lu]|-{\SelectTips{cm}{}\object@{}}_<{0}^>{}   \ar@{-}[l]|-{\SelectTips{cm}{}\object@{}}_<{}^>{}  \\
 &&*{\circled{$A_1$}}* &\\
}
\end{equation*}

\caption{  Bipartite dual graph $\Gamma$ associated to a plane quintic} \label{F:PQuin}
\end{figure}
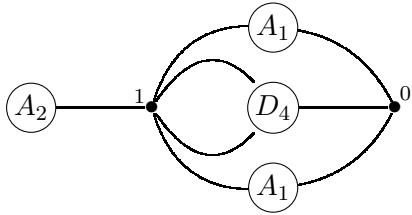

\begin{rem}
Associated to a singularity is a genus (this is related to stable reduction, discussed below).    In this way, we can associate to each vertex $v$ of the graph a genus $g(v)$, whether or not the vertex is black or white.  The genus of the graph is then defined to be $g(\Gamma):=b_1(\Gamma)+\sum_{v\in V(\Gamma)}g(v)$ (which is equal to the arithmetic genus of any stable reduction of the curve with such a bipartite dual graph).
The genus of an $A_1$-singularity is $0$, the genus of an $A_2$-singularity, and of a $D_4$-singularity, is $1$.    Therefore, in Example \ref{E:D4graph} the genus is $((0+0+0)+1)+(3-4+1)=1$ (the arithmetic genus of the plane cubic), and
 in Example \ref{exa:singularquintic1} the genus is $((1+0)+(1+1+0+0))+(8-6+1)=6$ (the arithmetic genus of the  plane quintic).
\end{rem}

We are interested in \'etale double covers $\widetilde C \to C$, where C is a plane quintic. This means that over each singularity of $C$ lie two singularities of
$\widetilde C$ of the same type. In terms of bipartite graphs this means that over each white vertex of the bipartite dual graph $\Gamma$ of $C$ lie two identical
white vertices of the bipartite dual graph $\widetilde \Gamma$ of $\widetilde C$.
Note, however, that it may well happen that a black vertex of $\Gamma$ is only covered by one black vertex of $\widetilde \Gamma$.
This leads us to:

\begin{dfn}\label{def:singularetale} We define
\begin{itemize}
\item[(1)] A double cover $\widetilde C \to C$ is called {\em \'etale over the singularities} if every singularity of $C$ is covered trivially by two copies of the same singularity, contained in $\widetilde C$.
\item[(2)] An involution $\iota$ of a labeled bipartite graph $\widetilde\Gamma$  is called {\em \'etale over the singularities} if
the involution preserves the colors and labels of the vertices, and does not fix any white vertex.
\end{itemize}
\end{dfn}

Starting with a singular curve  $C$ (a plane quintic in our case) we want to replace it by a stable reduction.
In terms of graphs this means that we replace each white node by the graph corresponding to the tail of a stable reduction.
Note that the type of the singularity does not determine the graph of the tail uniquely as the tail itself can degenerate.

Heuristically, the white vertices in the bipartite graph, drawn as they are as large circles, are meant to stand in for an unknown,  more complicated graph, the graph of the stable reduction, about which the only thing we know are its genus, and the collection of edges meeting the rest of the graph.

\begin{dfn}
Given a bipartite dual graph $\Gamma$ associated to a curve $C$ with isolated singularities, a (usual, not bipartite) graph $\Gamma_\real$  obtained from $\Gamma$ by
replacing all  white vertices by graphs of tails associated  to some stable reduction of the corresponding singularity, will be called a    {\em realization} of the bipartite graph $\Gamma$.
\end{dfn}

Let us explain in detail what this means for the $AD$-singularities which we encounter.
For further details we refer the reader to \cite{has}. The case
of an $A_1$-singularity is somewhat special: they are simply removed from the graph, and the two edges going to them are joined. In particular if an $A_1$ is connected to two different black vertices, then we simply connect these two vertices by an edge, while if an $A_1$ is connected by two edges to the same black vertex, we replace it by a loop attached to this black vertex.

Next we consider the  $A_{2k+1}$-singularities, $k\geq 1$. These  have two local branches and we replace the white vertex of~$\Gamma$ by the (usual) dual graph of a stable tail $T\in \overline{\calH}_{k,2}$, where from now on we denote by $\overline{\calH}_{g,n}$
the Deligne--Mumford compactification of the moduli space of hyperelliptic curves of genus $g$ with $n$ marked points.
More precisely, $T$ is a (possibly degenerate) hyperelliptic curve of genus $k$ with two marked
points $p_1$ and $p_2$ which are conjugate under the hyperelliptic involution.  We use the two marked points to attach $T$ to the two branches on which the singularity lies.
On the other hand the $A_{2k}$-singularities are unibranched, and the corresponding white vertices of~$\Gamma$ are replaced with dual graphs of stable tails $T\in \overline{\calH}_{k,1}$
where~$T$ is again hyperelliptic, attached to $C$ in a Weierstrass point.

Now we move to the $D_k$-singularities. The tail of a $D_4$-singularity is a (possibly degenerate) elliptic curve.
Similarly the $D_{2k+4}$-singularities for $k \geq 1$ have $3$ branches and the tail has genus $k+1$, more precisely
$T\in \overline{\calH}_{k+1,3}$ is hyperelliptic and the three marked points
$p_1,p_2.p_3$ have the property that two of them are conjugate under the hyperelliptic involution, while the third is general.
Finally  the $D_{2k+3}$-singularities for $k \geq 1$ have $2$ branches
and the tail $T\in \overline{\calH}_{k+1,2}$ is hyperelliptic. Of the two marked points $p_1,p_2$, one is a Weierstrass point and the other is general.

If we apply this procedure to the plane  cubic in Example \ref{E:D4graph}, then we have to replace the $D_4$-singularity by
a (possibly degenerate) elliptic curve. Let us, say, replace the  $D_4$-singularity by a smooth elliptic curve,
then we obtain the graph in Figure \ref{F:D4r}.

\begin{figure}[htb]
\begin{equation*}
\xymatrix@R=1em@C=1em{
&*{\bullet} * \ar@{-}[d]|-{\SelectTips{cm}{}\object@{}}^<{0}^>{1} &\\
&*{ \bullet}* \ar@{-}[rd]|-{\SelectTips{cm}{}\object@{}}^<{}^>{0} \ar@{-}[ld]|-{\SelectTips{cm}{}\object@{}}^<{}_>{0}  &\\
*{\bullet} * &&*{\bullet} * \\
}
\end{equation*}
\caption{ A realization $\Gamma_\real$ of the bipartite dual graph of the $D_4$ plane cubic}\label{F:D4r}
\end{figure}

If we apply this procedure to the plane quintic in Example \ref{exa:singularquintic1} then we have to replace both the $A_2$ and the $D_4$-singularity by
a (possibly degenerate) elliptic curve. Let us, say, replace the $A_2$-singularity by a nodal elliptic curve and the $D_4$-singularity by a smooth elliptic curve,
then we obtain the graph in Figure \ref{F:PQuinr}.
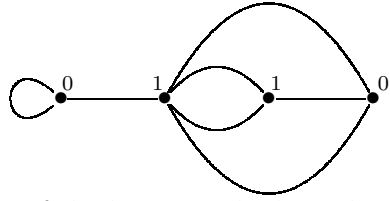
\begin{figure}[htb]
\begin{equation*}
\xymatrix@R=1em@C=3em{
&&&\\
 *{\bullet}*    \ar@{-}@(lu,ld) \ar@{-}[r]|-{\SelectTips{cm}{}\object@{}}^<{0}^>{1} &*{\bullet} * \ar @{-}@/^3pc/[rr]|-{\SelectTips{cm}{}\object@{}}^<{}^>{0}    \ar @{-}@/_3pc/[rr]    \ar @{-}@/^1pc/[r]    \ar @{-}@/_1pc/[r]  &*{\bullet}* & *{\bullet } *      \ar@{-}[l]|-{\SelectTips{cm}{}\object@{}}^<{}_>{1}  \\
 &&&\\
}
\end{equation*}

\caption{ A realization $\Gamma_\real$ of the bipartite dual graph associated to the  plane quintic} \label{F:PQuinr}
\end{figure}

If we apply this realization procedure to the dual graph of   a curve $C$, then  we obtain the dual graph of a nodal, but not necessarily stable curve $C_\real $.   Indeed,  it can happen that there are irreducible components of~$C_\real$ that are smooth rational curves which are only attached by one or two points.
If this is the case, we  then contract these irreducible components of $C_\real$ and obtain a stable curve  $C_{\stab}$ whose dual graph we denote by $\Gamma_{\stab}$. The above process is called {\em stabilization}.  At the level of graphs, the graph map $\Gamma_\real\to \Gamma_\stab$ is obtained by iteratively contracting all edges that connect to a genus $0$ vertex of valency $1$, and for every genus $0$ vertex of valency or $2$, contracting one of the two edges emanating from it (and so mapping such a genus 0 vertex to the vertex at the other end of that edge).
If we apply this procedure to the plane cubic in Example~\ref{E:D4graph}, we obtain the graph in Figure~\ref{F:D4s}.

\begin{figure}[htb]
\begin{equation*}
\xymatrix@R=1em@C=1em{
*{\ \  \bullet\   {}^1 }*
}
\end{equation*}
\caption{A stabilization $\Gamma_\stab$ of the bipartite dual graph of the $D_4$ plane cubic}\label{F:D4s}
\end{figure}

In the case of the plane quintic in Example~\ref{exa:singularquintic1}, the chosen realization of the curve was already stable.   To show that $\Gamma_\real$ and $\Gamma_\stab$ can differ for plane quintics, we include the following example.

\begin{exa}
Consider  a quintic curve
$D=C \cup M$ where $C$ is a smooth plane quartic and $M$ is a line meeting $C$ with multiplicity $4$, thus forming  an $A_7$-singularity.
This gives the bipartite graph $\Gamma$ with two black vertices and one white vertex, labeled with an  $A_7$.
Replacing the $A_7$-singularity by a smooth genus $3$ tail we obtain the graph $\Gamma_{\real}$, with three vertices and two edges.  We note that the curve $D_{\real}$ is not stable. Contracting the smooth rational curve $M$ we obtain a stable curve $D_{\stab}$ with dual graph $\Gamma_{\stab}$, consisting of two vertices and one edge.
\end{exa}

Starting from Proposition~\ref{pro:FScombinatorics} for the cover $\widetilde\Gamma_\stab\to\Gamma_\stab$, we will now deduce necessary conditions for $\widetilde\Gamma_\real\to\Gamma_\real$, and then for $\widetilde\Gamma\to\Gamma$, to admit a stabilization lying in the closure of a Friedman--Smith locus. The first easy step is the following:
\begin{lem}\label{L:realFS}
Let  $\widetilde\Gamma\to\Gamma$ be a degree $2$ map of bipartite graphs which is \'etale over the  white vertices.  Assume that there exists a realization $\widetilde\Gamma_\real\to\Gamma_\real$ whose stabilization $\widetilde\Gamma_\stab\to\Gamma_\stab$ gives  the dual graph of a double cover that lies in $\overline{FS}_n$ for some $n\ge 2$. Then the double cover of graphs $\widetilde\Gamma_\real\to\Gamma_\real$ satisfies the conclusion of Proposition~\ref{pro:FScombinatorics}.
\end{lem}
We note than unlike Proposition~\ref{pro:FScombinatorics}, which gives a combinatorial condition equivalent to an admissible double cover lying in $\overline{FS}_n$, here, since we are working with realizations that need not be stable,  we only claim that these are necessary conditions for a realization to have a stabilization in  $\overline{FS}_n$ --- which will be enough for our purposes of showing that stabilizations of realizations of double covers of plane quintics are never contained in the Friedman-Smith loci.

\begin{proof}
Indeed, as discussed above, the graph $\Gamma_\stab$ (resp.~$\widetilde\Gamma_\stab$) is obtained from the graph $\Gamma_\real$ (resp.~$\widetilde\Gamma_\real$) by contracting the edges
that connect to a genus $0$ vertex of valency $1$, and by contracting one of the two edges at each valency $2$ genus $0$ vertex.
In the latter case, we will always choose, arbitrarily, which edge to contract on $\Gamma_\real$, and will then contract the two preimages of that edge in $\widetilde\Gamma_\real$ correspondingly.
By this we mean that the resulting
graph map $\widetilde\Gamma_\real\to\widetilde\Gamma_\stab$ is then equivariant with respect to the covering involution.

Let $\widetilde V_\stab=\widetilde V_{\stab,1}\sqcup\widetilde V_{\stab,2}$ be the decomposition of $V(\widetilde\Gamma_\stab)$ provided by Proposition~\ref{pro:FScombinatorics}. We claim  that the preimages of $\widetilde V_{\stab,i}$ in $V(\widetilde \Gamma_\real)$ provide a decomposition of $V(\widetilde\Gamma_\real)$ satisfying the same properties.

Indeed, it is clear that these preimages are disjoint,
and also  invariant under the involution of $\widetilde\Gamma_\real$, so (1) holds; the preimage of a connected graph under a map that contracts some edges is connected, so (2) holds. Finally, the graph map $\widetilde\Gamma_\real\to\widetilde\Gamma_\stab$ is an isomorphism on the preimage of every open edge (minus the endpoints), and thus the preimages of the $2n$ edges are still $2n$ edges, so property (3) holds. 
\end{proof}
Our next goal is to obtain a necessary combinatorial property of the double cover of bipartite dual graphs $\widetilde\Gamma\to\Gamma$ that would be implied by some stabilization of some realization lying in~$\overline{FS}_n$. Thus the main result about double covers of bipartite graphs is the following.

\begin{cor}\label{cor:bipartiteFS}
Let  $\widetilde\Gamma\to\Gamma$ be a degree $2$ map of bipartite graphs which is \'etale over the white vertices.
If there exists a realization such that its stabilization is contained in $\overline{FS}_n$ for some $n\ge 2$, then the vertices  of~$\widetilde\Gamma$ can be decomposed as $V(\widetilde\Gamma)=\widetilde V_1\cup \widetilde V_2$, where
\begin{enumerate}
\item $\widetilde V_1$ and $\widetilde V_2$ are both non-empty;
\item $\widetilde V_1$ and $\widetilde V_2$ are both invariant under the involution;
\item the intersection $\widetilde V_1\cap \widetilde V_2$ contains no black vertices;
\item denoting by  $\widetilde \Gamma_i$  the subgraph of $\widetilde\Gamma$ consisting of the vertices of $\widetilde V_i$ together with  all edges such that both their endpoints lie in $\widetilde V_i$, we have that the subgraphs $\widetilde \Gamma_1$ and $\widetilde \Gamma_2$ are both connected.
\end{enumerate}
\end{cor}
\begin{rem}
Of course the approach to proving the corollary is by looking at the decomposition of $V(\widetilde\Gamma_\real)$ provided by Lemma~\ref{L:realFS}, and then transferring it to the vertices of $\widetilde\Gamma$.
Note, however, that in general the surjective morphism of curves $C_\real\to C$ does not give rise to a surjective graph map from $\Gamma_\real$ to $\Gamma$. Indeed, any white vertex of~$\Gamma$ gets replaced in $\Gamma_\real$ by a certain subgraph, except that any white vertex of~$\Gamma$ that corresponds to an $A_1$-singularity is simply erased in~$\Gamma_\real$, with the corresponding edges joined. Thus the $A_1$ white vertices of $\Gamma$ disappear on $\Gamma_\real$, and so there is no natural way to define a graph map from $\Gamma_\real$ to $\Gamma$.
\end{rem}
We are now ready to prove the corollary, while dealing carefully with the issue just described.

\begin{proof}[Proof of Corollary~\ref{cor:bipartiteFS}]
Let $V(\widetilde\Gamma_\real)=\widetilde V_{1,\real}\sqcup \widetilde V_{2,\real}$ be the decomposition of the vertices of $\widetilde\Gamma_\real$  provided by  Lemma \ref{L:realFS}, and let $\widetilde\Gamma_{i,\real}$ be the corresponding graphs consisting of all edges such that both of their endpoints lie in $\widetilde V_{i,\real}$, for $i=1,2$; note in particular that by property (2) each of these two graphs is  connected. Though, as discussed in the remark, there is in general no graph map from $\widetilde\Gamma_\real$ to $\widetilde\Gamma$, for every vertex of $\widetilde\Gamma_\real$, its image in $V(\widetilde\Gamma)$ is well-defined: if this vertex corresponds to some component of the stabilization of a singularity, then the image is the white vertex corresponding to that singularity, and if this vertex corresponds to an irreducible component of $\widetilde C$, there is also a corresponding black vertex in $\widetilde\Gamma$. We denote this map $p_\real:V(\widetilde\Gamma_\real)\to V(\widetilde\Gamma)$, and note that every black vertex has precisely one preimage under this map.

To define the sets $\widetilde V_1,\widetilde V_2$, it is natural to first start with ${\widetilde W}_i:=p_\real(\widetilde V_{i,\real})$ for $i=1,2$. However, by the remark above, the image $p_\real(V(\widetilde\Gamma_\real))$ is the complement in $V(\widetilde\Gamma)$  of the set of $A_1$ white vertices, and thus to decompose $V(\widetilde\Gamma)$ we need to decide where the $A_1$ white vertices will go. We let $\widetilde V_1$ be the union of ${\widetilde W}_1$ and all the $A_1$ white vertices of $\widetilde\Gamma$ such that {\em both} endpoints of the two edges emanating from them are contained in ${\widetilde W}_1$. We let $\widetilde V_2$ be the union of ${\widetilde W}_2$ and all the other $A_1$ white vertices --- i.e., of all $A_1$ vertices of $\widetilde\Gamma$ that are connected by an edge to at least one vertex in ${\widetilde W}_2$.

Since $\widetilde V_{i,\real}$ are both non-empty, it follows that ${\widetilde W}_i$, and thus also $\widetilde V_i$, are both non-empty. Since $\widetilde V_{i,\real}$ are both invariant under the involution, so are ${\widetilde W}_i$, and since our treatment of $A_1$ vertices is also invariant under the involution, it follows that both $\widetilde V_i$ are invariant under the involution. Since every black vertex has precisely one preimage under $p_\real$, no black vertex is contained in $\widetilde V_1\cap\widetilde V_2$, but we note that the intersection ${\widetilde W}_1\cap {\widetilde W}_2$, and thus also $\widetilde V_1\cap\widetilde V_2$, may be non-empty, if the vertices of some stable tail in $\widetilde\Gamma_\real$ are split between $\widetilde V_{1,\real}$ and $\widetilde V_{2,\real}$. Thus properties (1)-(3) are proven, and it remains to prove (4)

For (4) recall that $\widetilde\Gamma_i$ is defined to be the subgraph of $\widetilde\Gamma$ consisting of all edges such that both of their ends lie in $\widetilde V_i$. If we had a graph map from $\widetilde\Gamma_\real$ to $\widetilde\Gamma$,  we would have a surjective graph map from $\widetilde \Gamma_{i,\real}$ to $\widetilde \Gamma_i$, and (4), i.e.,  the connectedness of $\widetilde \Gamma_i$,   would follow immediately from the connectedness of $\widetilde\Gamma_{i,\real}$, since the image of a connected set under a continuous map is connected;
however,  due to the presence of $A_1$ white vertices in $\widetilde\Gamma$, we have to be more careful.
The basic idea to circumvent this issue is to simply replace $\widetilde \Gamma$ with a graph obtained by contracting, for each $A_1$ vertex of $\widetilde \Gamma$, one of the two edges emanating from it.
We now make this more precise.

For $i=1,2$ consider a contraction of the graph $\widetilde\Gamma_i$ obtained by taking each of its $A_1$ white vertices and contracting one of the edges adjacent to it, and denote $\widetilde F_i$ the graph thus obtained, so that there exists a graph map $\widetilde\Gamma_i\to  \widetilde F_i$; note that there is a natural bijection of sets $V(\widetilde F_i)$ and ${\widetilde W}_i$.
We claim that there also exists a graph map $\widetilde \Gamma_{i,\real}\to \widetilde F_i$ extending the map on vertices $p_\real:\widetilde V_{i,\real}\to {\widetilde W}_i$. Indeed, as discussed above, the only obstruction to having a {\em graph map} $\widetilde\Gamma_\real\to\widetilde\Gamma$
was the (possible) existence of  $A_1$ white vertices of $\widetilde\Gamma$. In particular, for any edge of $\widetilde\Gamma_\real$ such that at least one the vertices at its endpoints maps to a white vertex in $\widetilde\Gamma$, the image of this edge in $\widetilde\Gamma$ is naturally defined. It thus remains to define the graph map on those edges of $\widetilde\Gamma_{i,\real}$ such that both of their endpoints map to black vertices of $\widetilde\Gamma$ --- which is precisely to say that the two corresponding irreducible components of $\widetilde C$ intersect transversely, forming an $A_1$ singularity.

Let $e$ be such an edge of $\widetilde\Gamma_{i,\real}$ connecting vertices $v_1,v_2\in V(\widetilde\Gamma_{i,\real})$, for some $i=1,2$. Then by definition $p_\real(v_1),p_\real(v_2)\in {\widetilde W}_i$, and thus by definition the $A_1$ white vertex $w\in V(\widetilde\Gamma)$ corresponding to the edge $e$ is contained in $\widetilde V_i$. Thus one of the edges of $\widetilde\Gamma_i$ connecting $v_1$ to $w$, or connecting $v_2$ to $w$, is contracted in $\widetilde F_i$, and then we can define the image of the edge $e$ to be the remaining edge.

Thus $\widetilde F_i$ can be obtained from the connected graph $\widetilde\Gamma_{i,\real}$ as the image under a graph map, and thus $\widetilde F_i$ is connected. Since $\widetilde F_i$ is also obtained from $\widetilde\Gamma_i$ under a map contracting some edges, it follows that $\widetilde\Gamma_i$ is also connected, which is the desired property (4).
\end{proof}
We note, moreover, that condition (2) in the corollary, saying that the $\widetilde V_i$  are invariant under the involution, is equivalent to saying that the $\widetilde V_i$ are preimages of a decomposition $V(\Gamma)=V_1\cup V_2$.

\subsection{Stable reductions of bipartite graphs for \'etale double covers of plane quintics with $AD$-singularities}
Now if we start not only with a curve $C$ but also a double cover $\widetilde C \to C$ which is \'etale over $C$ (including the singularities), then the above process gives an \'etale double cover  $\widetilde C_{\real} \to C_{\real}$, and an  \'etale double cover  $\widetilde C_{\stab} \to C_{\stab}$ that is, in particular, an admissible double cover of stable curves. The main result of this section is to prove the following:

\begin{pro}\label{prop:nonFS}
Let $\widetilde D\to D$ be an \'etale (also above the singularities) double cover of a plane quintic with $AD$-singularities. Then the stabilization $\widetilde D_{\stab} \to D_{\stab}$ of any realization $\widetilde D_\real\to D_\real$ of $\widetilde D\to D$ is {\em not} contained in the closure  $\overline{FS}_n$ of any Friedman--Smith locus for any $n\geq 2$.
\end{pro}

Before we give the proof, we make a further definition that will be used in a simple but crucial lemma.
\begin{dfn}
A double cover $\widetilde \Gamma \to \Gamma$ of bipartite graphs is called a {\em properly \'etale degree 2} map of bipartite graphs if it is \'etale over all vertices (both black and white).
\end{dfn}

\begin{lem}\label{L:Tree}
Let $\Gamma$ be a bipartite graph that is a disjoint union of $k$ trees. Then for any properly \'etale degree $2$ map $\widetilde\Gamma\to\Gamma$ of bipartite graphs, $\widetilde \Gamma$ is the disjoint union of $2k$ trees.
\end{lem}
\begin{proof}
This is clear since a tree is simply connected. More precisely, since $h^0(\widetilde \Gamma) - h^1(\widetilde \Gamma) = 2( h^0(\Gamma) - h^1(\Gamma)) =2k$, we must have $h^0(\widetilde \Gamma) \geq 2k$. But since the
map is $2:1$, and connected components are mapped to connected components, it follows that $\widetilde \Gamma$ cannot have more than $2k$ connected components, and hence we must have $h^0(\widetilde\Gamma)=2k$, which then implies that $h^1(\widetilde\Gamma)=0$.
\end{proof}

We can now give the proof of  Proposition \ref{prop:nonFS}.
\begin{proof}
The proof uses in many crucial ways the fact that $D$ is a plane quintic curve. First of all, every irreducible component of $D$ is then itself a plane curve of degree at most 5. We distinguish two cases: when $D$ has an irreducible component of degree at least 3, and when all irreducible components of $D$ are plane curves of degrees 1 and 2 only. We crucially note that the only \'etale double cover of a genus 0 curve is the disconnected trivial double cover, and thus an \'etale double cover $\widetilde D\to D$ must be trivial over any degree 1 or 2 irreducible component of $D$.

\smallskip
If $D$ has an irreducible component $E$ that is a plane curve of degree at least 3, then such a component is unique, since the sum of degrees of all irreducible components of $D$ is equal to 5. Let $v$ be the black vertex of the bipartite dual graph $\Gamma$ of $D$ that corresponds to $E$. Let $\Gamma'$ be the possibly disconnected bipartite dual graph obtained from $\Gamma$ by removing the vertex $v$ and all the edges adjacent to $v$. Since the degree of $E$ is at least 3, while the degree of $D$ is 5, the sum of the degrees of the plane curves corresponding to black vertices of $\Gamma'$ is at most 2. Thus $\Gamma'$ has at most two black vertices, and since each of them corresponds to an irreducible plane curve of degree 1 or 2, each irreducible component of $D\setminus E$ is necessarily smooth.

If $\Gamma'$ has no black vertices, then $\Gamma'$ is the disjoint (possibly empty) union of a number of white vertices, and has no edges.

If $\Gamma'$ has only one black vertex $v'$, then $v'$ corresponds to an irreducible plane curve $E'$ of degree 1 or 2, which is thus necessarily smooth. Then for any singularity of $D$, at most one branch of the normalization of this singularity is contained in $E'$. Then $\Gamma'$ is the disjoint union of a bipartite graph that has one black vertex $v'$, connected by single edges to a number (possibly zero) of white vertices, union a number (possibly zero) of isolated white vertices.

If $\Gamma'$ has two black vertices $v_1$ and $v_2$ corresponding to irreducible plane curves $E_1$ and $E_2$, then both $E_1$ and $E_2$ must be plane curves of degree 1. Since two lines intersect in precisely one point, $\Gamma'$ has precisely one white vertex connected to both $v_1$ and $v_2$, with the type of singularity being whatever the union of $E$ and these two lines is at this one point. Any other vertex of $\Gamma'$ must then be white, and connected to at most one of the two vertices $v_1$ and $v_2$.

Altogether, we see that in each of these cases $\Gamma'$ is a (possibly empty) union of disjoint bipartite trees. Since all black vertices of $\Gamma'$ (if such exist) are plane curves of degree 1 or 2, the cover is \'etale over them, and thus $\widetilde\Gamma'\to\Gamma'$ (where $\widetilde\Gamma'$ denotes the preimage of $\Gamma'$ in $\widetilde\Gamma$) is a properly \'etale degree 2 map of bipartite graphs. Thus by Lemma~\ref{L:Tree}
the graph
$\widetilde\Gamma'$ is also a (possibly empty) union of disjoint trees.

Suppose now for contradiction that there exists a realization of $\widetilde\Gamma\to\Gamma$ whose stabilization lies in $\overline{FS}_n$ for some $n\ge 2$. The vertices of $\widetilde\Gamma$ can be decomposed following Corollary~\ref{cor:bipartiteFS} as $\widetilde V_1\cup\widetilde V_2$. By part (2) of the corollary, both sets $\widetilde V_i$ are invariant under the involution. Thus they are preimages of a decomposition $V_1\cup V_2$ of the vertices of $\Gamma$. Assume without loss of generality that $v\in V_1$. Since $v$ is a black vertex, it follows from part (3) of the corollary that $v\not\in V_2$.
Since, by definition, $\Gamma'$ contains all vertices of $\Gamma$ apart from $v$,  this implies that  $\Gamma_2\subset\Gamma'$.
From the above argument, $\Gamma'$ is a disjoint union of trees, and thus its subgraph $\Gamma_2$ is also a disjoint union of $k$ trees, where we know $k>0$, since $V_2$ is non-empty, by part (1) of Corollary~\ref{cor:bipartiteFS}. From Lemma~\ref{L:Tree} it follows that $\widetilde\Gamma_2$ is then the union of $2k>0$ disjoint trees. In particular, $\widetilde\Gamma_2$ is then disconnected, which contradicts property (4) of Corollary \ref{cor:bipartiteFS}. Thus for the case when $D$ has an irreducible component of degree at least 3 we have arrived
at a contradiction with the assumption that the stabilization of some realization of $\widetilde\Gamma\to \Gamma$ lies in $\overline{FS}_n$.

\smallskip
For the case when all irreducible components of $D$ have degree 1 or 2, note first that in this case the cover $\widetilde D\to D$ is \'etale over every irreducible component of $D$. Thus $\widetilde\Gamma\to\Gamma$ is a properly \'etale degree 2 map of bipartite graphs. Assume again for the sake of  contradiction that the stabilization of some realization of $\widetilde \Gamma\to \Gamma$ lies in $\overline{FS}_n$ for $n\ge 2$. Let then $\widetilde V_1\cup \widetilde V_2$ be the decomposition of $V(\widetilde\Gamma)$ as provided by Corollary~\ref{cor:bipartiteFS}, which is the preimage of a decomposition $V(\Gamma)=V_1\cup V_2$. Since the sum of the degrees of plane curves corresponding to all black vertices of $\Gamma$ is equal to 5, and every black vertex is contained in precisely one of $V_1$ and $V_2$, we can assume without loss of generality that the sum of the degrees of all the irreducible components of $D$ corresponding to the black vertices contained in $V_2$ is at most 2. We can thus apply the argument used in the previous case for $\Gamma'$ to the graph $\Gamma_2$ consisting of edges with both endpoints in $V_2$. To begin, the same argument shows that $\Gamma_2$ is a non-empty (since $V_2\ne\emptyset$) disjoint union of trees. But then again by Lemma~\ref{L:Tree} the bipartite graph $\widetilde\Gamma_2$ is a union of $2k>0$ disjoint trees, and thus is not connected. We have arrived at a contradiction with property (4) of Corollary~\ref{cor:bipartiteFS}.

\smallskip
Thus finally in both cases, of whether or not $D$ contains an irreducible component of degree at least 3, starting with the existence of a realization whose stabilization lies  in $\overline{FS}_n$ for $n\ge 2$, we have arrived at a contradiction. Thus for any \'etale over the singularities double cover of a plane quintic with AD-singularities, the stabilization of any realization of it does not lie in any $\overline{FS}_n$ for any $n\ge 2$.
\end{proof}

\begin{cor}\label{cor:ExtIsol}
The intermediate Jacobian map ${\widetilde {IJ}}^{V}$ extends over the locus of cubics with allowable singularities to a morphism
${\widetilde {IJ}}^V: (\tM \setminus\tDh) \to \Vor[5]$.
\end{cor}

\begin{proof}
This is now the consequence of three reduction steps. The first reduction step was achieved in  Corollary \ref{C:IJtoP}:
in order to prove that ${\widetilde {IJ}}^{V}$ extends near a point in $\widetilde {\mathcal M}$ that lies in the preimage of a point in $\overline {\mathcal M}$ corresponding to a polystable cubic $X$ with isolated singularities,
it is enough to consider  corresponding  $1$-parameter degenerations to a pair $(X,\ell)$ where $\ell \subset X$ is a non-special line and to show that the Prym map $\overline R_6\dashrightarrow \Vor[5]$ is regular  in the neighborhood of the admissible cover  given by the stable reduction of
the \'etale double cover of the discriminant plane quintic arising from projection from $\ell$. For the second reduction step we use  \cite[Thm. 3.2 (1), (4)]{abh}, in conjunction with \cite[Thm.~0.1]{vologodsky}, see also
\cite[Thm.~5.6]{cmghl}: the indeterminacy locus of the Prym $\overline R_{g+1}\dashrightarrow \Vor$ is the union of the closures of the Friedman--Smith loci~$FS_n$, for all~$n\geq 2$. Finally, the third reduction
step is now Proposition \ref{prop:nonFS}:  the Prym variety of any \'etale double cover of a plane quintic with $AD$-singularities is never in the closure of  the union of the loci $\overline{FS}_n, n\geq 2$.
\end{proof}


\section{Extensions of the period map along the chordal cubic locus}\label{sec:chordal}

In this section, we discuss the extension of the period map over the chordal divisor $\tDh\subset \tM$.
This will follow rather easily from general principles and the following facts: (1) extension away from the chordal divisor,
(2) the chordal divisor generically parameterizes hyperelliptic genus $5$ curves (thus giving finite monodromy),
and (3) the chordal divisor meets the other boundary divisors in $\tM$ transversally.

We start by recalling that at the level of $\tM$ the chordal divisor $\tDh$ can be identified on the one hand with the strict transform of the  exceptional divisor $E$ (see Thm. \ref{resgitball}(3)) of the Kirwan blow-up
$\widehat{\mathcal M} \to \overline{\mathcal M}$ of the GIT quotient $\overline{\mathcal M}$, and on the other hand with the strict transform of the chordal Heegner divisor $D_h^*$  in the Baily--Borel compactification of the ball quotient model of \cite{act}. Since the ball quotient description plays a fundamental role in this section, we will now use the notation $\hDh$  for the divisor $E$ in $\widehat \calM$ (N.B. $\hDh$ and $\widehat \calM$ and are obtained as the $\QQ$-factorialization of $D_h^* \subset (\calB/\Gamma)^*$). Using the GIT perspective, one sees that $\hDh\subset \widehat{\mathcal M}$ is naturally identified with the GIT quotient for $12$ (unordered) points in $\PP^1$. The ball quotient perspective gives
further structure to  $\hDh$ and to the embedding $\hDh\subset \hM$.  In particular, one sees that  $\hDh$ is (the Baily--Borel compactification of) a $9$-dimensional ball quotient (namely, the open part $D_h=\calD_h/\Gamma$
is a Heegner divisor in $\calB/\Gamma$, which, in particular, gives that $D_h$ is uniformized by a $9$-dimensional complex ball; we write $D_h=\calB'/\Gamma'$ \footnote{Strictly speaking, given an irreducible Heegner divisor $D=\calD/\Gamma\subset \calB/\Gamma$, there is a natural map
$$\calB'/\Gamma'\to D\subset \calB/\Gamma,$$
which is the normalization of $D$; here $\calB'$ is one of the components of $\calD$ (which is a $\Gamma$-invariant arrangement of hyperplanes in $\calB$), and $\Gamma'$ is the normalizer in $\Gamma$ of that component. In our case, due to Proposition \ref{prop:transversal}, $D_h$ is normal (in fact, smooth in a stack sense), thus we can write $D_h=\calB'/\Gamma'$.  }). The fact that $\hDh$ can be identified with the GIT quotient for $12$ points in $\PP^1$ is originally due to Deligne--Mostow \cite{dm}, and discussed at length in \cite[\S4]{act}.

From our point of view the following result from \cite{act}, which says that the chordal divisor meets the discriminant divisor in $\hM$ transversally,
is of special relevance. Here we recall that  the complement of the locus of smooth cubic threefolds $\calM\subset \calB/\Gamma$
in the Allcock--Carlson--Toledo ball quotient model consists of two Heegner divisors $D_h=\cDh/\Gamma$ and $D_n=\cDn/\Gamma$
corresponding to the chordal cubic threefold and the nodal divisor respectively.

\begin{pro}[{\cite[Thm. 7.2]{act}}] \label{prop:transversal}
The following holds:
\begin{itemize}
\item[(1)] No two (distinct) hyperplanes of the chordal hyperplane arrangement $\cDh$ meet in $\calB$.
\item[(2)] If a nodal hyperplane and a chordal hyperplane meet in $\calB$, then they are orthogonal (more precisely their normal vectors are orthogonal).
\end{itemize}
\end{pro}

As discussed in Section \ref{sec:tM}, $\tM$ is obtained from $\hM$ by performing the wonderful blow-up,
which has the effect of making the boundary of $\calM\subset \tM$ normal crossing (up to passing to finite covers). The fact that the chordal divisor is already transversal to the discriminant (or nodal) divisor at the level of $\hM$
 implies that (1) there is no extra blow-up with center contained in  the chordal divisor (although the centers may meet the chordal divisor), and that (2) the wonderful blow-up of $\hM$ restricts to the wonderful blow-up of $\hDh$  (when we regard $\hDh$ as a $9$-dimensional ball quotient). More concretely, $\tM\to\hM$ is obtained by blowing up the closures of the loci of cubics with $A_2,\dots, A_5,D_4$-singularities respectively. The chordal divisor $\hDh$ does not pass through the point corresponding to the $3D_4$ cubic in $\hM$, thus the blow-up leading to the $\widetilde{D}_{D_4}$ divisor will not affect the chordal locus. The remaining blow-ups, associated to the $A_2,\dots,A_5$ loci, will simply restrict to the steps of the wonderful blow-up associated to $\hDh$. By the construction of the wonderful blow-up (we refer the reader to \cite[\S2 and \S3]{cml2} for a more general discussion of wonderful blow-ups for ADE discriminants), and the transversality statement of Proposition \ref{prop:transversal}, all that remains to be noted is that the blow-up centers in $\hM$  restrict to the blow-up centers of  $\hDh$.

\begin{pro} The closure of the locus of cubics with
$A_k$-singularities meets the hyperelliptic divisor $\hDh$ in the locus corresponding to where $k+1$ points in $\PP^1$ coincide (under the identification  of $\hDh$  with the GIT quotient for $12$ points in $\PP^1$).
\end{pro}
\begin{proof} This is a simplified version of \cite[Thm. 2.1]{act}, which gives a much more precise statement including the restriction of combinations of $A_k$ singularities. Indeed, for the wonderful blow-up a single $A_k$ singularity suffices. Briefly, using the ball quotient description for both $\hM$ (cf. \cite{act}) and $\hDh$ (cf. \cite{dm}), the essential fact to note is that the nodal divisor $D_n=\calD_n/\Gamma$, parameterizing cubics with a
single $A_1$-singularity (see Theorem \ref{thm_ballmodel}) in $\hM$ restricts to the nodal hyperplane $D_{n}^{(h)}$ in $D_h$ corresponding to $2$ out of the $12$ points in $\PP^1$ coinciding. Then, the $A_k$ locus in $\hM$ corresponds to a codimension $k$ Shimura subvariety in $\hM$ associated to an intersection of type $A_k$ of hyperplanes in the arrangement $\calD_n\subset \calB/\Gamma$ (i.e., $\frac{k(k+1)}{2}$ hyperplanes corresponding to $\pm$-roots in an $A_k$ root system; see \cite[\S2]{cml2} for a full discussion). Each such hyperplane restricts to a hyperplane of the arrangement $\calD_n^{(h)}\subset \calB'/\Gamma'$. Proposition \ref{prop:transversal} guarantees that the resulting sub-arrangement of $\calD_n^{(h)}$ is again of type $A_k$, which geometrically means that $k+1$ of the $12$ points in $\PP^1$ come together.
\end{proof}

\begin{rem}\label{rem12pts}
The following holds generally for the GIT moduli space $\widehat {\mathcal H}_g$  of $2g+2$ (unordered) points in $\PP^1$. By an easy adaptation of results in \cite{cml2}, one sees that locally (up to a finite cover, in analytic or \'etale coordinates) near a point in moduli corresponding to some $k+1$ (of the $2g+2$) points collapsing, the discriminant is of type $A_{k}$ (as before this corresponds to $\frac{k(k+1)}{2}$ hyperplanes meeting according to the incidences of roots in the $A_k$ root system).
Consequently,
the results of \cite{cml2}    say that the wonderful blow-up $\widetilde \calH_g$
 resolves the rational map $\widehat \calH_g \dashrightarrow \overline \calM_{g}$. It is immediate to see (when restricted to the hyperelliptic case) that the wonderful blow-up process is the reverse of the contraction process of \cite{fedhyp}. Indeed, the various steps correspond to $\calH_g[k]$ in the notation of \cite{fedhyp}. Thus, quite generally $\widetilde \calH_g\to \overline \calM_{g}$ is a closed embedding which identifies $\widetilde \calH_g$ with the closure $\overline{\calH}_g$ of the hyperelliptic locus.
\end{rem}

Summarizing the discussion, we get:

\begin{pro}
The following holds:
\begin{itemize}
\item[(1)] $\hDh$ is naturally identified with the GIT quotient for $12$ (unordered) points in $\PP^1$.
\item[(2)] $\widetilde D_h$ is naturally identified with the closure of the hyperelliptic locus in $\bM_5$ (and thus isomorphic to $\bM_{0,12}/\Sigma_{12}$).
 \end{itemize}
\end{pro}
\begin{proof}
The identification of $\hDh$ with the GIT quotient is discussed in \cite{act}, in particular Section 4. Thus, we can regard $\hDh$ as the moduli space of hyperelliptic curves with up to
$A_5$-singularities (in the sense of \cite{fedhyp}). Furthermore, the transversality statement of
the previous Proposition \ref{prop:transversal}
implies that the locus in $\hDh$ corresponding to a configuration $T$ of singularities (of type $A_k$) is in the closure of the locus of cubics with the same configuration $T$ of singularities. The wonderful blow-up $\tM\to \hM$ corresponds to successively blowing up the locus of cubics with $A_5, \dots, A_2$-singularities. Again, by transversality, and the fact that everything is locally modeled by a hyperplane arrangement, we conclude that $\tDh\to \hDh$ is precisely the wonderful blow-up applied to
$\hDh$. At the same time the wonderful blow-up applied to the GIT moduli space for $12$ points in $\PP^1$ leads to the closure of the hyperelliptic locus in
the Deligne--Mumford moduli space $\bM_5$.
\end{proof}

As noted in \cite{act} (also \cite{col}), the limit intermediate Jacobian associated to a generic degeneration to the chordal cubic $X_0$ is a pure Hodge structure. This can be seen
as follows. First recall that $\Sing(X_0)$ is a rational normal curve of degree $4$. if $X_0=V(F_0)$ is the central fiber of a general pencil $X_t=V(F_0 + t F_1)$ then $F_1$ cuts out
$12$ points on  $\Sing(X_0)$ and these define a hyperelliptic curve of genus $5$. The Jacobian of this hyperelliptic curve is then the limit point of the intermediate Jacobians of
this pencil. It follows (see \cite[Sect. 4]{act} for a fuller discussion) that the period map $\calM\to \calA_5$ extends to a regular map at the generic point of $D_h$ (with image in
$\calA_5$, and not in the boundary of the Satake or a toroidal compactification). Note also that due to Proposition \ref{prop:transversal}, the hyperelliptic divisor is smooth (in a stack sense). By a standard result of Griffiths (e.g. \cite[Thm. 13.4.5]{carlson}), it follows that the monodromy around the hyperelliptic divisor is finite.
Thus, the extension of the period map along the chordal divisor will follow immediately from Corollary \ref{cor:finalextension}, which is a consequence of the following proposition.

\begin{pro}[Extension along trivial monodromy divisors]\label{prop:trivialextension}
Let $S$ be a smooth variety, $\overline S$ be a simple normal crossing compactification, and $D\subseteq \overline S\setminus S$ an irreducible boundary divisor. Assume that we are given a period map $P:S\to \mathcal A_g$ (in particular, $P$ is locally liftable to $\mathbb H_g$)
such that the monodromy around each of the boundary divisors is unipotent and such that it is trivial around $D$. Then, $P$ extends to a morphism  $P:\overline{S}\to \overline{\mathcal A}^\Sigma_g$ to a fixed toroidal compactification if and only if it extends to a morphism  $P:\overline{S}\setminus D\to \overline{\mathcal A}^\Sigma_g$.
\end{pro}
\begin{proof}
Under the given assumptions  Borel's extension theorem gives us an extended period map
$$P^*:\overline{S}\to \mathcal A_g^*$$
to the Satake compactification.
The question of lifting $P^*$ to a toroidal compactification is a combinatorial question, which can be phrased as follows: {\it for each $x\in \overline S\setminus S$, let $\sigma_x$ be the associated monodromy cone (see e.g. \cite[\S2.1.2]{cmghl}). Then, the period map $P$ extends at $x$ if and only if there exists a cone $\tau$
in the chosen admissible decomposition $\Sigma$ such that  $\sigma_x\subseteq \tau$.} (We refer the reader to \cite[Sect. 2]{cmghl} for a review of monodromy cones for weight $1$ variations of Hodge structure, and for their relevance to extension questions for period maps.)

Let $x\in D$ and let $D_1,\dots, D_k$ be the other boundary divisors passing through $x$. As discussed in \cite[\S2.1.3]{cmghl}  the closure of the monodromy cone is given by
$\overline{\sigma}_x=\RR_{\ge 0} \langle N_0,N_1,\dots, N_k\rangle\subset \mathfrak{sp}(2g)_{\RR}$,
where $N_0$ is the (log) monodromy around $D$, and $N_i$ the (log) monodromies around $D_i$. (Note that in general $N_i=\log T_i$ are  defined over $\QQ$, but in the weight $1$ case they can be defined over $\ZZ$.) In the situation considered here, we have $N_0=0$, and thus
$$\overline{\sigma}_x=\RR_{\ge 0} \langle N_1,\dots, N_k\rangle$$
giving
an identification of monodromy cones: $\overline{\sigma}_x=\overline{\sigma}_y$ for any $y\in \left((D_1\cap\dots D_k)\setminus D\right)\cap U_x$, where $U_x$ is a small neighborhood around $x$,
(see \cite[Diagram (2.8)]{cmghl} and accompanying discussion for precise statements about the comparison of monodromies under partial smoothings).
Since by assumption $\overline{\sigma}_y \subset \tau$ for a cone $\tau \in \Sigma$, the proof is finished.
\end{proof}

We are now able to conclude the proof of our main result:

\begin{cor}\label{cor:finalextension}
The intermediate Jacobian map ${\widetilde {IJ}}^{V}$ extends along the chordal divisor and hence to a morphism
${\widetilde {IJ}}^V: \tM \to \Vor[5]$.
\end{cor}
\begin{proof}
In the previous sections we have established that the intermediate Jacobian map extends  outside the chordal locus (cf. Corollary \ref{cor:ExtIsol}). The corollary now follows from Proposition \ref{prop:trivialextension} by noting that the  extension is insensitive to finite covers (see Proposition \ref{P:ExtProp}) and thus, we can assume (without loss of generality) that  the monodromy is unipotent as required in the assumptions of
Proposition \ref{prop:trivialextension}.
\end{proof}

\subsection{A geometric observation}  We sketch here another approach to the extension result along $\tDh$.  While this approach is not as efficient as that given above, it provides a very nice geometric picture which we would like to convey.
As discussed, we can view $\tDh$ as the hyperelliptic locus inside $\overline \calM_5$.
Since the Torelli map extends to a morphism $\overline \calM_5\to \Vor[5]$ it  extends in particular naturally along $\tDh$.
The results from the previous sections tell us that we have an extension $\tM\setminus \tDh \to \Vor[5]$.
In general, these two facts together are, of course, not enough to conclude that there is an extension $\tM\to \Vor[5]$.

However in our situation this indeed suffices.
Namely, as before let $x\in \tDh$ and consider, as in the proof above, the monodromy cone $\sigma_x=\langle N_1,\dots,N_k\rangle$ (where $N_i$ are
the monodromies around the other boundary divisors, and $N_0=0$ for the monodromy around the hyperelliptic locus).
There is an (a priori different) monodromy cone $\sigma'_x=\langle N'_1,\dots,N'_k\rangle$ coming from monodromies $N'_i$ by considering loops around the
discriminant but contained inside the hyperelliptic locus. However, it is clear that there are natural identifications
$N_i'=N_i$, leading to an identification $\sigma_x=\sigma_x'$.
This is a consequence of the transversality: consider a
loop $\gamma$ around a boundary divisor $D_i$ inside the hyperelliptic locus (or rather around $D_i\cap \tDh$ inside $\tDh$), this loop can be
moved out of $\tDh$ keeping everything smooth (in the sense of abelian varieties or VHS).
Now, since  we have an extension $\tDh\to\Vor[5]$, it follows that there
is a second Voronoi cone $\tau$ containing $\sigma'_x$, and thus $\sigma_x=\sigma'_x\subseteq \tau$, giving an extension $\tM\to \Vor[5]$
near $x\in \tDh$.

In other words, the two proofs of the extension along the chordal locus (or more generally trivial monodromy divisors)
are perfectly complementary: one says, that if there is an extension away from $\tDh$, there should be an extension also along $\tDh$, while the other
proof says that extension along $\tDh$ implies extension near $\tDh$.
The ingredient for the proof is the same in both cases: the behavior of the monodromy cone. In one case we say that the monodromy cone remains the same if we move
away from $\tDh$, while in the other case we say that the monodromy
 cone stays the same when restricted to $\tDh$.

 Another way to interpret this is the following.  The arrangement of boundary divisors in $\tM$ is locally stratified in a natural way by a canonical log resolution of a hyperplane arrangement of $AD$ root systems.  In the case of cubics with isolated singularities, this arises from the singularities of the cubics. In the case of $\tDh$, this arises naturally from degenerations of the branch locus for genus $5$ hyperelliptic curves, which in turn gives type $A$ arrangements.  All of these arrangements, and corresponding   monodromies are naturally identified, since they are all obtained from versal spaces of the singularities of the corresponding type.  Therefore the monodromy cone for $\tM$ around points in
 $\tDh$, and the monodromy cone for $\tDh$ at the corresponding point, are naturally identified. In particular, this allows us to make the following observation:

\begin{rem}[The monodromy cones are spanned by rank $1$ quadrics]\label{rem_moncone}
We know from the discussion in Section \ref{sec:compare} that in a small analytic neighborhood of a point  $x\in \widetilde \calM\setminus \tDh$, i.e., away from the hyperelliptic divisor, $\widetilde \calM$ can be thought of as parameterizing a family of Pryms. Since the monodromy transformations of Prym varieties are of Picard--Lefschetz type, it follows that the associated monodromy cone at $x$ is spanned by rank $1$ quadrics
(see \cite[\S 4]{cmghl} for further discussion). Now assume that $x\in \tDh$. We first note that $\tDh$ has no self-intersections (Proposition \ref{prop:transversal}). Assume that $x$ lies on finitely many further boundary components $D_1, \ldots, D_k$. We saw
in the proof of Proposition \ref{prop:trivialextension} that the monodromy cone is then spanned by $N_0, N_1, \ldots , N_k$ where $N_0$ comes from the hyperelliptic divisor, and in fact $N_0=0$, and the $N_i, i\geq 1$
are spanned by rank $1$ quadrics (as in the previous case). In either case the image lies in the matroidal locus.
\end{rem}


\section{Torus rank $1$ degenerations and images of boundary divisors of $\tM$}\label{sec:components}
In this section we will  identify the torus rank $1$ images of the boundary divisors of $\tM$  in $\AV$ and in $\AP$, thus proving Theorems \ref{theo:codimension1intro} and
\ref{theo:divisors}.
For this we will work  on Mumford's partial compactification  $\calA_5'$ . Recall that for any genus the partial
compactification  $\calA_g'$  is contained in every toroidal compactification $\calA^{\operatorname{tor}}_g$:
under the natural map
$\varphi: \calA^{\operatorname{tor}}_g \to \Sat$ this is the pre-image $\calA_g'= \varphi^{-1}(\calA_g \cup \calA_{g-1})$.  Since all toroidal compactifications
agree over $\calA_g \cup \calA_{g-1}$ the discussion concerning the partial compactification  is independent of which toroidal compactification we are working with,
in particular it holds for both $\AV$ and $\AP$.

Note that $\calA_g'$ parameterizes semi-abelic varieties of torus rank up to one.
The boundary
of $\calA_g' $, i.e.,~$\varphi^{-1}(\mathcal A_{g-1})$, is the universal Kummer variety
$\partial \calA_g'=\calX_{g-1}/\pm 1$, where $\calX_{g-1}\to\calA_{g-1}$ denotes the universal family of abelian varieties. There exists a universal family over $\calA_g'$, with the
boundary point $(A, \pm b)\in\calX_{g-1}$ (with $A\in\calA_{g-1},  b\in A$) corresponding to the semi-abelic variety $\overline G$ obtained by identifying, with a shift by $b$, the zero and infinity sections of the line bundle over $A$ given by $b$. The open part $G\subset\overline G$ is a group scheme, more precisely a semi-abelian variety, namely the extension
$$
 1\to\CC^*\to G\to A\to 0
$$
given by $b\in \operatorname{Ext}^1(A,\CC^*) \cong A^\vee  = A$, where the principal polarization is used to identify $A$ with its dual $A^\vee$. Note that $\pm b$ leads to  isomorphic degenerations.
This whole discussion has, as always, to be read in a stack sense, i.e.,  when working with concrete families and varieties one typically has to go to a finite cover, due to torsion elements
in the
symplectic group.

In \cite{grhu1} the boundary $\overline{IJ}\cap\partial\AP$ of the locus $IJ$  of intermediate Jacobians was determined, and it was shown that this consists of two irreducible
components of dimension $9$ which were  labelled  $(I)$ and $(III)_2$ there.
To simplify notation we will simply relabel the components $\mathbf A$ and $\mathbf B$ in this paper (and trust that there will be no confusion with an abelian variety $A$).
To describe these,  recall that we  denote by $\calJ_g\subset\calA_g$ the Jacobian locus, and by $\calH_g\subset\calA_g$ the locus of hyperelliptic Jacobians in $\calA_g$. Let $\Theta_A$ be a symmetric  polarization divisor on a ppav $A$ (defined up to translation by a $2$-torsion point), and let $\operatorname{Sing}\Theta_A$ be its singular locus. For a curve $C$ we denote a symmetric
theta divisor of the Jacobian $J(C)$ by $\Theta_C$.

We recall that the computations in \cite{grhu1} come from the well-known Fourier--Jacobi expansion of theta functions near the boundary.  This was applied to compute the boundary of the locus of ppav with a vanishing theta gradient in \cite[Prop.~12]{grsmconjectures}.  At this point we also take the opportunity to correct an unfortunate
typographical error in \cite[Thm.~9.1]{grhu1}, where there is an extra factor of 2 multiplying $z_3$ in the second of the loci below.

\begin{teo}[{\cite[Thm.~9.1]{grhu1}}]\label{theo:codimension1}
The intersection of $\overline{IJ}$ with the boundary of the partial compactification consists of two irreducible components, each of dimension $9$, namely the closures of
\begin{equation}\label{locusI}
 {\bf A}:=\lbrace (A,z_4)\mid A=J(C) \in\calJ_4,z_4\in 2_*\Sing\Theta_C\rbrace \ \ (=(I))
\end{equation}
and
\begin{equation}\label{locusIII2}
 {\bf B}:=\lbrace (A,z_4)\mid A=E\times J(C')\in\calA_1\times \calH_3,
  z_4=(z_1,z_3), z_3\in \Theta_{C'} = C' - C'\rbrace \ \  (=(III_2))
\end{equation}
where $\calJ_4$ denotes the locus of Jacobians of smooth genus 4 curves, $\calH_3$ denotes the hyperelliptic locus in genus $3$ and where we use $z_k$ to denote a point in an abelian variety of dimension $k$.
\end{teo}

Let us make a few comments explaining why the loci above are well defined (independent of choices).
First, since any two symmetric theta divisors differ by a $2$-torsion point and since we are multiplying by $2$, the choice of a symmetric $\Theta_C$ in $\mathbf A$ does not matter. In fact, in the case where $C$ is not hyperelliptic, if  $N$ and $\hat N$ are the (not necessarily distinct) $g^1_3$'s on $C$, then $2_*\operatorname{Sing}\Theta_C=\{(N\otimes \hat N^{-1})^{\otimes \pm1}\}$. For the locus $\mathbf B$, we would like to recall a basic fact about theta divisors of hyperelliptic genus $3$ curves. In general there is no canonical representative of the theta divisor of a curve in its
degree $0$ Jacobian.  For  genus $3$ hyperelliptic curves, however, we can use the hyperelliptic pencil  $\kappa=g^1_2$, which is a distinguished theta-characteristic,  to identify $J^2(C')$ with $J^0(C')=J(C')$. We use this to define a distinguished theta divisor $\Theta_{C'}=W_2 - \kappa$ and note that this is also the same as the difference variety $C' - C'$. In particular $\Theta_{C'}$ is a well defined symmetric theta divisor,  characterized by the fact that its singularity is at the origin.
This also explains the translation from the language of theta functions in \cite[Prop.~12]{grsmconjectures} to the geometric language used here.  Namely, the theta function determined by theta-null  in \cite[Prop.~12]{grsmconjectures} is the one
characterized by the fact that its singularity is at the origin.

Theorem \ref{theo:codimension1} together with the results of \cite{cml} give a quick way to identify the components $ {\bf A}$ and $\mathbf B$  as the images of
boundary divisors in $\tM$.

\begin{teo}[{\cite{cm}}]\label{prop:locusA}
The intermediate Jacobian of a generic cubic threefold with a unique $A_1$-singularity corresponds to a generic point of the locus ${\bf A}$.
\end{teo}

\begin{proof}
Since $\widetilde{IJ}^V:  \tM \to \Vor[5]$ is a morphism, there must be at least one boundary divisor in $\tM$ that is mapped to the closure of the locus ${\bf A}$.
By \cite[Table 1]{cml} the only boundary divisor of $\tM$ which has the property
that the compact part of the degenerate intermediate Jacobian is a general point in $\calJ_4$
is the divisor $\tDA[1]$, which must therefore map to the closure of ${\bf A}$.
\end{proof}

\begin{teo}\label{prop:locusB}
The intermediate Jacobian of a generic cubic threefold with a unique $A_3$-singularity corresponds to a generic point of the locus ${\bf B}$.
\end{teo}
\begin{proof}
As in the previous case the proof is very straightforward: by \cite[Table 1]{cml}  the only divisor which has the property that the compact abelian part of the intermediate Jacobian
of a general point on this divisor is a product of an elliptic curve with a hyperelliptic curve of genus $3$ is $\tDA[3]$.
\end{proof}

We can in fact be more explicit about the identifications in the theorems above.
Since
this also involves some beautiful geometry, we will sketch this in the following subsections, thus also recovering a classical result due to Collino--Murre \cite{cm} (see also \cite{cg}).

\subsection{Cubic threefolds with an $A_1$-singularity and the locus $\mathbf A$} In this section, we prove
Theorem \ref{theo_CollMurre}, below, which  is essentially due to Collino--Murre \cite{cm} (see also \cite{cg}), and expands on Theorem \ref{prop:locusA}.    We explain a proof here that generalizes to semi-stable  cubic threefolds with isolated singularities.

First recall that if $X$ is a cubic threefold with AD-singularities, then projecting from a singularity $x_0$  of $X$ gives a birational  map $X\dashrightarrow \mathbb P^3$.    Blowing-up at $x_0$ we obtain  $f: \widetilde X \to X$ and a morphism
$g: \widetilde X \to \PP^3$ resolving the rational projection map.  The exceptional locus $E$ of the morphism $g$ maps to a curve $\Sigma \subset \PP^3$, which parameterizes the lines contained in $X$  passing through $x_0$. The curve $\Sigma$ lies on a quadric $Q$, namely the projectivized tangent cone of $X$ at the node $x_0$.  It is the complete intersection in $\mathbb P^3$  of the quadric $Q$ with the cubic $X$,  and $\widetilde X$ is the blow-up of $ \mathbb P^3$ along $\Sigma$. We call $\Sigma$ the associated $(2,3)$-curve, and we note that $p_a(\Sigma)=4$.     The singularities of $\widetilde X$ are in 1-1 correspondence with the singularities of $\Sigma$, including the type.
 We can summarize this in the diagram
$$
\xymatrix{
&Q  \ar@{^(->}[r] \ar[ld]&\widetilde X \ar[ld]_f \ar[rd]^g&E \ar@{_(->}[l] \ar[rd]&\\
x_0\ar@{^(->}[r]&X\ar@{-->}[rr]^{\pi_{x_0}}&&\mathbb P^3&\ar@{_(->}[l]\Sigma.\\
}
$$
Note also that given a $(2,3)$-complete intersection curve $\Sigma$ with AD-singularities, there is an associated cubic threefold with AD-singularities and a given singular point making $\Sigma$ the associated $(2,3)$-curve.
For more details, see \cite[\S 3.1]{cml}.

\begin{teo}[{Collino--Murre \cite{cm}}]\label{theo_CollMurre}
Let $X$ be a cubic threefold with a unique $A_1$-singularity. The extended intermediate Jacobian map $\widetilde{IJ}^{V}:\tM \to\Vor[5]$ near $X$ ($\tM$ and $\overline {\mathcal M}$ agree near $X$)
maps $X$ to a torus rank $1$ degeneration given by an extension
$$
1 \to \CC^* \to IJ(X) \to J(\Sigma) \to 0
$$
where $\Sigma$ is the $(2,3)$-curve associated to $X$, a smooth genus 4 curve that is  non-hyperelliptic and  has  no vanishing theta-null,  and the extension datum is given by
$$
(N \otimes \hat N^{-1})^{\otimes \pm 1} \in J(\Sigma)
$$
where $N$ and $\hat N$ are the two $g^1_3$'s on $\Sigma$.  Conversely, given such an extension, there is a cubic threefold $X$  with a unique $A_1$-singularity with $IJ(X)$ identified with the given extension.
\end{teo}

\begin{rem} This follows easily from what we have shown above.
From Theorem \ref{prop:locusA}, the only thing to do is to identify the genus $4$ curve $C$ with $\Sigma$.  But we know that $\widetilde X$ is isomorphic to the blow-up of $\mathbb P^3$ along $\Sigma$, and so $J(\widetilde X)=J(\Sigma)$.   By basic results on degenerations of Hodge structures, the compact part of $IJ(X)$ is identified with $J(\widetilde X)$, and we are done.    Conversely, given such an extension, one has a $(2,3)$-curve, and the associated cubic $X$ has $IJ(X)$ identified with the extension.  The main point in what follows is to utilize the Prym construction to prove the theorem above, in order to illustrate how to extend these results to cubics with more complicated singularities.
\end{rem}

\subsubsection{The Prym construction in the $A_1$-case}
For this we start with a pair $(X,\ell)$ where $X$ is a cubic with exactly one $A_1$-singularity and $\ell$ is a non-special line. Projection from $\ell$ defines a conic bundle with
discriminant curve $D \subset \PP^2$, a quintic with exactly one node $p$. Note that a quintic curve with exactly one singularity, which is of type $A_1$, is necessarily irreducible (see, e.g., \cite[Cor.~3.7]{cml}). We will be in this situation throughout our treatment of the $A_1$ case.
Let $\nu : N(D) \to D$ be the normalization. Then $N(D)$ has genus $5$ and by \cite[p.207]{acgh} it is non-hyperelliptic
with a unique $g^1_3$, which is given by the pencil of lines through the node of $D$.

The conic bundle structure defines an \'etale $2:1$ cover $\pi_D: \widetilde D \to D$ given by a $2$-torsion line bundle $\eta_D$, which is nontrivial (Remark \ref{R:B1}).
Consequently, the cover $\widetilde D\to D$ is nontrivial;  in fact, $\widetilde D$ is irreducible as shown in the following lemma:

\begin{lem}[{\cite{cml}}]\label{lem:irreducible}
Let $(X,\ell)$ be a cubic threefold with a single allowable singularity, together with a non-special line $\ell$. Let $(\widetilde D, D)$ be the \'etale double cover of the discriminant plane quintic $D$ determined from projection from $\ell$. Then both $D$ and $\widetilde D$ are irreducible.
\end{lem}

\begin{proof} This assertion is made in the proof of \cite[Prop.~5.2]{cml}.  We include the details here.
As discussed (Proposition \ref{P:Fano}), the singularities of $D$ are in correspondence with the singularities of $X$. Thus, $D$ has a unique allowable singularity say at $p\in D$ (at worst $A_5$ or $D_4$). An elementary application of Bezout's theorem  allows us to conclude that $D$ is irreducible \cite[Cor.~3.7]{cml}.
As we have seen, the  \'etale double cover $\widetilde D\to D$ arising from projection from $\ell$ is nontrivial (Remark \ref{R:B1}), and consequently, it is immediate to see that $\widetilde D$ is irreducible in the $A_{2k}$ case.  While one can deduce the statement for the $A_{2k+1}$ case via a degeneration argument, we present here a short argument that works in all cases, including $D_4$.

Let $\eta_D$ be the $2$-torsion line bundle associated to the
 \'etale double cover $\widetilde D\to D$, and let $\kappa_D=\eta_D\otimes \mathcal O_D(1)$ be the associated theta characteristic.  Taking normalizations, $\nu:N(D)\to D$ and $\tilde \nu:N(\widetilde D)\to \widetilde D$, we obtain an \'etale double cover $N(\widetilde D)\to N(D)$.  Since $D$ is irreducible, the same is true of $N(D)$, and we would like to establish that $N(\widetilde D)$ is irreducible.  Since $\nu^*\eta_D$ is the $2$-torsion line bundle associated to $N(\widetilde D)\to N(D)$, our goal is to show that $\nu^*\eta_D$ is nontrivial.

 To this end, recall that $\kappa_D=\mathcal O_D(\sum_{i=1}^5p_i)$, where $\sum_{i=1}^52p_i$ is the divisor with support in the smooth locus of $D$, cut by the smooth plane conic arising as the second discriminant (\eqref{E:MatB}, Proposition \ref{P:Fano}, Remark \ref{E:KappaConic}).   In particular, since $h^0(\kappa_D)=1$ (Proposition \ref{P:B1}),  we have that $p_1,\dots,p_5$ are not collinear.
Let $L$ be a general hyperplane section of $D$ (i.e., of $\mathcal O_D(1)$), in particular, not passing through any of the points $p,p_1,\dots,p_5$.  Then $\eta_D=\mathcal O_D(L-\sum_{i=1}^5p_i)$, and we have $\nu^*\eta_D\cong \mathcal O_{N(D)}$
if and only if $\sum_{i=1}^5p_i\sim L$ on $N(D)$.
Now, by virtue of  the fact that the genus of $N(D)$ is equal to $3$, $4$, or $5$, Riemann--Roch and Clifford's Theorem imply that for any divisor $E$ of degree $5$ on $N(D)$, we have $h^0(N(D),\mathcal O_{N(D)}(E))\le 3$.  Since $h^0(\mathcal O_D(1))=3$, it follows that the natural inclusion, $\nu^*H^0(D,\mathcal O_D(1))\subseteq H^0(N(D),\nu^*\mathcal O_D(1))$ is an equality.
In other words,  $\sum_{i=1}^5p_i\sim L$ on $N(D)$ implies that $\sum_{i=1}^5p_i$ on $D$ is the zero locus of a global section of $\mathcal O_D(1)$; i.e., the $p_1,\dots,p_5$ are collinear, a contradiction.
 Thus $\nu^*\eta_D$ is nontrivial, so that $N(\widetilde D)$, and therefore $\widetilde D$, is irreducible.
\end{proof}

\begin{rem}\label{R:irreducible}
We note here that the proof above also proves the following statement.
Let $(X,\ell)$ be a cubic threefold with allowable singularities, together with a non-special line $\ell$. Let $(\widetilde D, D)$ be the \'etale double cover of the discriminant plane quintic $D$ determined from projection from $\ell$. Then if $D$ is irreducible, and $g(N(D))\ge 3$, then $\widetilde D$ is irreducible.  We also have that $N(D)$ is non-hyperelliptic, since it has a base-point free $g^1_3$ (see e.g., \cite[p.12]{acgh}).
\end{rem}

By Proposition \ref{P:B1} the
line bundle ${\mathcal O}_D(1) \otimes \eta_D$ is an odd theta characteristic with $h^0(D,{\mathcal O}_D(1) \otimes \eta_D)=1$. Let $\eta_{N(D)}=\nu^*(\eta_D)$ and
$\pi: N(\widetilde  D) \to N(D)$ be the corresponding \'etale $2:1$ cover.  Since the period map for cubics extends at $X$, we have $IJ(X)=P_{\widetilde D/D}$.
The results of \cite{abh} give us a description of $P_{\widetilde D/D}$ as an extension
$$
1 \to \mathbb C^*\to IJ(X) \to P_{N(\widetilde D)/N(D)}\to 0
$$
with extension data given by the line bundle:
$$
 \mathcal O_{N(\widetilde D)}(\tilde p_1^+ -\tilde p_2^+ -\tilde p_1^- +\tilde p_2^- ).
$$

Here we are using the following notation.
For the node $p$ of $D$ we set $\nu^{-1}(p)=\{p_1,p_2\}$.
Let $\tilde p^+$ and $\tilde p^-$ be the pre-images of $p$ under the \'etale cover $\pi': \widetilde D \to D$.
Now let $\tilde p_1^{\pm}$ and $\tilde p_2^{\pm}$ be the pre-images of
$p_1$ and $p_2$ respectively under the cover $\pi: N(\widetilde D) \to N(D)$ of the normalizations. Our convention is such that
the normalization map $\widetilde \nu: N(\widetilde D) \to \widetilde D$  maps $\widetilde \nu(\tilde p_1^+) =\widetilde \nu(\tilde p_2^+)= \tilde p^+$
and $\widetilde \nu(\tilde p_1^-) =\widetilde \nu(\tilde p_2^-)= \tilde p^-$.

The following lemma shows that this extension data is of the same type as described in Theorem \ref{theo_CollMurre}:
$$
 \mathcal O_{N(\widetilde D)}(\tilde p_1^+ -\tilde p_2^+ -\tilde p_1^- +\tilde p_2^- )\in 2_*\Sing(\Theta_{P_{N(\widetilde D)/N(D)}}).
$$

\begin{lem}\label{lem:extensiondata}  For brevity in notation, set $C=N(D)$, and $\widetilde C=N(\widetilde D)$.
The following holds:
 \begin{itemize}
\item[(1)] Let $G$ be the unique $g^1_3$ on $C$. Then
\begin{equation}\label{equ:thetasing5}
\Sing(\Theta_{P_{\widetilde C/C}})=\{\pi^*G\otimes \mathcal O_{\widetilde C}(\tilde p_1^+ +\tilde p_2^-), \pi^*G\otimes \mathcal O_{\widetilde C}(\tilde p_1^- +\tilde p_2^+)\}
\subseteq \operatorname{Pic}^{8}(\widetilde C).
\end{equation}
\item[(2)] Translation by the inverse of a theta characteristic and multiplication by $2$ gives
\begin{equation}\label{equ:thetasing4}
2_*\Sing(\Theta_{P_{\widetilde C/C}}) = \{{\mathcal O}_{\widetilde C}(\tilde p_1^+ - \tilde p_1^- + \tilde p_2^- - \tilde p_2^+)^{\otimes \pm 1}   \}
\subseteq \operatorname{Pic}^{0}(\widetilde C).
\end{equation}
\end{itemize}
\end{lem}

\begin{proof}
We first note that $K_C=G^{\otimes 2}\otimes \mathcal O_C(p_1+p_2)$. This follows from $K_D=\mathcal O_{D}(2)$, the fact that the $g^1_3$ is given by the lines through the node of $D$, and normalization.
Hence the second claim follows from the first, which we will prove now.
To start, the trigonal construction, recalled below, implies that $P_{\widetilde C/C}$ is a Jacobian of a genus $4$ tetragonal curve, which is not hyperelliptic.
Therefore,  $\Sing(\Theta_{P_{\widetilde C/C}})$ consists of two (not necessarily distinct) points, which must be exchanged by $\iota^*$.
Hence, via Mumford's description of the Prym \cite{mprym},
$$
\operatorname{Sing}\Theta_{P_{\widetilde C/C}}=\{\tilde L,\iota^*\tilde L\}\subseteq \operatorname{Pic}^8(\widetilde C)
$$
for some $\tilde L\in \operatorname{Pic}^8(\widetilde C)$ with $\operatorname{Nm}\tilde L=K_C$ and $h^0(\tilde L)\equiv 0 \pmod 2$.
Moreover, since $\dim \operatorname{Sing}\Theta_{P_{\widetilde C/C}} =0\ge g-5$, Mumford's lemma~\cite[p.~345]{mprym}
implies that
$$
\tilde L=\pi^*M\otimes \mathcal O_{\widetilde C}(B), \ \ \iota^*\tilde L=\pi^*M\otimes \mathcal O_{\widetilde C}(\iota^*B)
$$
with  $M$ a line bundle on $C$ with $h^0(M)\ge 2$,  and $B\ge 0$ an effective divisor on $\widetilde C$.  Since $C$ is not hyperelliptic, we know that $\deg M\ge 3$,  so that there are two cases to consider:

\begin{enumerate}

\item $\deg M=4$, $B=0$ (in which case $\tilde L=\pi^*M$, and $M$ is a theta characteristic with sections; recall $K_C=\operatorname{Nm}\tilde L =\operatorname{Nm}\pi^*M=M^{\otimes 2}$);

\item $\deg M=3$,  $\deg B=2$ (in which case $M=G$, and $B=\tilde p_1^\pm +\tilde p_2^\pm$).
\end{enumerate}

Case (1) can be ruled out since $D$ is nodal, and therefore does not have a unibranched singularity.
Indeed, the theta characteristic $M$ is in particular a $g^1_4$, so it is either of the form $M=G(q)$ for some point $q$, or $M=G(p_1+p_2-q)$ (see \cite[p.208]{acgh}, and note that since $\nu^*\mathcal O_D(1)=G(p_1+p_2)$, the line bundle $M=G(p_1+p_2-q)$  corresponds to the linear system of lines in the plane through $q$).  The first gives   $G^{\otimes 2}(2q)=K_C=G^{\otimes 2}(p_1+p_2)$, which is only possible if $2q=p_1+p_2$; in other words, $q=p_1=p_2$, and the singularity of $D$ is unibranched, a contradiction. The second gives $G^{\otimes 2}(2p_1+2p_2-2q)=K_C=G^{\otimes 2}(p_1+p_2)$, which for the same reason would force $q=p_1=p_2$, and therefore the singularity of  $D$ to be unibranched, again a contradiction.

For Case (2),
it suffices by the symmetries to rule out
 $\tilde L=\pi^*G\otimes \mathcal O_{\widetilde C}(\tilde p_1^+ +\tilde p_2^+)$; this can be accomplished by showing that this line bundle has an odd dimensional space of global sections, since by assumption $h^0(\tilde L)\equiv 0\pmod 2$.    To this end, observe that $
\tilde L=\pi^*G\otimes \mathcal O_{\widetilde C}(\tilde p_1^+ +\tilde p_2^+)=\pi^*\nu^*\mathcal O_D(1)\otimes \mathcal O_{\widetilde C}(-\tilde p_1^--\tilde p_2^-) =(\tilde \nu)^* (\pi_D)^*\mathcal O_D(1)\otimes \mathcal O_{\widetilde C}(-\tilde p_1^--\tilde p_2^-)$,
so that $(\pi_D\tilde \nu)_*\tilde L=\mathcal O_D(1)\otimes (\pi_{D})_{*}\tilde \nu_* \mathcal O_{\widetilde C}(-\tilde p_1^--\tilde p_2^-)=\mathcal O_D(1)\otimes (\pi_{D})_{*}\mathcal I_{\tilde p^-}$, where $\mathcal I_{\tilde p^-}$ is the ideal sheaf of $\tilde p^-$ in $\mathcal O_{\widetilde D}$.
Applying $(\pi_{D})_{*}$ to the short exact sequence  $0\to \mathcal I_{\tilde p^-}\to \mathcal O_{\widetilde D}\to \mathbb C_{\tilde p^-}\to 0$, and tensoring by  $\mathcal O_D(1)$, we obtain
$$
0\to \mathcal O_D(1)\otimes (\pi_{D})_{*}\mathcal I_{p-}\to\mathcal O_D(1) \oplus( \mathcal O_D(1)\otimes \eta_D) \to \mathcal O_D(1)\otimes \mathbb C_p\to 0.
$$
The map on the right is clearly surjective on global sections, and thus we have
$$
h^0(\tilde L)=h^0((\pi_{D}\tilde \nu)_*\tilde L)= h^0(\mathcal O_D(1)\otimes (\pi_{D})_{*}\mathcal I_{p-})=3+h^0(\mathcal O_D(1)\otimes \eta_D)-1=3,
$$
completing the proof.
\end{proof}

\subsubsection{The trigonal construction in the $A_1$-case}
Again, for notational convenience, let us set $C=N(D)$, and $\widetilde C=N(\widetilde D)$.
Since  $C$ is trigonal,
 it  follows from Recillas' trigonal construction \cite{rec} that there exists a non-hyperelliptic
tetragonal genus $4$ curve $\Sigma_1$ such that the Prym variety $P_{\widetilde C/C}=\Prym (\widetilde C \to C)$ is isomorphic to the Jacobian of $\Sigma_1$:
\begin{equation}\label{equ:trigonal}
(P_{\widetilde C/C},\Theta_{P_{\widetilde C/C}}) \cong (J(\Sigma_1), \Theta_{\Sigma_1}).
\end{equation}
We do not want to go into the details of the trigonal construction, but, for future reference, we want to explain how the curve $\Sigma_1$ can be constructed
from the double cover $\pi: \widetilde C \to C$ and the $g^1_3$ on $C$, namely as
\begin{equation}\label{equ:trigonalcon}
\Sigma_1=\{\tilde p_1+\tilde p_2+\tilde p_3\in \operatorname{Sym}^3\widetilde C: \exists p\in \mathbb P^1, \ \ \pi(\tilde p_1+\tilde p_2+ \tilde p_3)=t^{-1}(p)\}/\sim
\end{equation}
where $t: C \to \mathbb P^1$ is the $3:1$ map given by the $g^1_3$, $\iota: \widetilde C \to \widetilde C$ is the covering involution, and the equivalence relation is given by
\begin{equation*}
\tilde p_1+\tilde p_2+\tilde p_3\sim \iota(\tilde p_1)+\iota(\tilde p_2)+\iota(\tilde p_3).
\end{equation*}

Both $\Sigma_1$ and the associated $(2,3)$-curve are smooth complete intersection genus $4$ curves in $\mathbb P^3$.  Here we give a direct proof via the trigonal construction that the two curves are isomorphic.

\begin{pro}\label{pro:IdentificationCollMurre}
The trigonal construction gives an isomorphism
$\Sigma_1 \cong \Sigma$, where $\Sigma$ is the $(2,3)$-curve associated to $X$.
\end{pro}

Before we give the proof of this proposition it is useful to review the Bruce--Wall description of the lines on a
cubic surface $S$ with exactly one singularity, which is of type $A_1$.
There are $6$ lines (counted with multiplicity $2$ as points of the Fano scheme of lines) $\ell_1,\cdots,\ell_6$ passing through the singular point of $S$.
There are $\binom{6}{2}=15$ lines $\ell_{ij}$, $1\le i<j\le 6$, determined by the plane $\langle \ell_i,\ell_j\rangle$ (i.e., the residual line of intersection, which meets $\ell_i$ and $\ell_j$ away from their point of intersection, since otherwise  $\ell_{ij}$ would pass through the singular point of the surface).
These  $21$  lines (or $27$ counted with multiplicity) are all of the lines on $S$.    Their further  incidence can be described as follows.   The line $\ell_{ij}$ meets  $\ell_k$ if and only if $k=i,j$, and meets $\ell_{kp}$ if and only if $i,j,k,p$ are distinct.  Moreover, as can be deduced from what was just explained, the lines~$\ell_{ij}$, $\ell_{kp}$, and $\ell_{qr}$ are coplanar if $i,j,k,p,q,r$ are all distinct.

\begin{proof}[Proof of Proposition \ref{pro:IdentificationCollMurre}]
We want to show that the genus $4$ curve $\Sigma_1$ from the trigonal construction is identified with the $(2,3)$-curve $\Sigma$.   We know that points on the $(2,3)$-curve $\Sigma$ correspond to lines $\ell'$ in $X$ passing through the singular point.  Let $\ell$ be the non-special line we chose in $X$ in order to define the \'etale double cover $\widetilde D\to D$, which gives rise to the curve $\Sigma_1$, as explained above.  If we choose a general such $\ell'\in \Sigma$ (general line through the singular point), then we will get a
cubic surface with exactly one singularity, which is of type $A_1$, defined by intersecting $X$ with the $\mathbb P^3$ given by $\langle \ell ,\ell' \rangle$.   This $\mathbb P^3$ corresponds under projection from $\ell$ to a line in $\mathbb P^2$  passing through the $A_1$-singularity of the plane quintic $D$.  This line has $3$ residual points on $D$, giving an effective divisor $B = p_{34}+p_{35}+p_{36}$ in the $g^1_3$ on the normalization $C$ of $D$.    Now in the notation of Bruce--Wall, let us take $\ell$ to be $\ell_{12}$, and $\ell'$ to be $\ell_6$.   From the Bruce--Wall description, there are $4$ pairs of coplanar lines meeting $\ell=\ell_{12}$, namely,  $(\ell_{34},\ell_{56})$, $(\ell_{35},\ell_{46})$, $(\ell_{36}, \ell_{45})$, and $(\ell_1,\ell_2)$.  We have chosen the labeling so that the first pair of lines maps to  $p_{34}$, the second to $p_{35}$ and the third to $p_{36}$.   Note that  the lines $\ell_{34},\ell_{56}, \ell_{35},\ell_{46},\ell_{36}, \ell_{45}$ correspond to the points in $\widetilde C$ lying over the respective points of $C$, which correspond to the pairs  $(\ell_{34},\ell_{56})$, $(\ell_{35},\ell_{46})$, $(\ell_{36}, \ell_{45})$.  The line $\ell_6$  picks out the lines $\ell_{56}$, $\ell_{46}$ and $\ell_{36}$ (it meets these lines).    The point $\ell_{56}+\ell_{46}+\ell_{36}\in \operatorname{Sym}^3(\widetilde C)$  is by definition a point of $\Sigma_1$, the curve obtained from the trigonal construction (note the equivalence in the construction means that if we chose lines not meeting $\ell_6$, i.e., indices not including $6$, instead of the indices including $6$,  we would get the same point of the genus 4 curve).  This process is reversible, and gives a birational map (defined on general points) $\Sigma\dashrightarrow \Sigma_1$.  This of course extends to an isomorphism.
\end{proof}
\begin{rem}
One can reverse the trigonal construction and thus show that the generic point in $\bf A$ can be obtained as the intermediate Jacobian of a cubic with a unique
$A_1$-singularity.
\end{rem}

\subsection{Cubic threefolds with an $A_3$-singularity and the locus $\mathbf B$}

We will now deal with the locus ${\bf B}$.
Again, in this case it is very instructive to understand the geometry of an $A_3$ degeneration in concrete terms.
That is, we give a  second geometric proof of Theorem \ref{prop:locusB} via the theory of Prym varieties.
In particular, this is the first example where we encounter the phenomenon of taking a stable reduction of a family of plane quintics with allowable singularities and
inserting a tail.
Let $\Delta\subset\CC$ be the unit disk, and let $\mathscr X\to\Delta$ be a family of cubic threefolds such that $X_0$ has a unique $A_3$-singularity, while all other fibers are smooth. By abuse of notation let $\widetilde{IJ}:\Delta\to\AV$ denote the induced intermediate Jacobian map --- which we know is a morphism. We want to show
that $\widetilde{IJ}(X_0)\in {\bf B}$.

The extension problem can be formulated in terms of double covers of plane quintics, so that $\widetilde{IJ}(\mathscr X)$ is then the Prym of the family $\widetilde{\mathscr D}\to\mathscr D\to\Delta$ of
\'etale double covers of plane quintics,
such that the central fiber $D$
has a unique singularity, which is of type $A_3$, while
$D_t$ is smooth for $t\ne 0$.
As was already mentioned in Lemma~\ref{lem:irreducible}, a plane quintic curve with exactly one singularity, which is of type $A_3$, is irreducible.
We will compute the intermediate Jacobians by computing the Pryms for some stable reduction $\pi:\widetilde{\mathscr C}\to\mathscr C\to\Delta$ of this family --- so after a finite base change, over the punctured disk $\Delta^0$ the families $\widetilde{\mathscr D}$ and $\widetilde{\mathscr C}$ agree.

To describe the central fiber $\widetilde C\to C$ of $\widetilde{\mathscr C}$ in more detail, recall from \cite[\S 6]{has}, see also Section \ref{sec:proof}, that the $A_3$-singularity
of $D$ is replaced in the stable reduction by a ``tail'' $T$, that is an elliptic bridge. Since $D$ has a unique $A_3$-singularity, its partial normalization $N(D)$ at this singularity
is smooth, and thus $C$ has two irreducible components $T$ and $N(D)$ meeting in two points. Clearly $N(D)$ has genus $4$. Let $(x,y)\in T$ and $(p,q)\in N(D)$ be the
points where the tail $T$ and the normalization $N(D)$ meet.  The pointed curve $(N(D),p,q)\in \calM_{4,2}$ has the following properties: the pencil of lines
through the $A_3$-singularity
defines a half-canonical $g^1_3$ on $N(D)$. This implies that the canonical model of $N(D)$ lies on a quadric cone.
We also note that the $g^1_3$ contains a member of the form $p+q+r$  for some point on $N(D)$ and thus the three points $p,q$ and $r$ lie on a ruling of the quadric cone.

We now describe the curve $\widetilde C$.  To begin, the cover $\widetilde D\to D$ is nontrivial (Remark \ref{R:B1}), and~$\widetilde D$ is irreducible 
(Lemma \ref{lem:irreducible}).
It follows that  $\widetilde C$ has three irreducible components, $T_1$ and $T_2$ which are interchanged by the covering
involution, and $N(\widetilde D)$, the normalization of $\widetilde D$, which gives a nontrivial \'etale double cover $N(\widetilde D)\to N(D)$.
We denote by $(p_1,q_1,p_2,q_2)\in N(\widetilde D)$ the nodes, and we can now compute the associated Prym of the cover $\widetilde C\to C$.

First note that the geometric genus of $C$ (by which we mean the sum of the genera of the irreducible components of the normalization) is equal to $1+4=5$, while the geometric genus of $\widetilde C$ is $1+1+7=9$. Thus the abelian part of the Prym has
dimension $4=9-5$,  and the torus rank is equal to $1$; thus the Prym lies in $\partial\calA_5'$.
The abelian part of the Prym is clearly equal to $T\times  P_{N(\widetilde D)/N(D)}$, where $ P_{ N(\widetilde D)/N(D)}$ is the Prym variety associated to the double
cover $ N(\widetilde D)\to N(D)$,
see also \cite[Thm.~5.4]{beauville}.
Since $N(D)$ is trigonal, one can use Recillas' construction \cite{rec}
to determine the Prym variety $ P_{ N(\widetilde D)/N(D)}$.
This shows in particular that the Prym variety $ P_{ N( \widetilde D)/N(D)}=J(Y)$ is a Jacobian of a genus $3$ curve $Y$. Moreover, since $N(D)$ has a half-canonical $g^1_3$ its Jacobian has a vanishing theta null,
and the Schottky--Jung proportionality (see Mumford~\cite{mprym}) then implies that $ P_{ N(\widetilde D)/N(D)}$ also has a vanishing theta null, and thus $Y$ is hyperelliptic. The crucial step is now to determine the
extension datum. The following lemma, together with the fact that $Y$ is hyperelliptic, reproves Theorem \ref{prop:locusB} (with $E=T$ and $C'=Y$).
\begin{lem}
The extension datum is of the form
\begin{equation}
\left(\mathscr O_T(x-y),\mathscr O_{Y}(\xi_1 - \xi_2)\right)\in T\times \Theta_Y
\end{equation}
where $x,y$ are the nodes on $T$, and $\xi_1, \xi_2 \in Y$.
\end{lem}
\begin{proof}
It follows from  \cite[Prop.~1.5]{abh} that the extension datum is of the form
\begin{equation*}
\left(\mathscr O_T(x-y),\mathscr O_{Y}\left((p_1+q_2)-(p_2+q_1)\right)\right)\in T\times P_{\widetilde N(D)/N(D)}.
\end{equation*}
where $x,y$ are the nodes on $T$ and $(p_1,q_1,p_2,q_2)$ are the nodes on~$\widetilde N(D)$.
Now recall that there exists $r\in N(D)$ such that $p+q+r$ is the $g^1_3$; let $\tilde r\in N(\widetilde D)$ be a preimage of $r$.
As we have already pointed out, it follows from Recillas' trigonal construction  that $P_{ N(\widetilde D)/N(D)}=J(Y)$ for some hyperelliptic genus $3$ curve $Y$.
Recall further from (\ref{equ:trigonalcon}) that we can view $Y$  as sitting in
$\operatorname{Sym^3}(N(\widetilde  D))$ via the trigonal construction, and taking $\xi_1=p_1+q_2+\tilde r\in Y$ and
$\xi_2=p_2+q_1+\tilde r \in Y$, we find that $\mathscr O_{Y}\left((p_1+q_2)-(p_2+q_1)\right) = \mathscr O_{Y}(\xi_1 - \xi_2)$. Since $Y$ is hyperelliptic, the theta divisor $\Theta_Y$ is given by the
difference variety $Y - Y$, and this proves the lemma.
\end{proof}
\begin{rem}
One can also show that given an irreducible \'etale double cover $\widetilde D \to D$,
where $D$ has a unique singularity, which is of type $A_3$, one can reconstruct a cubic
threefold $X$ with unique singularity,  which is an $A_3$-singularity.
A similar statement holds for families of double covers of plane quintics (Proposition \ref{P:B2}).
In particular the map from $D_{A_3}$
to $\bf B$ is dominant (which also follows from our first proof of Theorem \ref{prop:locusB}).
\end{rem}

\subsection{Images of boundary divisors in $\widetilde M$}
Using similar analysis, one can compute the extension $\widetilde{IJ}$ to a generic cubic with a unique $A_n$-singularity, $n \leq 5$, or a unique $D_4$-sin\-gu\-la\-rity.   Together with work of Collino on chordal cubics, this allows for the description of the images of each of the boundary divisors in $\widetilde M$ under the morphism $\widetilde {IJ}:\widetilde M\to \overline {IJ}$.
We give the torus rank $1$ results in the table below: this table expands on \cite[Table 1]{cml} --- which is concerned with the map to the Satake compactification, and thus only provides the compact part; the rows for $A_1$ and $A_3$ summarize the results above.    The description of the image of the $\tDD[4]$ divisor is not as concise, and since this divisor is   mapped to the torus  rank $2$ locus, we only give a brief description  here (see Remark \ref{R:D4Div}).  Moreover, the results in the next section show that there are no divisorial boundary components of $\overline{IJ}$ with torus rank $2$; in particular,  the $\tDD[4]$ divisor is  contracted, and therefore does not map to a divisorial boundary component of $\overline{IJ}$.

In the table below, the  first column describes the unique singularity of the cubic.
The second column describes the associated $(2,3)$-curve.
It is easy to see that the Jacobian of the normalization of the $(2,3)$-curve has the same dimension
as the  Prym of the normalization of the discriminant curve  (independent of the degeneration), and is for instance a hyperelliptic Jacobian if and only if the Prym is.
In an argument similar to Proposition \ref{pro:IdentificationCollMurre}, Krisztian Havasi \cite{havasi}
has shown that the two abelian varieties are in fact always  isomorphic.
 This is the data in the column $C_{(2,3)}$.  For more details, see \cite{cml, cmjl1}.
The third column describes the tails arising from the stable reduction of the singularity of the associated plane quintic.  The compact part of the Prym is an abelian variety that is a product of abelian varieties of the type in columns two and three. This is summarized in column four.  The dimension of this locus is tallied in column five.  Next, we  give the extension data, which given the abelian variety $A$, is determined by a point $x\in A/\pm$.  The locus of points $x\in A$ that arise as extension data is given in column 6.    The dimension of this locus is given in column 7.  The total dimension of the locus associated to cubics with the given singularity is the sum of columns 5 and 7.  This is given in column 8.  The corresponding loci in \cite{grhu1} are given in the last column.

We note also that curiously enough the boundary divisor $\tDA[5]\subset\partial\tM$ is contracted under~$\widetilde{IJ}$ to a codimension one locus within
the closure of the locus $\bf{A}$, which is  the image of the $A_1$-divisor $\tDA[1]$ under the map $\widetilde {IJ}$.

\small
\begin{table}[htb]
\begin{center} \begin{tabular}{||c||c|c||l|c|c|c||c||c||}
\hline $\operatorname{Sing}(X\subset \mathbb P^4)$ & $C_{(2,3)}\subset\mathbb P^3$& Tail $T$ & Compact Part& dim & Extension&dim&dim& Thm.~\ref{theo:divisors}\\
\hline\hline secant&ribbon  &  --&$\calH_5$&9&--&--&9&$\mathcal H_5$ \\
 \hline $A_1$&$\calM_4-\theta_{\textnormal{null}}$ &  -- &$\calJ_4$& 9&$\pm(g^1_3-\hat g^1_3)$&0&9&${\bf A}$\\
\hline  $A_2$&$\theta_{\textnormal{null}}\cap  \calM_4$ & $ \calM_{1,1}$& $\calA_1\times (\calJ_4 \cap \theta_{\textnormal{null}})$&1+8&--&--&9&$\mathbf K$ \\
\hline $A_3$& $ \calH_{3,2}$& $ \calM_{1,2}$&$ \calA_1 \times  \calH_{3}$&1+5& $T\times \Theta_C$&3 &$9$&${\bf B}$\\
\hline \hline  $A_4$& $ \calH_{3,1}$ & $ \calM_{2,1}$& $\calA_2\times  \calH_{3}$&3+5&--&--&$8$&$\subsetneq \overline{\mathcal H_5}$\\
\hline $A_5$& $ \calM_{2,2}$& $ \calM_{2,2}$&  $\calA_2\times \calA_2$ &3+3 &$\subseteq  2_*\Theta_{T} \times 2_*\Theta_{C}$&2&$\le 8$&$\subsetneq \overline{\mathbf A} $\\
\hline
\end{tabular}
\vspace{0.2cm}
\caption{The extended period map on the boundary divisors.}\label{tablecon}
\end{center}
\end{table}
\normalsize

\begin{rem}\label{R:D4Div}
In the $D_4$ case,  the $(2,3)$-curve is an arithmetic genus $4$ curve lying on the union of two planes in $\mathbb P^3$, consisting of two elliptic curve components, meeting each other in $3$ nodes along the line of intersection of the two planes.    The tail from the $D_4$ singularity is an elliptic curve with $3$ marked points.   One can conclude from the Prym construction that in the $D_4$ case, the compact part consists of the product of $3$ elliptic curves.  The remaining degeneration data can also  be obtained from the Prym construction, however, the description is less concise.
 We suspect that the $D_4$ case may give rise to  the $8$ dimensional locus $\mathbf B22$ described in Theorem \ref{theo:degenerationscod2}.
\end{rem}


\section{Boundary strata of $\overline{IJ}$ of torus rank 2}\label{sec:2strata}

In this section we would like to describe the geometry of the intersection of $\overline{IJ}$ with (the main stratum of)  the torus rank $2$ stratum in a toroidal
compactification $\calA_5^{\operatorname{tor}}$. Again, given the natural
map
$$
\varphi: \calA_g^{\operatorname{tor}} \to \Sat = \calA_g \sqcup \calA_{g-1} \sqcup \calA_{g-2} \sqcup \ldots \sqcup \calA_0
$$
we let $\beta_i^{0,\operatorname{tor}}:= \varphi^{-1}(\calA_{g-i})$. Since all known toroidal compactifications, in particular $\Perf$ and $\Vor$, coincide outside
$\varphi^{-1} (\Sat[g-4])$ the stratum  $\beta_2^{0,\operatorname{tor}}$ does not depend on the chosen toroidal compactification and we will
simply denote it by $\beta_2^0$. The codimension of $\beta_2^0$ in $ \calA_g^{\operatorname{tor}}$ is $2$, in the case of genus $g=5$ it thus
has dimension $13$. The stratum $\beta_2^0$ is stratified into two substrata $\beta(\sigma_{1+1})$ and $\beta(\sigma_{K3})$
depending on the two (up to change of coordinates)
non-degenerate cones
in either the second Voronoi or the perfect cone decomposition in genus $2$. In terms of generating forms these are given by $\sigma_{1+1}=\langle x_1^2,x_2^2 \rangle$
and  $\sigma_{K3}=\langle x_1^2,x_2^2, (x_1-x_2)^2 \rangle$. The codimension of $\beta(\sigma_{1+1})$ in $\calA_g^{\operatorname{tor}} $ is
$2$ and that of $\beta(\sigma_{K3})$ is $3$. We recall that, up to a finite group quotient, the stratum
$\beta(\sigma_{1+1})$ is a $\CC^*$-bundle over the two-fold cartesian product $\mathcal X_{g-2}^{\times 2} \to \calA_{g-2}$ of the universal family
$\mathcal X
_{g-2} \to \calA_{g-2}$. Intrinsically, this $\CC^*$-bundle is the Poincar\'e bundle $\mathcal P$ over $\mathcal X_{g-2}^{\times 2} $ trivialized along the $0$-section of
the cartesian product of the universal families with its own $0$-section removed. The stratum $\beta(\sigma_{K3})$ lies in the closure of the stratum $\beta(\sigma_{1+1})$. It
is, again up to the action of its symmetry group, isomorphic to  $\mathcal X_{g-2}^{\times 2} \to \calA_{g-2}$ and can be thought of as the $0$-section of the Poincar\'e bundle.
In \cite[Prop.~10.3]{grhu2} the intersection $\overline{IJ}\cap\beta(\sigma_{K3})$ was determined and shown to be purely 7-dimensional.
Here we will discuss the intersection of $\overline{IJ}$ with $\beta_{2,0}^0:=\beta(\sigma_{1+1})$. Geometrically $\beta_{2,0}^0$ parameterizes semi-abelic varieties with torus rank $2$ whose
normalization is a $\PP^1\times\PP^1$ bundle over an abelian variety of dimension $g-2$.

We shall now restrict ourselves to the case we are interested in, namely $g=5$. As most of our arguments depend on concrete calculations we must first fix the
coordinates with which we will be working. To do this start with a period matrix in genus $5$:
$$
\tau_5=
\begin{pmatrix}
\tau_{11} & \tau_{12}  & \tau_{13}  & \tau_{14} & \tau_{15} \\
* & \tau_{22}  & \tau_{23}  & \tau_{24} & \tau_{25} \\
* &  *  & \tau_{33}  & \tau_{34} & \tau_{35} \\
* & *  &*  & \tau_{44} & \tau_{45} \\
* & *  & *  & * & \tau_{55}
\end{pmatrix}.
$$
In order to first describe $\beta_1^0$ we consider the map
$$
{\mathbb H} _5\to \CC^* \times \CC^4 \times {\mathbb H} _4,\qquad \tau_5 \mapsto (t_1, z_4, \tau_4)
$$
where $t_1=e^{2 \pi i \tau_{11}}$, $z_4=(\tau_{12}, \tau_{13},\tau_{14},\tau_{15})$ and $\tau_4=({\tau_{ij}}), i,j=2,\ldots,5$.
The partial compactification $\calA_5\sqcup\beta_1^0$ is obtained by the torus embedding $\CC^* \hookrightarrow \CC$, in other words by adding the origin $\{0\}$ to $\CC^*$, and $\beta_1^0$ is then given by $t_1=0$.
This also shows that $\beta_1^0$ is the universal Kummer family
family ${\mathcal X}_{4}/(\pm 1)$ where the base is given by $\tau_4$ and $z_4=(\tau_{12}, \tau_{13},\tau_{14},\tau_{15})$ are coordinates on the fibers. The Kummer involution
is given by $z_4 \mapsto -z_4$ which is induced by an involution in $\Sp(5,\ZZ)$. In terms of semi-abelic varieties parameterized by points on $\beta_1^0$ the situation is the following. Given $(\tau_4,z_4)$, then  $z_4\in A_{\tau_4}$ defines a $\CC^*$-extension of $A_{\tau_4}$. This is compactified to a $\PP^1$-bundle where the $0$-section and the $\infty$-section are glued with a shift by $z_4$.
Then $\pm z_4$ define isomorphic semi-abelic varieties.

We now move to the stratum $\beta_{2,0}^0$. For this we consider the partial quotient
 $$
{\mathbb H} _5\to (\CC^*)^3 \times \CC^3 \times \CC^3 \times {\mathbb H} _3,\quad \tau_5 \mapsto (t_1,t_2,t_3), (b, z_3, \tau_3)
$$
where $t_1=e^{2 \pi i \tau_{11}}$, $t_2=e^{2 \pi i \tau_{12}}$, $t_3=e^{2 \pi i \tau_{22}}$, $b=(\tau_{23},\tau_{24},\tau_{25})$, $z_3=(\tau_{13},\tau_{14},\tau_{15})$
and $\tau_3=({\tau_{ij}}), i,j=3,\ldots,5$. The partial compactification given by the cone $\sigma_{1+1}=\langle x_1^2,x_2^2 \rangle$ is then obtained by considering the torus
embedding $(\CC^*)^3 \hookrightarrow \CC \times \CC^* \times \CC$, and  $\beta_{2,0}^0$ is given by $t_1=t_3=0$. Above we have described $\beta_{2,0}^0$ as the
$\CC^*$ bundle over the product $\mathcal X_{3}^{\times 2} \to \calA_{3}$ given by the Poincar\'e bundle with the $0$-section removed. The connection to the variables we have just
introduced is the following: $\tau_3$ is a point in the base $\calA_3$, $(b,z) \in A_{\tau_3} \times A_{\tau_3}$, and $t_2$ is the fiber coordinate of the Poincar\'e bundle.
We will rename this variable $x:=t_2$ as this fits in better with calculations which have appeared previously in the literature.  The points in  $\beta_{2,0}^0$ have the following
interpretation in terms of semi-abelic varieties. First of all the pair $(b,z_3)$ determines a $(\CC^*)^2$-extension over the abelian threefold $A_{\tau_3}$. This is then compactified
to a $\PP^1 \times \PP^1$-bundle over $A_{\tau_3}$, the opposite sides of which are further glued with shifts by $b$ and $z$ on the base, respectively, and via multiplication by $x\in \CC^*$ on $\PP^1$ (i.e., the identifications are $(z_3,0,t)\sim(z_3+b, \infty,xt)$ and $(z_3,t,0)\sim(z_3+z,xt,\infty)$).

We will find it convenient to consider points in $\beta_2^0$ as limit points of $\beta_1^0$, understood as points on a partial compactification $\calX_4'/\pm 1$ of the universal Kummer variety $\calX_4 / \pm1$, over the partial compactification $\ab[4]'$. Recall that Mumford's partial toroidal compactification is defined as $\ab[4]'=\ab[4] \sqcup \calX_3/\pm 1$.
In order to avoid confusion with Kummer involutions we will describe the boundary $\partial\calX_4'$ before taking the involution, as a family over $\calX_3/\pm 1$. Recall from the discussion above that we have the universal semi-abelian family over $\calX_3$: for every $(\tau_3,b)\in\calX_3$ we take the $\CC^*$-extension over $A_{\tau_3}$ given by $b$. Varying $(\tau_3,b)$ then gives us a $\CC^*$-bundle over $\calX_3$ which we compactify to a $\PP^1$-bundle.
We now glue the
$0$-section and the $\infty$-section of this $\PP^1$-bundle as follows:
denoting, as before, by $z\in A_{\tau_3}$ and $x\in \PP^1$ the coordinates in the universal family and fibers of the $\PP^1$ bundle respectively, the gluing is given by
$(z,0)\sim (z+b,\infty)$. In this way, up to Kummer involutions,  $\beta_2^0$ can be written as the union of two strata: $\beta(\sigma_{K3})$ corresponds to the locus where $x=0$ (which is identified with the locus where $x=\infty$), and $\beta(\sigma_{1+1})$ is the total space of the universal $\CC^*$-bundle over $\calX_3$.

The reader will notice that this description is
very similar to the discussion of the family of torus rank $2$ semi-abelic varieties in the previous paragraph. Clearly these two are closely related, but should not be confused.
The former is a family of semi-abelic varieties {\em over} the stratum $\beta_2^0(\sigma_{1+1})$, the latter construction takes place {\em inside} the toroidal compactification of
$\ab[5]$ itself, namely it describes the part in any of the standard toroidal compactifications of $\ab[5]$ which lies over $\ab[3]$ in $\Sat[5]$ (in fact this description of the torus rank $2$ part holds, adapted suitably, for all genera).

The description in the last paragraph seemingly destroys the symmetry of the two factors in the family $\calX_3^{\times 2}$ in the previous description. However, as the intrinsic
description of $\beta_{2,0}^0$ shows there is a symmetry exchanging the two factors of  $\mathcal X_{3}^{\times 2}$. This symmetry comes from the fact that the cone
$\sigma_{1+1}=\langle x_1^2,x_2^2 \rangle$ allows the symmetry which interchanges $x_1$ and $x_2$.
Geometrically one should think of this as follows: on a suitable level cover of $\calA_5^{\operatorname{tor}}$ the two boundary divisors given by $t_1=0$ and $t_3=0$ correspond to different divisorial irreducible components of the boundary, say $D_1$ and $D_3$, which intersect along some two-dimensional locus on the level cover, which covers the stratum $\beta_2^0\subset \calA_5^{\operatorname{tor}}$. The deck transformation group of this level cover acts transitively on the set of its irreducible boundary divisors, and in particular contains an involution $j$ that interchanges $D_1$ and $D_3$, and thus induces an involution on their intersection --- which $j$ preserves as a set.
The involution induced by $j$ on $\beta_{2,0}^0$ is just interchanging the two factors of $\mathcal X_{3}^{\times 2}$.

Our goal is to describe $\overline{IJ}\cap\beta_{2}^0$; in principle one could approach this by using the computations of the theta gradients on the boundary components performed in \cite{grhu2}, but as the locus where a suitable gradient of theta vanishes has an additional component  $\calA_1\times\theta_{\rm null}^{(4)}\subset\calA_5$, one would then need to distinguish its boundary from that of $IJ$. Thus we take a different approach.
We have already recalled the description of $\overline{IJ}\cap\beta_1^0={\bf A} \cup {\bf B}$ from \cite[Thm.~9.1]{grhu1}, see Theorem \ref{theo:codimension1}.

\begin{pro}
$\overline{IJ}\cap\beta_{2}^0=(\overline{\bf A}\cup\overline{\bf B})\cap\beta_{2}^0$.
\end{pro}
\begin{proof}
As $\beta_2^0$ is a stratum in any toroidal compactification, it suffices to prove the proposition for the case of $\AP$.  We claim that for this compactification the equality $\overline{IJ}\cap\partial\AP=\overline{\bf A}\cup\overline{\bf B}$ holds; this claim immediately implies the proposition, by intersecting both sides with $\beta_2^0$.

To prove the claim, we first note that the boundary $\partial\AP$ is a Cartier divisor in~$\AP$, and thus every irreducible component of the intersection $\overline{IJ}\cap\partial\AP$ must have codimension one within $\overline{IJ}$. Thus to prove the claim we need to show that there does not exist a $9$-dimensional irreducible component of $\overline{IJ}\cap\partial\AP$ contained in $\beta_2$ --- from which it would follow that each irreducible component of $\overline{IJ}\cap\partial\AP=\overline{IJ}\cap\beta_1$ intersects $\beta_1^0$, which by the previous results would mean it is the closure of either $A$ and $B$. Suppose for contradiction that there were an 9-dimensional component of $\overline{IJ}\cap\AP$ contained in $\beta_2$. As any such component would have to come from a divisor in $\tM$, we see
that no possibilities in Table~\ref{tablecon} give such divisorial components, and thus the only case to potentially investigate is the image of $\tDD[4]$ --- which by Remark~\ref{R:D4Div} is indeed mapped generically into $\beta_2^0$. One can then investigate the extension data further to show that the image of $\tDD[4]$ is indeed at most 8-dimensional,
but also we note that computations in \cite{grhu2} describe explicitly the intersection $\overline{IJ}\cap\beta_2^0$ by computing the Fourier--Jacobi expansion of theta gradients --- and one easily sees in fact that this intersection is 8-dimensional.
\end{proof}

Thus to compute the intersection $\overline{IJ}\cap\beta_2^0$, we will need to study the closures of ${\bf A}$ and ${\bf B}$ and we will do this in the picture of the
partial compactification of the Kummer family. To fix notation, we write $J(C)=\Pic^0(C)$ for the (degree zero) Jacobian of the curve, and write $\Theta\subset\Pic^{g-1}(C)$ for the canonically defined theta divisor, which is the vanishing locus of a function $\theta$. In degree zero different symmetric (under the involution $\pm 1$) theta divisors differ by points of order two, so the notion of $2_*\Theta_C\subset\Pic^0(C)$ makes sense, as the image of any symmetric degree zero theta divisor under the multiplication by two map. We also note that any hyperelliptic curve $C$ of genus 3 has precisely one vanishing theta-null, and thus inside its $\Pic^0(C)$ there exists a unique symmetric theta divisor whose unique singularity is at the origin; it is simply given as $C-C$. Also note that for an elliptic curve $E_t, t \in \mathbb H_1$ we have $g-1=0$, and thus
the canonical theta divisor lives in $\Pic^0$.
This is, working with the standard theta function, the unique odd $2$-torsion point, which in $E_t$ is given by $(1+t)/2$.
\begin{teo}\label{theo:degenerationscod2}
The intersection $\overline{IJ}\cap\beta_{2}^0$ of the closure of the locus of intermediate Jacobians with the locus of semi-abelic varieties of torus rank $2$ consists of the closure of the  following four irreducible loci in $\beta_{2,0}^0$, denoted ${\bf A}11b$, ${\bf B}1$, ${\bf B}21$, ${\bf B}22$, and given as follows:
\begin{flalign*}
  {\bf A}11b:=\lbrace \tau_3&=J(C),z=2(p+q)-K_C, b=-z=K_C-2(p+q) , x=y^2\mid&\\ &C\in\calM_3, p,q\in C,
  \ y=-\operatorname{grad}_z\theta(\tau,p+q)/\operatorname{grad}_z\theta(\tau,K_C-p-q) \rbrace&
\end{flalign*}
\begin{flalign*}
  {\bf B}1:=\lbrace \tau_3&\in \calH_3, b=0, z\in \Theta_{\tau_3}, x\in \CC^*\rbrace&
\end{flalign*}
\begin{flalign*}
  {\bf B}21:=\lbrace \tau_3=&t\times\tau_2 \in \calA_1 \times \calA_2,  b=(0,b_2) \in \{0\} \times 2_*\Theta_{\tau_2}, z=(z_1,z_2) \mid&\\
 &\theta(\tau_2,z_2)+x\theta(\tau_2,z_2+b_2)=0\rbrace&
\end{flalign*}
\begin{flalign*}
  {\bf B}22:=\lbrace \tau_3=&t\times t'\times t'' \in {\calA_1} \times \calA_1 \times \calA_1, b=(0,b',b''), z=(z,z',z'') \mid &\\ &\theta(t',z')\theta(t'',z'') +x\theta(t',z'+b')\theta(t'',z''+b'')=0\rbrace.&
\end{flalign*}
We denote throughout by $\tau_i$ a period matrix of an abelian $i$-fold, use $t$ for period matrices of elliptic curves, $b_i,z_i\in A_{\tau_i}$, and $b',b'',z,z',z''$
are points on elliptic curves.
\end{teo}

\begin{rem}
Note that the theorem above describes the intersection with all of $\beta_2^0$, as a union of these 4 irreducible 8-dimensional components. The intersection $\overline{IJ}\cap\beta(\sigma_{K3})$ was already computed explicitly in \cite[Proposition~10.3]{grhu1} and shown to be purely 7-dimensional. The theorem above gives an independent proof of this result, as in none of the components above we have $x=0$ identically, so none of them are contained in $\beta(\sigma_{K3})$.
\end{rem}
The proof of the theorem will use the following statement about the geometry of cubics:
\begin{lem}\label{lem:A2}
Let $X$ be a cubic threefold with precisely one $A_1$- and one $A_2$-singularity. Then $IJ(X)$ has torus rank $1$ and is thus not contained in $\beta_2$.  Here, by $IJ(X)$, we mean the image under  $\widetilde {IJ}$ of a point in $\widetilde {\mathcal M}$ lying over $X$.
\end{lem}
\begin{proof}
Projecting from a non-special line  $\ell$ on $X$ we obtain a conic bundle whose discriminant $D$ is a plane curve with one $A_1$- and one $A_2$-singularity and
a double cover $\widetilde D \to D$ which is \'etale over the singularities. $D$ must be irreducible, since a cubic or a quartic with one cusp would intersect the residual conic or line
in more than just an ordinary node.
It follows from Remark \ref{R:irreducible} that $\widetilde D$ is also irreducible.

The stable reduction of $D$ is a stable curve with two irreducible components, one being an elliptic tail (which
replaces the cusp), and the other being an arithmetic  genus $5$ curve with one self-node, which intersects the elliptic curve in one
point.  The stable reduction of $\widetilde D$ is a stable curve with three irreducible components, two being elliptic tails (which
replace the cusps and are interchanged by the involution), and the other being an irreducible arithmetic genus $7$ curve with two self-nodes (interchanged by the involution), which intersects each of the  elliptic tails in one
point.  Setting $\widetilde \Gamma$ to be the dual graph of the stable reduction of $\widetilde D$, and $H_1(\widetilde \Gamma,\mathbb Z)^-$ to be the anti-invariant part of the homology with respect to the involution, one can easily compute that $\operatorname{rank}H_1(\widetilde \Gamma,\mathbb Z)^-=1$, so that the Prym has torus rank $1$ (\cite[Prop.~1.3]{abh}).  In fact, one can describe the Prym in more detail, lying in the intersection of the closures of the loci  $\mathbf A$ and $\mathbf K$ (see \cite[\S 4.1.12]{havasi}, which was completed after the first version of our preprint appeared, and contains a  number of further examples).
\end{proof}

We will now prove the theorem --- the method is a combination of analytic computations (partially known) for degenerations of theta functions on semi-abelic varieties, combined with various results on the loci of cubic threefolds with singularities that were obtained above.
\begin{proof}[Proof of Theorem \ref{theo:degenerationscod2}]
Since $\beta_2^0$ is a Cartier divisor in the partial boundary  $\beta_1^0 \cup \beta_2^0$ it follows that all components of
$\overline{IJ}\cap\beta_{2}^0=(\overline{\bf A}\cup\overline{\bf B})\cap\beta_{2}^0$ have dimension $8$. It also follows from \cite[Prop.~10.3]{grhu1} that all components
of the intersection $(\overline{\bf A}\cup\overline{\bf B})\cap\beta(\sigma_{K3})$ have dimension $7$, hence it will be sufficient to consider the intersection
$(\overline{\bf A}\cup\overline{\bf B})\cap\beta_{2,0}^0$.

The proof now proceeds by a rather lengthy enumeration of possible cases.
For this, we first need to recall the analytic description of the theta divisor of a semi-abelic variety of torus rank one corresponding to a point $(\tau,b)\in\calX_{g-1}/\pm 1\subset {\calA}'_g$. The semi-abelic variety is glued from the $\PP^1$-bundle over $A_{\tau}$ given by $b$ (thought of as the point of the dual abelian variety, identified with $A_\tau$ by the principal polarization), by identifying the $0$ and $\infty$ sections with a shift by $b$. As shown by Mumford \cite{mumfordAg} (see \cite{grhu2} for many more details), the equation for the theta divisor of the semi-abelic variety is
$$
  \theta(\tau,z)+x\theta(\tau,z+b)=0
$$
where $z\in A_\tau$ is the coordinate on the abelian part, and $x\in\PP^1$ is the coordinate on the fiber.
We now describe the closures of the components ${\bf A}$ and ${\bf B}$ in the partial toroidal compactification of $\calX_4/\pm 1$, and the approaches we take are somewhat different. While for component {\bf B} we are dealing with the closure of the universal theta divisor in the partial toroidal compactification, and the above description would suffice directly, for component {\bf A} such a direct approach would be much harder. Indeed, locus {\bf A} is defined using the geometry of the locus of Jacobians and the global family of singularities of theta divisors over it, and describing the degenerations of this would require a suitable study of limit linear systems. Instead of taking this approach, we use the known results on the global family of singularities of theta divisors over $\calA_g$, and the fact that the locus of cubics with two $A_1$-singularities is irreducible, see Remark \ref{rem:components}.

\subsection*{Case {\bf A}}
Mumford \cite{mumfordAg}  introduced the universal family $\calS:=\operatorname{Sing}_{vert}\Theta\subset\calX_g$ of singularities of theta divisors in the vertical direction.
This has been further studied by Debarre \cite{debarredecomposes}, Ciliberto and van der Geer \cite{cvdg1,cvdg2} and Salvati Manni and the second-named author \cite{grsmordertwo}. In particular it is known, see \cite[Cor.~8.10]{cvdg2} and references therein, that $\calS$ has three irreducible components, $\calS=\calS_{\rm null}\cup\calS_{\rm dec}\cup \calS'$, where $\calS_{\rm null}$ is the locus of two-torsion points that are singular on the theta divisors over $\Theta_{\rm null}\subset\calA_g$, $\calS_{\rm dec}$ denotes the locus
\begin{equation}\label{Sdec}
 \calS_{\rm dec}=\lbrace \tau=t\times \tau_{g-1}, z\in\Theta_t\times \Theta_{\tau_{g-1}}\rbrace\subset\calX_g
\end{equation}
that projects to $\calA_1\times\calA_{g-1}$, and finally $\calS'$ is the remaining irreducible component of $\calS$.
For $g=4$ the theta divisor is singular at a point not of order two only if the ppav is a Jacobian, and since $\theta_{\rm null}$ does not contain the Jacobian locus $\calJ_4$, it follows that  $2_*\calS'$ (which means the image under the fiberwise multiplication by two) is precisely the locus {\bf A}.
To describe the boundary of ${\bf A}=2_*\calS'$ in the partial toroidal compactification of $\calX_4$, we will thus first describe the boundary of all of $2_*\calS$, and then identify which irreducible component(s) of the boundary of $2_*\calS$ are in fact contained in $2_*\calS'$.

\subsubsection*{The boundary of $\calS$}
In this case we are dealing with the singularities of the semi-abelic theta divisor, which come in two different flavors, described by Mumford \cite{mumfordAg}, depending on whether $x\ne0,\infty$ (and the singularity is at a smooth point of the semi-abelic variety) or $x=0$. The semi-abelic theta divisor is singular at some point with $x\ne0,\infty$ --- which is a smooth point of the semi-abelic variety --- if the gradient of the
semi-abelic theta function is zero at such a point, which is equivalent to saying that
\begin{equation}\label{equ:A11}
  \theta(\tau,z)=\theta(\tau,z+b)=0\quad{\rm and}\quad \operatorname{grad}_z\theta(\tau,z)+x\operatorname{grad}_z\theta(\tau,z+b)=0,
\end{equation}
where the last condition means simply that the two gradient vectors are proportional.

The semi-abelic theta divisor is always singular at points with $x=0$, as these are singularities of the semi-abelic variety itself. However, as shown by Mumford \cite{mumfordAg}, a point on the semi-abelic theta divisor with $x=0$ is the limit of singular points of theta divisors on abelian varieties if and only if such a point lies on the singular locus of the theta divisor of the base abelian variety, i.e., if $z\in\Sing\Theta_\tau$.
We will label these two cases as ${\bf A}1$ and ${\bf A}2$. Moreover, in each of these cases we could a priori have different irreducible components of the locus corresponding to the different possibilities of the dimension of $\Sing\Theta_\tau$. As is well-known, the theta divisor of a genus $3$ curve $C$ is smooth
if and only if $C$ is not hyperelliptic  (see eg.~\cite[p.~250]{acgh}).  The theta divisor of the Jacobian of a hyperelliptic genus $3$ curve has one singular point,
namely its $g^1_2$. For a product of an elliptic curve and an abelian surface, the singular locus of the theta divisor is a curve: if $\tau_3=\tau_1\times\tau_2$, then we have
\begin{equation}\label{thetaproduct}
 \Theta_{\tau_3}=(A_{\tau_1}\times\Theta_{\tau_2})\cup (\Theta_{\tau_1}\times A_{\tau_2})
 \quad {\rm and}\quad \Sing\Theta_{\tau_3}=\Theta_{\tau_1}\times\Theta_{\tau_2}.
\end{equation}
Also in the case of a product of $3$ elliptic curves, the singular locus of the theta divisor has dimension $1$, however it is no longer irreducible.
We will label these three possibilities by 1,2,3 (corresponding to $\calJ_3,\calH_3,\calA_1\times\calA_2$) written as the second digit in our numbering.
We note that a priori they can lead to loci in $\overline{IJ}\cap\beta_{2}^0$ of different dimensions (and in fact they do), while as we know that $\overline{IJ}\cap\beta_{2}^0$ is equidimensional of dimension 8, we are only interested in the $8$-dimensional components. We finally recall that the boundary $\overline{\calJ}_4\cap\beta_1^0$ is the locus $(\tau_3,p-q)$ where $p,q$ lie on the curve of which $\tau_3$ is the Jacobian matrix. We thus have the following cases.

\subsubsection*{Case ${\bf A}11$} Recall this means we are looking for the singularities for $x\ne 0$, and $\tau_3\in\calJ_3$ is not hyperelliptic. Then we must have $b=p-q$, $\theta(\tau,z)=\theta(\tau,z+b)=0$ and $\operatorname{grad}_z\theta(\tau,z)+x \operatorname{grad}_z\theta(\tau,z+b)=0$. This gives 12 complex variables (6 for $\tau_3$, 1 each for $p,q$, 3 for $z$, 1 for $x$) and 5 conditions. Nevertheless, we claim that this locus is $8$-dimensional. In fact it has two irreducible components. The first is given by $b=0$, thus
we obtain singularities for the theta divisor for all $z\in \Theta_C \times \{0\}$ and $x=-1$. Remembering that we have to multiply by $2$ we obtain
the following locus:
\begin{flalign*}
 {\bf A}11a:=\lbrace \tau_3=&J(C), b=0 ,z\in 2_*\Theta_C, x=1\rbrace.&
\end{flalign*}

There is, however, also a component with $b\neq 0$.
To describe this we will first work in the Jacobian $J^2(C)$ of degree $2$.
Recall that the canonical model of a non-hyperelliptic curve $C$ of genus $3$
is a plane quartic and the theta divisor $\Theta_C=S^2(C)$ is just the second symmetric product of $C$. Given a pair ${p,q}$ of points on $C$ the geometric form of the
Riemann--Roch theorem tells us that the image of the Gauss map of  ${p,q}$ in the dual projective plane is the line $L$ spanned by $p$ and $q$. Now let
$C \cap L= \{p,q,p',q'\}$ (which is to say that we have $K_C=p+q+p'+q'$). Then any pair of points in this intersection has the same image under the Gauss map (which also shows that the Gauss map is $6:1$). We can then
take $z:=p+q$ and $z+b:=p'+q'=K_C-p-q$.
These points are identified under the Gauss map, i.e., the gradients $ \operatorname{grad}_z\theta(\tau,z)$ and
$\operatorname{grad}_z\theta(\tau,z+b)$ are proportional.
To translate this into  degree $0$ we replace $z$ and $b$ by $z  + \kappa$ and $b+\kappa$ respectively, where $\kappa$ is a theta characteristic (whose choice does not matter
since we multiply by $2$).
This gives us precisely the parametrization of the locus ${\bf A}11b$ in the theorem.

\subsubsection*{Case ${\bf A}12$.} We claim that this does not give a component in $\overline{IJ}\cap\beta_{2}^0$.
In this case $\tau_3\in \calH_3$, which means that the curve $C$ varies in a $5$-dimensional family.
We claim that we do not get an $8$-dimensional family because the Equations \eqref{equ:A11} define a variety of codimension at least two for every smooth hyperelliptic curve $C$.
Let us first assume that $z$ and $z+b$ are not in the singular locus of the (degree $2$) theta divisor $\Theta_C$. On the regular part of the theta divisor the Gauss map is a finite map of degree $4$.
Hence the condition that the two gradients are proportional is codimension $2$. If $z$ is the $g_2^1$ and $b\neq 0$, then $x=0$ and again we get something of codimension $2$
(moreover we would be in case ${\bf A}2$ from the start). The case that $z+b$ is the $g_2^1$ and $b\neq 0$ is analogous (formally leading to $x=\infty$). Finally we could have $b=0$ and $z$ being the
$g^1_2$, in which case $x$ can be arbitrary. But this is only $1$-dimensional.

\subsubsection*{Case ${\bf A}13$.} In this case we have $\tau_3=t\times\tau_2\in\calA_1\times\calA_2$, and the theta divisor becomes reducible and singular as described by Equation \eqref{thetaproduct}. If we have $b\in E_t\times\Theta_{\tau_2}$, then the conditions $z\in\Theta_{\tau_3}$ and $z+b\in\Theta_{\tau_3}$ are independent. In this case for the dimension count we would have $4$ for $\tau_3$, $2$ for $b$, $3$ for $z$, $1$ for $x$ for a total of $10$, but then would have two conditions for the points to lie on the theta divisor, while the condition for the gradients is not satisfied automatically, and thus cuts the dimension of the locus down to at most 7.

However, if $b=(0,b_2)$, then for the case when $z\in \Theta_t\times A_{\tau_2}\subset \Theta_{\tau_3}$, we would automatically have $b+z\in\Theta_{\tau_3}$. Noticing that $\Theta_t=\lbrace (1+t)/2\rbrace$,
we compute for the gradients:
$$
  \partial_{z_1}(\theta(z)+x\theta(z+b))=\partial_{z_1}\theta(t,(1+t)/2)(\theta(\tau_2,z_2)+x\theta(\tau_2,z_2+b_2))
$$
and vanishing of this gives one condition determining $x$ uniquely, while
$$
\partial_{z_2}(\theta(z)+x\theta(z+b))=\theta(t,(1+t)/2)\cdot(\ldots)=0
$$
is automatically satisfied. Remembering to multiply the $z$ by 2, we thus get the 8-dimensional locus
\begin{flalign*}
  {\bf A}13:= \lbrace\tau_3= &t \times\tau_2  \in \calA_1 \times \calA_2, b=(0,b_2), z=(0,2z_2), x=y^2\mid& \\   &\theta(\tau_2,z_2)+y\theta(\tau_2,z_2+b_2)\rbrace.&
\end{flalign*}

We now proceed to the case ${\bf A}2$, that is when $x=0$, and we must have $z\in 2_*\Sing\Theta_{\tau_3}$.

\subsubsection*{Case ${\bf A}21$} This would be the case when $\tau_3$ is a non-hyperelliptic Jacobian of genus 3, but then its theta divisor is smooth, hence the case ${\bf A}21$ is impossible.

\subsubsection*{Case ${\bf A}22$} Here we have $\tau_3\in \calH_3$, so $z\in 2_*\Sing\Theta_{\tau_3}$ is simply the one point $0$, and the dimension count gives $5$ for $\tau_3$, and $2$ for $b$, which is too small to yield a component of $\overline{IJ}\cap\beta_{2}^0$.

\subsubsection*{Case ${\bf A}23$} Here we have $\tau_3=t\times\tau_2$ (so 4 parameters), $b$ gives two extra parameters, $x=0$, and $\dim\Sing\Theta_{\tau_3}=1$, so the total dimension is at most 7, and we do not get an 8-dimensional component of $(\overline{\bf A}\cup\overline{\bf B})\cap\beta_{2}^0$.

\smallskip
We have thus finally determined the boundary of $2_*\calS$ within the partial compactification of $\calX_4$ to consist of the three components ${\bf A}11a, {\bf A}11b, {\bf A}13$ described above, and will now need to argue that only ${\bf A}11b$ is in fact contained in the closure of the locus ${\bf A}$.
By the properties of the wonderful blow-up any component of the intersection $\overline{\bf A}\cap\beta_2^0\subset{\mathcal A}_5^{\operatorname{tor}}$, must correspond to some further degeneration of $A_1$-cubics, i.e., must arise as the image of cubics with some collection of singularities that is a degeneration of $A_1$, or from the chordal cubic. We can easily see by inspecting the abelian parts and dimensions, that none of the components ${\bf A}11a, {\bf A}11b, {\bf A}13$ are contained in the image of $\tDA[4]$ (which maps to an 8-dimensional locus not contained in the boundary), $\tDA[5]$ (which maps to an 8-dimensional locus contained in $\beta_1$ but not in $\beta_2$), or $\tDD[4]$ (for which the compact part is a product of three elliptic curves, and this is not the case for any of the components ${\bf A}xx$).

Moreover, by Lemma \ref{lem:A2}  none of the ${\bf A}xx$ components can come from $\tDA[1] \cap \tDA[2]$.
As our discussion of case ${\bf B}$ will show, see Remark \ref{remidentify}, no ${\bf A}xx$ component comes from $\tDA[1] \cap \tDA[3]$. Finally, the chordal cubic locus maps to the locus of Jacobians of hyperelliptic curves. The locus of hyperelliptic genus 5 curves is 9-dimensional, and as one easily sees (eg.~from its identification with the moduli of twelve unordered points on $\PP^1$), its intersection with $\beta_2$ is of codimension two (the curve must have at least two nodes). Thus the intersection of the locus of hyperelliptic Jacobians with $\beta_2$ is 7-dimensional, and cannot contain any of the 8-dimensional loci ${\bf A}xx$.

Hence, finally we see that the only locus in the wonderful compactification that could map to one of the components ${\bf A}xx$ is the self-intersection $\tDA[1] \cap \tDA[1]$. By Remark \ref{rem:components}, the locus of cubic threefolds with two $A_1$-singularities is irreducible, and hence exactly one of the components
${\bf A}11a$, ${\bf A}11b$ or ${\bf A}13$ must occur as the image of the irreducible locus of $2A_1$ cubics.

We will now show that the locus ${\bf A}11b$ is not contained in the boundary of either $2_*\calS_{\rm null}$ or $2_*\calS_{\rm dec}$. Since it is nonetheless contained in the boundary of
$2_*\calS$, it must thus be contained in the boundary of ${\bf A}=2_*\calS'$, and by the above, must thus be the only such component. This finishes the proof in this case.

Indeed, to see that ${\bf A}11b$ is not contained in the boundary of $2_*\calS_{\rm null}$, we note that the limit of two-torsion points on smooth ppav under rank $2$ semi-abelic degenerations has been studied in detail in \cite{grhu1,grhu2} --- in particular the $x$ coordinate of such points must be equal to $\pm 1$. Since the coordinate $x$ of a generic point of ${\bf A}11b$ is different from $\pm 1$, it follows that ${\bf A}11b$ is not contained in the closure of $\calS_{\rm null}.$
It remains to show that  ${\bf A}11b$ is not contained in the boundary of $\calS_{\rm dec}$.
To obtain the boundary of $\calS_{\rm dec}$, defined by \eqref{Sdec}, one has to degenerate either the $t\in\calA_1$ or the $\tau\in\calA_3$ factor of the compact part.
If $\tau$ is degenerated, then the base would be $t\times\tau_2$ for some $\tau_2\in\calA_2$, and thus it could not be a generic point of ${\bf A}11b$, which is over an indecomposable base. On the other hand, if the $t$ factor were degenerated to $i\infty$, then we would have $b=0$, while $x$, being the limit of twice the point $\Theta_t=(1+t)/2$, would become zero.
This shows that ${\bf A}11b$ is not contained in the boundary of $\calS_{\rm dec}$. At the same time we get precisely the description of component ${\bf A}11a$ (which is thus contained in the boundary of $2_*\calS_{\rm dec}$).
Note that this argument does not suffice to show that ${\bf A}11a$ cannot {\em also} be contained in the boundary of $2_*\calS'$, for this we have to use  the fact that
the locus of cubics with two $A_1$-singularities is irreducible.

We shall now proceed to describe the boundary of locus ${\bf B}$.

\subsection*{Case ${\bf B}$}
In this case the geometry is much simpler, but one has to be careful with the combinatorics.
First of all, the boundary of $\calA_1\times\calH_3$ in $\beta_1^0\subset {\calA_4^{\operatorname{tor}{}}}$, where  ${\calA_4^{\operatorname{tor}{}}}$ denotes any toroidal compactification of $\ab[4]$,
has two components, corresponding to which factor degenerates to a semi-abelic variety. We will denote these possibilities by 1 and 2, corresponding to whether $t\in\calA_1$ or $\tau_3\in \calH_3$ degenerates.

\subsubsection*{Case ${\bf B}1$} This case turns out to be very easy:  the abelian part of the semi-abelic variety is simply given by $\tau_3\in \calH_3$. In this case there cannot be any extension data (we degenerated the elliptic curve, the corresponding abelian part is trivial), i.e., we have $b=0$. The dimension count gives $5$ for $\tau_3$, plus $3$ for the point $z$, and we thus get the component ${\bf B}1$ in the theorem.

\subsubsection*{Case ${\bf B}2$} This is the situation where $\tau_3$ degenerates. Since the locus ${\bf B}$ is the product of the universal family of elliptic curves $\calX_1$, and the universal theta divisor over $\calH_3$, we need to describe the degeneration of the universal theta divisor over $\calH_3$, and then take arbitrary $t\in\calA_1$,  $z=(z_1,z_2)$ with $z_1\in E_t$ arbitrary, and  $b=(0,b_2)$. Note furthermore that $\calH_3$ is the theta-null divisor in $\calA_3$. The degeneration  of the singularities of the universal theta divisor
over the theta-null divisor in $\calA_3$
is well-known, see \cite[Equation (21)]{grhu1}. Noticing that we have $\theta_{\rm null}^{(2)}=\calA_1\times\calA_1$, we have in any toroidal  compactification ${\calA_3^{\operatorname{tor}{}}}$ (which in fact all coincide in genus $3$)
that
$$
 \overline{\theta_{\rm null}^{(3)}}\cap\beta_1^0=2_*\Theta_2\cup\pi^{-1}(\calA_1\times\calA_1)
$$
where $\Theta_2\subset\calX_2$ is the universal theta divisor, and $\pi:\calX_2\to\calA_2$.
We then consider separately the two cases corresponding to the two components on the right.

\subsubsection*{Case ${\bf B}21$} This is the case when $(\tau_2,b_2)\in\calX_2$ lies on $2_*$ applied to the universal (symmetric) theta divisor, while $(x,z_2)$ lies on the corresponding semi-abelic theta divisor. These two equations yield precisely the component ${\bf B}21$ of the theorem, and one easily checks that it is $8$-dimensional (1 parameter for $t$, $3$ for $\tau_2$, $1$ for $b_2$, $1$ for $z_1$, $2$ for $z_2$, and $1$ for $x$, for a total of 9, minus one equation for $\tau_2,z_2,x$, so we get $8$).

\subsubsection*{Case ${\bf B}22$}
In this case we must have $\tau_2\in\calA_1\times\calA_1$, while $b_2$ can be arbitrary.
Note that this also means that the theta function becomes a product of two theta functions of genus $1$. We thus get the locus ${\bf B}22$ from the statement of Theorem  \ref{theo:degenerationscod2}, and check that its dimension is equal to 8.
\end{proof}

\begin{rem}\label{remidentify}
The way the proof proceeds, and the way we have labeled the loci, is such that the locus ${\bf A}11b$ is contained in the locus ${\bf A}$, which is the image of the divisor $\tDA[1]$, while the loci ${\bf B}yy$ are contained in the locus ${\bf B}$, which is the image of the divisor $\tDA[3]$. We note that as ${\bf A}11b$ is not equal to any of ${\bf B}yy$, this means that none of the four loci above are contained in the image of $\tDA[1]\cap\tDA[3]$. This can be stated as saying that no component of $\overline{IJ}\cap\beta_2^0$ is contained in $\overline{\bf A}\cap\overline{\bf B}$, or as saying that no component of $\overline{\bf A}\cap\overline{\bf B}$ is contained in $\beta_2^0$ --- and in fact one can easily see that $\overline{\bf A}\cap\overline{\bf B}\cap\beta_1^0$ is non-empty and thus 8-dimensional as expected.

As discussed in the course of the proof, the images of the boundary divisors $\tDA[4]$ and $\tDA[5]$  are 8-dimensional, and not contained in $\beta_2$. Lemma \ref{lem:A2} shows that the image of the (8-dimensional) intersection $\tDA[1]\cap\tDA[2]$ is not contained in $\beta_2$, either. As we have already seen, ${\bf A}11b$ is thus the image of the self-intersection of $\tDA[1]$ (which is the irreducible locus of cubics with $2A_1$-singularities), while the ${\bf B}xx$ loci could be the image of an irreducible component of the self-intersection of $\tDA[3]$, or of $\tDA[3]$ with $\tDD[4]$ (note the image of a generic point of $\tDA[2]\cap\tDA[3]$ has torus rank $1$, and thus $\tDA[2]\cap\tDA[3]$ maps to an 8-dimensional locus not contained in $\beta_2$). Indeed, it seems likely to us that ${\bf B}22$ is equal to the image of $\tDD[4]$, while we expect  that
the locus of singular cubics with $2A_3$-singularities has two connected components which map to
${\bf B}1$ and ${\bf B}21$ respectively.
\end{rem}


\bibliographystyle{amsalpha}

\bibliography{cubics}

\end{document}